\documentclass[11pt,leqno]{amsart}

\usepackage{tikz, centernot}
\usetikzlibrary{shapes,arrows}
\RequirePackage[colorlinks,citecolor=blue,urlcolor=blue]{hyperref}
\usepackage{pkgfile}

\newcommand{\deltx}{\wh \delta}

\renewcommand{\ge}{\geqslant}
\renewcommand{\geq}{\geqslant}
\renewcommand{\le}{\leqslant}
\renewcommand{\leq}{\leqslant}

\def \dist {{\rm dist}}

\newcommand{\Graph}{{\sf G}}

\newcommand{\edge}{{\bm e}}

\def\corA{}

\begin{document}


\title[{Large deviations of the spectral radius}]{{Upper tail of the spectral radius of sparse Erd\H{o}s-R\'{e}nyi graphs}}

\author{Anirban Basak}
\address{Anirban Basak, International Centre for Theoretical Sciences, Tata Institute of Fundamental Research, Bangalore, India}
\email{anirban.basak@icts.res.in}


%
%
%
%

\date{\today}



\maketitle

\begin{abstract}
We consider an Erd\H{o}s-R\'{e}nyi graph $\mathbb{G}(n,p)$ on $n$ vertices with edge probability $p$ such that
\[
\sqrt{\frac{\log n}{\log \log n}} \ll np \le n^{1/2-o(1)}, \label{eq:abs} \tag{$\dagger$}
\]
and derive the upper tail large deviations of $\lambda(\mathbb{G}(n,p))$, the largest eigenvalue of its adjacency matrix. Within this regime we show that, for $p \gg n^{-2/3}$ the $\log$-probability of the upper tail event of $\lambda(\mathbb{G}(n,p))$ equals to that of planting a clique of an appropriate size (upon ignoring smaller order terms), while for $p \ll n^{-2/3}$ the same is given by that of the existence of a high degree vertex. 
We also confirm that in the entire regime \eqref{eq:abs} the large deviation probability is asymptotically approximated by the solution of the mean-field variational problem, 
and further identify the typical structure of $\mathbb{G}(n,p)$ conditioned on the upper tail event of $\lambda(\mathbb{G}(n,p))$ in a certain sub-regime of $p$. 

For $p$ such that $\log(np) \gtrsim \log n$ the large deviations of $\lambda(\mathbb{G}(n,p))$ is deduced from those of the homomorphism counts of the cycle graph of length $2t$, ${\rm Hom}(C_{2t}, \mathbb{G}(n,p))$, for $t \ge 3$ and $p$ such that $n^{1/2-o(1)} \ge np \gg n^{1/t}$. In this latter regime the typical structure of $\G(n,p)$ conditioned on the upper tail of ${\rm Hom}(C_{2t}, \mathbb{G}(n,p))$ is identified and the asymptotic tightness of the mean-field approximation is also established. 
 


\end{abstract}

\section{Introduction and main results}

Study of spectral statistics of large random matrices are of significant interest. There is a host of results regarding the typical behavior of spectral observables of random matrices. However, results on atypical behaviors, such as large deviations, of spectral observables, e.g.~extreme eigenvalues, the \abbr{ESD} (empirical spectral distribution), are few and far between. Using an explicit formula for the joint density of the eigenvalues, the large deviations of the largest eigenvalue and of the \abbr{ESD} of \abbr{GOE} (Gaussian orthogonal ensemble) matrices were obtained in \cite{BDG, BeG}. Beyond such exactly solvable models of random matrices, the first breakthrough was due to Bordenave and Caputo \cite{BC}, where they derived large deviations of the \abbr{ESD} of Wigner matrices \corA{with stretched exponential tails}. There they showed that the large deviation event is created by a relatively few large entries. This was later extended for the largest eigenvalue by Augeri \cite{aug0}. Very recently,  Guionnet and Husson \cite{GH}, using asymptotics of {\em spherical integrals}, derived the large deviations for the largest eigenvalue for Wigner matrices with entries possessing a {\em sharp sub-Gausian tail} (e.g.~Rademacher distribution), and this has been extended to the case of entries with sub-Gaussian tails, see \cite{AGH}. 

The results mentioned thus far are in the context of {\em dense} random matrices. The goal of this article is to derive the large deviations of the largest eigenvalue of a class of {\em sparse} random matrices, namely, the adjacency matrices of Erd\H{o}s-R\'{e}nyi graphs. An Erd\H{o}s-R\'enyi graph $\G(n,p)$ on $n$ vertices is the random graph obtained by joining the edge between each pair of vertices with probability $p=p(n) \in (0,1)$, and independently of every other pair. Here we consider the case $p \to 0$ as $n \to \infty$. 

Before moving further let us mention that the typical behavior of $\lambda(\G(n,p))$, the spectral radius of $\G(n,p)$, is well understood: It follows from \cite{KS} and \cite[Lemma 2.2]{BBG} that, \corA{with $d=np$},
\[
\lambda(\G(n,p)) =  \left\{
\begin{array}{ll}
(1+o(1))\corA{d} & \mbox{ if } \corA{d} \gg \sqrt{\f{\log n}{\log \log n}},\\
(1+o(1)) \max\{\corA{d}, \sqrt{\gL_p}\} & \mbox{ if } \corA{d} \asymp \sqrt{\f{\log n}{\log \log n}},\\
(1+o(1))\sqrt{\gL_p} & \mbox{ if } \corA{d} \ll \sqrt{\f{\log n}{\log \log n}} \mbox{ and } \log n \gg \log (1/\corA{d}),
\end{array}
\right.
\]
almost surely (see Section \ref{sec:notation} for the notational \corA{conventions} used in this paper), where
\[
\gL_p:= \f{\log n}{\log \log n - \log\corA{d}}. 
\] 
The typical behaviors of the extreme eigenvalues have been further extended to the inhomogeneous setting by \cite{BBK1, BBK2}. 

Over the last decade or so there \corA{has} been extensive research on understanding the large deviation phenomena in random graph models. Chatterjee and Varadhan \cite{CV} first successfully considered the problem of upper tail large deviations of triangle counts in dense $\G(n,p)$, i.e.~$p \asymp 1$. To tackle this problem, \cite{CV} introduced a general framework for large deviation principle that uses Szemer\'edi’s regularity lemma \cite{Sz} and the theory of graph limits \cite{BCLSV, LS1, LS2}. They expressed the large deviation rate function as the solution of a {\em mean-field variational problem}, and obtained the structure of the random graph conditioned on the large deviation event in the {\em replica symmetry region}. Later this was extended by Lubetzky and Zhao \cite{LZ0} for {\em nice graph parameters}, e.g.~regular subgraph densities and the spectral radius. They also identified the phase boundary between the replica symmetry and symmetry breaking regimes.

In a breakthrough work Chatterjee and Dembo \cite{chd} proposed a new framework of nonlinear large deviations that can be applied to a host of problems, including, in particular, the upper tail large deviations of subgraph counts in sparse $\G(n,p)$. Similar to the dense setting, here also the large deviation rate function is given by the solution of an appropriate mean-field variational problem, which was solved in \cite{BGLZ}. The results of \cite{chd} have been extended and improved by \cite{aug, cod, eld}. 

The recent work \cite{cod}, in addition to the large deviations of homomorphism densities, derives the large deviation of the upper tail of the spectral radius of an Erd\H{o}s-R\'enyi graph where the rate function is again shown to be the solution of some mean-field variational problem. The solution to this variational problem was identified in \cite{BG}. These two results together imply that
\beq\label{eq:ut1}
\lim_{n \to \infty} -\f{\log \P(\lambda(\G(n,p)) \ge (1+\delta) \corA{d})}{\corA{d^2} \log(1/p)} = \min\left\{ \f{(1+\delta)^2}{2}, \delta(1+\delta) \right\}, \quad \text{ for } \corA{n^{1/2} \ll d \ll n},
\eeq
and $\delta >0$. 
On the other hand, in \cite{BBG}, using a completely different approach, it has been shown that 
\beq\label{eq:ut2}
\lim_{n \to \infty} - \f{\log\P(\lambda(\G(n,p) \ge (1+\delta)\sqrt{\gL_p})}{\log n} = 2 \delta +\delta^2,
\eeq
for $\delta >0$ and \corA{$d$} such that 
\beq\label{eq:ut2-rev}
\log n \gg \log(1/\corA{d}) \quad \text{ and } \quad \corA{d} \ll \sqrt{\f{\log n}{\log \log n}}.
\eeq
%
%

In this paper we will derive large deviations of the spectral radius in the intermediate regime of sparsity, i.e.~for \corA{$d$} such that
\beq\label{eq:p-int-sp}
 \sqrt{\f{\log n}{\log \log n}} \ll \corA{d} \le n^{1/2-o(1)}. 
\eeq
Therefore, {\em the results of this paper together with \cite{BG, BBG, cod} resolve the upper tail large deviations for the spectral radius of eigenvalues in the entire sparse regime, except for a couple of boundary cases}. The lower tail large deviations of $\lambda(\G(n,p))$ for $p \ll 1$ have been settled in \cite{BBG, cod}. 


\subsection*{\corA{Informal summary of the main result}}
\corA{This paper shows that for $d$ such that \eqref{eq:p-int-sp} holds and $\corA{d} \gg n^{1/3}$ the $\log$-probability of the upper tail of $\lambda(\G(n,p))$, upon ignoring smaller order terms, equals to that of {\em planting} a {\em clique} on $\lceil (1+\delta) \corA{d} \rceil$ vertices, while if $\corA{d} \ll n^{1/3}$ and \eqref{eq:p-int-sp} hold then the same equals to that of the existence of a vertex with degree greater than $(1+\delta)^2 n^2 p^2$.} \corA{See Theorem \ref{thm:eig-main} for a precise statement.} 

\corA{Perhaps, at the very first glance, the transition of the large deviation behavior at $d \asymp n^{1/3}$ may seem surprising. However, observe that the $\log$-large deviation probability should be at least as large as the maximum of the $\log$-probabilities of the two events described above (in Theorem \ref{thm:eig-main} we show that this is also essentially the upper bound). 
For $d \ll n^{1/3}$ the probability of the second event dominates that of the first, and hence the transition. 
} 


\corA{It is instructive to note that the existence of a high degree vertex continues to be the {\em primary reason} for an atypical large value of $\lambda(\G(n,p))$  in the regime \eqref{eq:ut2-rev} (see \eqref{eq:ut2}), although the threshold on the degree of a vertex to be called it a high degree degree vertex needs to be changed appropriately in that regime. On the other hand, by \eqref{eq:ut1}, the large deviations probability of the upper tail event for $\lambda(\G(n,p))$, in the regime $n^{1/2} \ll \corA{d} \ll n$, equals, asymptotically, to that of {planting} either a {clique} or a {\em hub} of appropriate sizes. It is also worth noting that this second event ceases to be a viable option if $\corA{d} \ll n^{1/2}$ and therefore one does not encounter it in the regime covered by \eqref{eq:p-int-sp}.}

\subsection*{\corA{Large deviation probability and mean field approximation}}
As a consequence of Theorem \ref{thm:eig-main} and some additional work we establish that in the entire regime \eqref{eq:p-int-sp} the $\log$-probability of the upper tail of $\lambda(\G(n,p))$ is asymptotically approximated by the solution of the mean-field variational problem (see Theorem \ref{cor:mf-lbd}(a)). {\em Prior to this work, the solution to the mean-field variational problem for the spectral radius $\G(n,p)$, in the regime given by \eqref{eq:p-int-sp}, was not known in the literature.} 
It is also worth adding that our approach of deriving the large deviation bound is somewhat reverse in nature compared to the other works in the area, e.g.~\cite{aug, chd, cod, eld} where the $\log$-probability is shown to be asymptotically equal to the solution of the mean-field variational problem, and the variational problem is solved separately in \cite{BG, LZ}.

\subsection*{\corA{Large deviations of homomorphism counts in cycle graphs}} Finally, let us add that Theorem \ref{thm:hom-main} of this article together with \cite{BB} show that in a certain regime of sparsity {\em the large deviation rate functions for the upper tail of subgraph counts and homomorphism counts for $C_{2t}$ differ but the large deviation speeds of these two events remain the same}. 
This difference in the rate function is probably due to lack of low complexity of the gradient of the homomorphism count function in the above mentioned regime of $\corA{d}$. See Section \ref{sec:mf} for further discussions. 

The upper tail large deviations of homomorphism counts of $r$-regular graphs in the regime $p \ll n^{-1/r}$ were not known previously. We believe that, with some additional efforts, Theorem \ref{thm:hom-main} extended to cover the case of all $r$-regular graphs in the above mentioned regime.

\subsection{Main results}\label{sec:res}
The following is the main result on the upper tail large deviations of $\lambda(\G(n,p))$. 
\begin{thm}\label{thm:eig-main}
Let $\G_n \stackrel{d}{=}\G(n,p)$. Fix $\delta >0$. 
\corA{Recall $d=np$. Let $d$ be such that} 
\beq\label{eq:p-ass-sparse-new}
\corA{d \gg \sqrt{\f{\log n}{\log \log n}} \qquad \text{ and } \qquad \lim_{n \to \infty} \f{\log d}{\log n} = \al \in [0, 1/2)}. 
\eeq 
\begin{enumerate}
\item[(a)] \corA{If $\al \in [0,1/3)$ then}  
\beq\label{eq:ld-lambda-1}
\lim_{n \to \infty} - \f{\log \P\left(\lambda(\G_n) \ge (1+\delta) \corA{d}\right)}{\corA{d^2} \log\corA{d}} = (1+\delta)^2. 
\eeq
\item[(b)] \corA{If $\al \in [1/3, 1/2)$ then}  
\beq\label{eq:ld-lambda-2}
\lim_{n \to \infty} - \f{\log \P\left(\lambda(\G_n) \ge (1+\delta) \corA{d}\right)}{\corA{d^2} \log(1/p)} = 
\f{(1+\delta)^2}{2}.
\eeq
\end{enumerate}
\end{thm}

\begin{rmk}
\corA{The reader may check that the negative of the $\log$-probability of having a vertex of degree $u^2d^2$ in $\G(n,p)$, for any $u >0$, is approximately $u^2 d^2 \log d$, while that for having a clique with $ud$ many vertices is roughly $\f12 u^2d^2 \log (1/p)$. Thus the preceding theorem shows that for $\al <1/3$ the upper large deviation event is primarily due to the presence of a high degree vertex while in the other regime that is due to the presence of a large clique.} 

\corA{Let us further add that the proofs for the cases $\alpha=0$ and $\alpha >0$ are different. In Section \ref{sec:pf-thm-eiga} we will prove Theorem \ref{thm:eig-main} for $d$ such that} 
\beq\label{eq:p-ass-sparse}
\corA{d \gg \sqrt{\f{\log n}{\log \log n}} \quad \text{ and } \quad \log\corA{d} \ll \log(1/p)}.
\eeq
\corA{The reader can note that the condition $\log d \ll \log(1/p)$ is equivalent to the condition $\al=0$}. \corA{For the other case, i.e.~when $\al >0$, we will prove a slightly wider regime than described in Theorem \ref{thm:eig-main}. Namely, we will show in Section \ref{sec:pf-thm} that} 
\beq\label{eq:eig-main1}
-\log \P\left(\lambda(\G_n) \ge (1+\delta) \corA{d}\right) = (1+o(1)) \min \left\{\f12 (1+\delta)^2\corA{d^2} \log(1/p), (1+\delta)^2\corA{d^2} \log\corA{d}\right\},
\eeq
for $\corA{d}$ such that  
\beq\label{eq:p-ass-nsparse11}
\log(\corA{d}) \gtrsim \log n \qquad \text{ and } \qquad \corA{d} \leq n^{1/2} (\log n)^{-\omega(1)}.
\eeq 
\end{rmk}


\vskip5pt

{
\corA{When $\al=0$ Theorem \ref{thm:eig-main} can be strengthened to derive the typical structure of $\G(n,p)$ conditioned on the atypical event that its largest eigenvalue is large}.  In particular, we obtain the following result. For ease in writing, let us introduce the following notation: For $\delta, \xi>0$ and $\G_n \subset K_n$ we set 
\[
{\rm UT}_\lambda(\delta):= \{\lambda(\G_n) \ge (1+\delta) \corA{d}\} \quad \text{ and } \quad {\rm UT}_\Delta(\xi):= \{\Delta(\G_n) \ge \xi \corA{d^2}\}.
\]
\begin{cor}\label{cor:typ-strc}
Let \corA{$d$} satisfy \eqref{eq:p-ass-sparse} and  $\G_n \stackrel{d}{=}\G(n,p)$.  
Then for any $\delta >0$ and $\chi \in (0,1)$, 
\beq\label{eq:typ-strc}
\P\left({\rm UT}_\Delta((1+\delta)^2(1-\chi)) \mid {\rm UT}_\lambda(\delta) \right) \to 1, \text{ as } n \to \infty. 
\eeq
\end{cor}
Observe that Corollary \ref{cor:typ-strc} treats the sparsity regime \eqref{eq:p-ass-sparse} and shows that in that regime the upper tail large deviation event of the largest eigenvalue is (primarily) due to the presence of a very large degree. }
We expect that an appropriate analogue of Corollary \ref{cor:typ-strc} should hold for the entire regime  \eqref{eq:p-int-sp}. See Section \ref{sec:ext-open}.

The proof of Theorem \ref{thm:eig-main} \corA{for $\al >0$} hinges on understanding the upper tail large deviations of the homomorphism counts of even cycles in Erd\H{o}s-R\'{e}nyi graphs. First let us provide necessary definitions. 

\begin{dfn}[Homomorphism counts and labelled copies]
Given graphs $H$ and $\Graph$ we write ${\rm Hom}(H, \Graph)$ to denote the number of homomorphisms of $H$ into $\Graph$. That is, 
\beq\label{eq:hom-dfn}
{\rm Hom}(H, \Graph) := \sum_{\varphi: V(H) \mapsto V(\Graph)} \prod_{\edge= (x,y) \in E(H)} a^\Graph_{\varphi(x), \varphi(y)},
\eeq
where the sum runs over {\em all} maps $\varphi$ from $V(H)$ to $V(\Graph)$ and  $\{a^\Graph_{u,v}\}$ are the entries of the adjacency matrix of $\Graph$. 

There is a closely related notion to the homomorphism counts, known as the number \corA{of} labelled copies of $H$ in $\Graph$, which is defined as follows:
\[
N(H, \Graph):= \sum_{\varphi: V(H) \hookrightarrow V(\Graph)} \prod_{\edge= (x,y) \in E(H)} a^\Graph_{\varphi(x), \varphi(y)}.
\]
Here the sum is over all {\em injective} maps. 
\end{dfn}

Below is the result on the upper tail large deviations of homomorphism counts of $C_{2t}$ in $\G(n,p)$, \corA{where for an integer $s \ge 3$ we write $C_s$ to denote the cycle graph of length $s$}. 

\begin{thm}\label{thm:hom-main}
Fix $t \ge 3$ and let $\corA{d}$ be such that 
\beq\label{eq:p-ass-nsparse2}
\corA{n d^t} \ll \corA{d^{2t}}. 
\eeq
Then, for $\G_n \stackrel{d}{=}\G(n,p)$,
\beq\label{eq:hom-exp}
\E[{\rm Hom}(C_{2t}, \G_n)] = (1+o(1)) \corA{d^{2t}}.
\eeq
\corA{If $\lim_{n \to \infty} \log d / \log n = \al$}
then, for any $\deltx >0$, 
\beq\label{eq:hom-main}
\lim_{n \to \infty}- \f{\log \P\left({\rm Hom}(C_{2t}, \G_n) \ge (1+\deltx)\corA{d^{2t}}\right)}{\corA{d^2} \log(1/p)} = \f12 \deltx^{1/t} \cdot \min\left\{ 2^{1-1/t} \cdot \corA{\f{\alpha}{1-\alpha}}, 1\right\}.
\eeq
\end{thm}

The case $t=2$ is excluded from Theorem \ref{thm:hom-main}, as for $t=2$ there is no $\corA{d}$ satisfying \eqref{eq:p-ass-nsparse2} such that $\corA{d} \ll n^{1/2}$. Let us note that, similar to \eqref{eq:ld-lambda-2}, one also sees a dichotomous large deviations behavior \eqref{eq:hom-main}. Recall that previous works (cf.~\cite{aug, cod, CDP}) considered upper tail large deviations of ${\rm Hom}(C_{2t}, \G(n,p))$ for $\corA{d} \gg n^{1/2}$.   


It can be noted that  in the regime $\log \corA{d} \gtrsim \log n$ a large atypical value of the spectral radius produces the same for ${\rm Hom}(C_{2t}, \G(n,p))$, for large $t$. This observation will be used to prove Theorem \ref{thm:eig-main} \corA{for $\al >0$}. 

Theorems \ref{thm:eig-main} and  \ref{thm:hom-main} will be extended below (cf.~Section \ref{sec:mf})  to show that $\log$-probabilities of upper tails of $\lambda(\G(n,p))$ and ${\rm Hom}(C_{2t}, \G(n,p))$ can be asymptotically approximated by the solutions of the mean-field variational problems.

\begin{rmk}
While proving Theorem \ref{thm:hom-main} we will work with a slightly more general assumption \corA{on $d$}. In particular, we will show that for any $\corA{d}$ 
such that
\beq\label{eq:p-hom-gen}
n \corA{d}^t \ll \corA{d}^{2t} \quad \text{ and } \quad \corA{d} \le n^{1/2} (\log n)^{-\omega(1)},
\eeq
 and any $t \ge 3$ fixed, we have  
 {\allowdisplaybreaks
\begin{equation}\label{eq:hom-main1}
-\log \P({\rm UT}_t(\deltx))= (1+o(1)) \min \left\{\f12 \deltx^{1/t} \corA{d^2} \log(1/p), \left(\f{\deltx}{2}\right)^{1/t}\corA{d^2} \log \corA{d} \right\}. 
\end{equation}
}
where, for any $\delta' >0$ and $\G_n \subset K_n$, 
\[
{\rm UT}_t(\delta'):= \left\{{\rm Hom}(C_{2t}, \G_n) \ge (1+\delta') \corA{d}^{2t} \right\}. 
\]
From this, under the additional assumption \corA{that $\lim_{n \to \infty} \log d / \log n =\alpha$} the large deviations result in \eqref{eq:hom-main} is immediate. 
\end{rmk}

The next result identifies the typical behavior of $\G(n,p)$ conditioned on an atypically large value of ${\rm Hom}(C_{2t}, \G(n,p))$. 

\begin{cor}\label{cor:hom-main}
Consider the same setup as in Theorem \ref{thm:hom-main}. Additionally, assume that 
$2^{1-1/t} \corA{ \alpha/(1-\alpha)} <1$. Then, for any $\chi >0$,
\[
\P\left({\rm UT}_{\Delta}((\deltx/2)^{1/t} (1-\chi))\mid {\rm UT}_t(\deltx)\right) \to 0, \quad \text{ as } n \to \infty.
\]
\end{cor}

\subsection{Connection to the na\"ive mean-field approximation}\label{sec:mf}
For a function $h: \{0,1\}^N \mapsto \R$ and the uniform measure $\mu$ on ${\sf C}_N:= \{0,1\}^N$, the {\em Gibbs variational principle} states that 
\beq\label{eq:mf-def}
Z_h:=\log \int e^h d\mu = \sup_\nu \left\{ \int h d\nu - D_{KL}(\nu \| \mu)\right\},
\eeq
where $D_{KL}(\cdot \| \cdot)$ denotes the {\em Kulback-Leibler divergence} and the supremum in \eqref{eq:mf-def} is taken over all probability measures $\nu$ on ${\sf C}_N$. If the supremum in \eqref{eq:mf-def} is replaced by product measures on ${\sf C}_N$, in statistical mechanics, the approximation is termed as the {\em na\"ive mean-field approximation}. 
Over the past ten years there have been several works, in different settings, attempting to find sufficient conditions on $h(\cdot)$, e.g.~appropriate {\em low-complexity} conditions on the (discrete) gradient of $h(\cdot)$, such that the mean-field approximation is asymptotically tight for the {\em log-partition function} $Z_h$ (see \cite{aug, aus, BM, chd, eld, Y}).

A natural extension is to ask whether the mean-field approximation is asymptotically tight for $\mu_p(f \ge (1+\delta) \E_{\mu_p}[f])$, where $f : [0,1]^N \mapsto \R$ is some `nice' function, $\delta >0$, and $\mu_p$ is the product of $N$ i.i.d.~$\dBer(p)$ measures. A heuristic computation shows that if the mean-field approximation is believed to be hold then one should have that 
\beq\label{eq:ut-mf}
\log \P(f \ge (1+\delta) \E_{\mu_p}[f]) = - (1+o(1)) \Psi_p(f, \delta),
\eeq
where
\beq\label{eq:mf-vp3}
\Psi_p(f, \delta):= \inf \left\{ I_p({\bm \xi}): {\bm \xi} \in [0,1]^N \text{ and } \E_{\mu_{\bm \xi}}[f] \ge (1+\delta) \E_{\mu_p}[f]\right\}.
\eeq
\[
I_p({\bm \xi}):= \sum_{\upalpha=1}^N I_p(\xi_i), \quad {\bm \xi}:= (\xi_1, \xi_2, \ldots, \xi_N), \quad \text{and} \quad I_p(x):= x \log \f{x}{p} + (1-x) \log \f{1-x}{1-p} \text{ for } x \in [0,1],
\]
with the convention $0 \log 0 = 0$ and  the probability measure $\mu_{\bm \xi}:=\otimes_{\upalpha=1}^N \dBer(\xi_\upalpha)$. Functions that are of interests and fit into this framework include $N({\sf H}, \G_n)$, ${\rm Hom}({\sf H}, \G_n)$, and $\lambda(\G_n)$ for ${\sf H} \subset K_n$ \corA{and $\G_n \stackrel{d}{=}\G(n,p)$} (set $N = \binom{n}{2}$ and identify $\binom{n}{2}$ possible edges of $\G_n$ to $[N]$). As already mentioned above, the mean-field approximation is shown to be asymptotically tight for large deviations of homomorphism counts and of the largest eigenvalue in sparse $\G(n,p)$ for various ranges of sparsity (cf.~ \cite{aug, chd, cod, CDP, eld}). 
One may enquire if the same phenomenon continues to hold under the setting of this paper. The following result confirms that.  

 \begin{thm}\label{cor:mf-lbd}
Fix $\delta, \deltx >0$,  $\chi \in (0,1)$, and $t \ge 3$. 

\begin{enumerate}
\item[(a)] Let $\corA{d}$ be such that
\[
\sqrt{\f{\log n}{\log \log n}} \ll \corA{d} \le n^{1/2} (\log n)^{-\omega(1)}.
\]
Then, for all large $n$,
\beq\label{cor:mf-lbd-eq1}
 \Psi_p(\lambda(\cdot), \delta(1-\chi)) \le -\log \P({\rm UT}_\lambda(\delta)) \le (1+\chi) \Psi_p(\lambda(\cdot), \delta(1+\chi)).
\eeq
\item[(b)] Let $\corA{d}$ satisfy \eqref{eq:p-hom-gen}. Then, for all large $n$,
\beq\label{cor:mf-lbd-eq2}
\Psi_p({\rm Hom}(C_{2t}, \cdot), \deltx(1- \chi)) \le -\log \P({\rm UT}_t (\deltx)) \le (1+\chi) \Psi_p({\rm Hom}(C_{2t}, \cdot), \deltx(1+\chi)).
\eeq
\end{enumerate}
\end{thm}

The reader may guess that the lower bounds in \eqref{cor:mf-lbd-eq1} and \eqref{cor:mf-lbd-eq2} are immediate from Theorems \ref{thm:eig-main} and \ref{thm:hom-main}. To prove the upper bound one needs to show that certain `error term' is small, and for that one needs bounds on $\Var_{\mu_{\bm \xi}}(f(\G_n))$ for $f(\cdot)= \lambda(\cdot)$ and ${\rm Hom}(C_{2t}, \cdot)$ (see Lemma \ref{lem:mf-var}). This error term arises as an indicator function needs to be approximated by an exponential function. 

\subsection*{\corA{A different formulation of the mean-field variational problem}} The variational problem \eqref{eq:mf-vp3} is slightly different than the one considered in \cite{aug, chd, CV, cod, CDP}. There they show that the $\log$-probability is approximated by 
\beq\label{eq:mf-vp1}
\Phi_p(f, \delta):=\inf \left\{ I_p({\bm \xi}): {\bm \xi} \in [0,1]^N \text{ and } f({\bm \xi}) \ge (1+\delta) \E_{\mu_p}[f]\right\},
\eeq
for various ranges of sparsity. The variational problems \eqref{eq:mf-vp3} and \eqref{eq:mf-vp1} coincide with each other when $|f({\bm \xi}) - \E_{\mu_{\bm \xi}}[f]| \ll \E_{\mu_p}[f]$. Examples of such $f(\cdot)$ include functions that are linear in each variable such as $f(\cdot)= N({\sf H}, \cdot)$. Another well known criteria is the `low-complexity' of $\nabla f(\cdot)$. However, in the absence of suitable low-complexity criteria, \eqref{eq:mf-vp3} and \eqref{eq:mf-vp1} may have a completely different behavior, and \eqref{eq:mf-vp1} may not even represent the correct large deviation behavior. To see this, we observe that from the proof of \cite[Theorem 1.2]{BG} it follows that for $f(\cdot)= \lambda(\cdot)$ and $p$ such that $\corA{d} \gg 1$ the first term in the \abbr{RHS} of \eqref{eq:eig-main1} equals $(1+o(1))\Phi_p(f, \delta) \asymp \corA{d^2} \log(1/p)$. If $\log \corA{d} <  \log n/3$ then the second term in \eqref{eq:eig-main1} is strictly smaller than the first term. Therefore, by Theorems \ref{thm:eig-main} and \ref{cor:mf-lbd}(a), in that regime $\Psi_p(f, \delta)$ is strictly smaller than $\Phi_p(f, \delta)$, and moreover in the regime $\log \corA{d} \ll \log n$ they are of different orders of magnitude.   

%





\subsection*{\corA{Emergence of non-planted optimizers}} \corA{It has been shown in \cite{BG, BGLZ, LZ} that any minimizer ${\bm \xi}_\star$ of $\Phi_p(f, \delta)$, for $f(\cdot)=\lambda(\cdot)$ and ${\rm Hom}({\sf H}, \cdot)$ satisfy ${\bm \xi}_\star \in \{p, 1\}^N$. Therefore, such a ${\bm \xi}_\star$ is also a minimizer of  the variational problem 
\beq\label{eq:mf-vp2}
\wh \Phi_p(f, \delta):= \inf\left\{e(\Graph) \log(1/p): \Graph \subset K_n \text{ and } \E_\Graph[f] \ge (1+\delta) \E_{\mu_p}[f] \right\},
\eeq
where $\E_{\Graph}[f] = \E_{\mu_p}[f({\bm \zeta}) | \zeta_\upalpha =1; \upalpha \in E(\Graph)]$. To see this for any ${\bm \xi} \in \{p, 1\}^N$ one associates a graph $\Graph({\bm \xi}) \subset K_n$ by letting $E(\Graph({\bm \xi}))=\{\upalpha \in [N]: \xi_\upalpha =1\}$.}



\corA{For any ${\bm \xi} \in [0,1]^N$ the probability measure $\mu_{\bm \xi}$ can be naturally associated to the inhomogeneous Erd\H{o}s-R\'enyi graph $\G(n, {\bm \xi})$ with edge probabilities given by $\xi_\upalpha$ for $\upalpha \in [N]$. Thus, for an optimizer ${\bm \xi}_\star \in \{p,1\}^N$ the edges in $\Graph({\bm \xi}_\star) \subset \G(n, {\bm \xi}_\star)$ are present with probability one. Therefore, in the literature the edge set $E(\Graph({\bm \xi}_\star))$ is commonly termed as {\em planted} and by an abuse of terminology we term the associated optimizer ${\bm \xi}_\star$ a {\em planted optimizer}. 
Such is the case for the upper tail large deviations of $\lambda(\G(n,p))$ in the regime $n^{1/2} \ll d \ll n$, where the $\log$-large deviations probability is asymptotically equal to the variational problem $\Phi_p(\lambda(\cdot), \cdot)$ whose optimizers are planted and the corresponding planted structures induced by $\Graph({\bm \xi}_\star)$ are either a clique or a hub of appropriate sizes (see \eqref{eq:ut1}).} 

\corA{In contrast, the proof of Theorem \ref{cor:mf-lbd} will show that, in certain sub regimes of $d$, an optimizer for the variational problem $\Psi_p(f(\cdot), \cdot)$, both for $f(\cdot)=\lambda(\cdot)$ and $f(\cdot)={\rm Hom}(C_{2t}, \cdot)$,  is $\wh{\bm \xi}$ where $\wh\xi_{i,j}= p + \uptau n p^2 {\bf 1}_{i =1}$, for $i \ne j \in [n]$, for some appropriate choices of $\uptau >0$. Clearly the optimizer $\wh{\bm \xi}$ is a non-planted optimizer. Therefore, unlike in the regime, $n^{1/2} \ll d \ll n$, we see an emergence of non-planted optimizers in the regime considered in this article.} 

\subsection{Extensions and open problems}\label{sec:ext-open}



Note that Theorem \ref{thm:eig-main} leaves out two boundary cases. If $\corA{d} \asymp \sqrt{\log n/\log \log n}$ it is natural to predict the large deviation speed to be $\log n$ and the rate function should be some combination of the rate function in Theorem \ref{thm:eig-main} \corA{for $\al=0$} and \cite[Theorem 1.1]{BBG}. We expect Theorem \ref{thm:eig-main} to extend for $\corA{d}$ such that $n^{-1/2-o(1)} \le \corA{d} \ll n^{1/2}$ as well. 
It seems that one can combine the ideas of Lemma \ref{lem:ps2c} and those in \cite[Section 7]{hms} to treat this regime. However, for the sake of brevity and clarity of the presentation we have not attempted this here. Next, in the regime $p \asymp n^{-1/2}$ one is expected to encounter some integrality issue, as was seen in \cite[Theorem 1.7]{hms}. 

Corollary \ref{cor:typ-strc} provides a description of the typical behavior of $\G(n,p)$ conditioned on ${\rm UT}_\lambda(\delta)$ in the regime \eqref{eq:p-ass-sparse}. The same behavior should extend for $\corA{d} \ll n^{1/3}$. See Remark \ref{rmk:hom-to-eig} for the challenge on proving the same. It is expected that in the regime $n^{1/3} \ll \corA{d} \ll n^{1/2}$ the graph $\G(n,p)$, conditioned on ${\rm UT}_\lambda(\delta)$, should typically contain an almost-clique of an appropriate size.  
Same is believed to hold for ${\rm Hom}(C_{2t}, \G(n,p))$ for $\corA{d}$ in the regime not covered in Corollary \ref{cor:hom-main}.
This would require analyzing the near-minimizers of \eqref{eq:mf-vp1} (or \eqref{eq:mf-vp2}). See the discussion in \cite[Section 10]{hms} in this regard. 

Another natural question would be to extend Theorem \ref{thm:hom-main} for all regular graphs. It would be interesting to check if there is a sparsity regime where the large deviation rate functions for the upper tails of ${\rm Hom}(H, \G(n,p))$ and $N(H, \G(n,p))$ are different, but the speed is the same. We expect this behavior to depend on whether $H$ is bipartite or not. For example, preliminary computations show that $H=C_{2t+1}$ upper tails of ${\rm Hom}(H, \G(n,p))$ and $N(H, \G(n,p))$ share the same large deviation speed and rate function for all $\corA{d}$ such that $\corA{d} \ge (\log n)^C$, for some $C< \infty$. It remains to be investigated whether the same should hold in the entire `localized regime'. 

A related problem is the large deviations of the largest eigenvalue of diluted Wigner matrices, i.e.~matrices of the form $W_n \circ B_n$, where $W_n$ is a Wigner matrix and $B_n$ is a symmetric matrix with i.i.d.~$\dBer(p)$ entries on the diagonal and above the diagonal positions, and $p \to 0$ with $n \to \infty$. The case of $p \asymp 1/n$ and the entries of $W_n$ are standard Gaussian was recently dealt in \cite{GN}.  
This problem is closely related to the problem of studying large deviations of the second largest eigenvalue of the adjacency matrix $\G(n,p)$. 
\corA{In this direction a very recent progress has been made in \cite{AB23}}. 

%


\subsection{Outline of the proofs of Theorem \ref{thm:eig-main} and \ref{thm:hom-main}, and Corollary \ref{cor:hom-main}}\label{sec:outline} 
{\corA{The proof of Theorem \ref{thm:eig-main}(a) splits into two parts: $\alpha =0$ and $\alpha >0$. To prove Theorem \ref{thm:eig-main}(a) for $\alpha=0$}} we split the random graph $\G(n,p)$ into vertices of high, moderate, and low degrees (see Definition \ref{dfn:graph-decompose}). At a high level, the idea of decomposition may similar to those in \cite{BBG, KS}. However, let us emphasize that the threshold used to define such subsets of vertices, as well as the arguments employed here 
are completely different from those in \cite{BBG, KS}. For example, the key ideas in \cite{BBG} are to show that (i) there are no `large' cycles, and (ii) any vertex is not incident to too many edge disjoint `small' cycles, at the large deviation scale. These do not hold for the entirety of the regime \eqref{eq:p-ass-sparse}. Hence, we need new ideas. 

Bounding the spectral radius of different subgraphs of $\G(n,p)$ require several different approaches. Below we illustrate some of them. In Lemma \ref{lem:M-connected} we show that it is unlikely, at the large deviations scale, to have a connected component $\Graph$ of $\G(n,p)$ such that all its vertices are of moderate degree and $v(\Graph) \gg \corA{d} \log \corA{d}$. Using this, the fact that $\lambda(\Graph) \le \sqrt{2 e(\Graph)}$, and \corA{the} Chernoff bound we then argue that the subgraph of $\G(n,p)$ spanned by moderate degree vertices has a negligible spectral radius. 

To bound the spectral radius of a certain subgraph \corA{$\Graph_0$} of $\G(n,p)$ such that all its vertices are of low degree (see the proof of Theorem \ref{thm:eig-main}(a) \corA{for $\al=0$} for a precise definition of this subgraph) we employ the following `bootstrap' strategy:~using a bound on the number of high degree vertices (see Lemma \ref{lem:V-H-bd}) we first show that for such a graph, say $\wt \Graph$, at the large deviations scale, one must have that $e(\wt \Graph) \le \corA{d}^C$, for some absolute constant $C< \infty$. Then, we argue that if $\lambda(\wt \Graph)$ is non-negligible (i.e.~$\gtrsim \corA{d}$) then, upon excluding an event of negligible probability (at the large deviations scale), one can procure a subgraph $\wt \Graph' \subset \wt \Graph$ such that $e(\wt \Graph') \le \corA{d}^{C-1}$ and $\lambda(\wt \Graph')$ is still non-negligible (see Lemma \ref{lem:G-low-lambda}). We iterate this argument to finally obtain a subgraph $\wt \Graph_0 \subset \wt \Graph$, with minimum degree two, such that $e(\wt \Graph_0) \le \corA{d}^2$ and $\lambda(\wt \Graph_0) \gtrsim \corA{d}$. Applying a standard Binomial tail probability bound we then deduce that the number of excess edges of $\wt \Graph_0$ (i.e.~$e(\wt \Graph_0) - v(\wt \Graph_0)$) should not be large, see Lemma \ref{lem:not-many-edges}. \corA{However, by Lemma \ref{lem:graph-eig}(viii) this upper bound then yields that $\lambda(\wt \Graph_0)/\corA{d}$ must be negligible thereby showing that the spectral radius of $\Graph_0$ must be negligible as well}.

To treat the spectral radius of the bipartite subgraph of $\G(n,p)$ with vertex bipartition consisting of moderate and low degree vertices we use the following idea: we apply estimates on the number of moderate degree vertices in the two-neighborhood of any given vertex (see Lemma \ref{lem:nbh-in-MH}) and a bound on the spectral radius of a bipartite graph 
(see Lemma \ref{lem:graph-eig}(vii)) to deduce that the spectral radius to be non-negligible there must be two sets of vertices of small sizes containing too many edges between them. By a standard bound on Binomial tail probability, and a union bound, this event turns out to be unlikely at the large deviations scale. To treat the rest of the subgraphs of $\G(n,p)$ we employ estimates that are already mentioned above, i.e.~Lemmas \ref{lem:nbh-in-MH}, and \ref{lem:not-many-edges}-\ref{lem:G-low-lambda}. See Section \ref{sec:pf-thm-eiga} for further details. 

 
 Let us now move to describe the main ideas of the proof of Theorem \ref{thm:hom-main}. From this Theorem \ref{thm:eig-main} \corA{for $\al >0$ essentially} follows upon letting $t \to \infty$. To derive Theorem \ref{thm:hom-main} we follow the general strategy developed in \cite{hms, BB}. However, \cite{hms, BB} deal with the upper tail large deviations of $N(C_{2t}, \G(n,p))$ (in fact they cover all regular graphs), while here we are interested in that of ${\rm Hom}(C_{2t}, \G(n,p))$. Since the large deviation results in these two cases differ, in the sparsity regime we consider, the argument require considerable modifications. As in \cite{hms, BB}, in the first step we apply Markov's inequality to deduce that conditioned on ${\rm UT}_t(\deltx)$ \corA{the random graph contains some} $\Graph \subset K_n$ such that $e(\Graph) \lesssim \corA{d}^2 \log(1/p)$ and the conditional expectation of ${\rm Hom}(C_{2t}, \G(n,p))$ given $\Graph \subset \G(n,p)$ is at least $\deltx(1-o(1))\corA{d}^{2t}$, holds with probability $1-o(1)$. The class of such graphs $\Graph$ will be called pre-seed graphs (see Definition \ref{dfn:seed-0}). This step allows us to reduce the configuration space to a smaller one that is `responsible' for the large deviation event. 

The next step is to reduce the configuration space further to the set $\G(n,p) \supset \Graph$ where $\Graph$ runs over the set of all {\em core graphs} (see Definition \ref{dfn:core-graph}). \corA{At a high level this step can be thought of procuring an appropriate net of pre-seed graphs so that one can perform a union bound}. This step was almost trivial in \cite{BB}. However, as $\varphi$ in \eqref{eq:hom-dfn} is allowed to be non-injective some of the edges of $C_{2t}$ may get mapped to the same edge in $\G(n,p)$, and thus the number of `independent' edges at our disposal may decrease. 
 To overcome this challenge we introduce a notion of generating and non-generating edges and vertices (see Definition \ref{dfn:frozen}) and perform a combinatorial analysis. See Section \ref{sec:ps2c}. 

Next, we split the set of core graphs into two further subsets: ~$e(\Graph) \gg \corA{d}^2$ and $e(\Graph) = O(\corA{d}^2)$. Using yet another combinatorial argument the first subset can be shown to be unlikely at the large deviations scale. Applying the pigeonhole  principle and the fact that $x \mapsto x^{1/t} + (1-x)^{1/t}, x \in [0,1], t \ge 2$, is strictly concave we further deduce that there exists a splitting for any core graph $\Graph$ belonging to the second subset (upon excluding a negligible number of edges) to $\Graph'$ and $\Graph''$ such that the number of homomorphisms of $C_{2t}$ in at least one of them is at least $\deltx(1-o(1))\corA{d}^{2t}$. Furthermore, all the vertices of $\Graph'$ are of moderately large degrees, while $\Graph''$ is a bipartite graph with parts $U_1$ and $U_2$ such that $\max_{u \in U_1}\deg_{\Graph''}(u) =O(1)$, and $\min_{u \in U_2}\deg_{\Graph''}(u)\gtrsim \corA{d}^2$. Using appropriate lower bounds on $e(\Graph')$ and $e(\Graph'')$, and a bound on the number of such graphs (see Lemma \ref{lem:core-bd}) we complete the argument.

To prove Corollary \ref{cor:hom-main} we need to strengthen some of the above steps. In fact, we argue that under the set up of Corollary \ref{cor:hom-main}, conditioned on ${\rm UT}_t(\deltx)$, the number of homomorphisms of $C_{2t}$ in $\Graph''$ is $\deltx(1- o(1))\corA{d}^{2t}$, and $e(\Graph'') - |U_1| \ll \corA{d}^2$, with probability $1-o(1)$. This shows that almost all the excess homomorpshims in $\G(n,p)$ is due to the stars centered at vertices of $U_2$, i.e.~vertices with degree $\gtrsim \corA{d}^2$. \corA{Using this finding together with an estimate on the Binomial tail probability and the {\em strict} convexity of the Binomial rate function we deduce that there must be a single large degree vertex, thereby yielding Corollary \ref{cor:hom-main}.}

\subsection{Notation}\label{sec:notation}
For $x \in \R$ we use the standard notation $\lceil x \rceil:= \min\{ n \in \Z: n \ge x\}$ and $\lfloor x \rfloor:= \max\{ n \in \Z: n \le x\}$. For $a, b \in \R$ we let $a \vee b:= \max\{a,b\}$ and $a \wedge b:= \min\{a,b \}$. For $N \in \N$ we let 
$[N]:=\{1,2,\ldots, N\}$. Given two sequences of positive reals $\{a_n\}$ and $\{b_n\}$ we write $a_n \lesssim b_n$ to denote $\limsup_{n \to \infty} a_n /b_n < \infty$. The notation $a_n \gtrsim b_n$ will be used to denote $\liminf_{n \to \infty} a_n/b_n >0$. We write $a_n \asymp b_n$ if $a_n \lesssim b_n$ and $a_n \gtrsim b_n$. We further write $a_n \ll b_n$, $b_n=\omega(a_n)$, and $b_n \gg a_n$ to denote $\lim_{n \to \infty} a_n/b_n =0$. Let $\{a_n\}$ be a sequence of reals and $\{b_n\}$ be a sequence of positive reals then we write $a_n=O(b_n)$ if $|a_n| \lesssim b_n$, and write $a_n=o(b_n)$ if $|a_n| \ll b_n$. 

For a graph $\Graph$ we write $V(\Graph)$ and $E(\Graph)$ to denote its vertex and edge sets, respectively. We let $v(\Graph):= |V(\Graph)|$ and $e(\Graph):= |E(\Graph)|$, where $|\cdot |$ is the cardinality of a set. Unless otherwise mentioned, the set $V(\Graph)$ will be assumed to consist only non-isolated vertices. The notation $\deg_\Graph(v)$ will be used to denote the degree of a vertex $v$ in the graph $\Graph$. When the choice of the graph is clear from the context, to lighten the notation we will write $\deg(v)$ instead of $\deg_\Graph(v)$. We will use $\Delta(\Graph)$ and $\updelta(\Graph)$ to denote the largest and the smallest degrees in $\Graph$, respectively. 
We will write $\Graph' \subset \Graph$ to denote that $V(\Graph') \subset V(\Graph)$ and $E(\Graph') \subset E(\Graph)$. The notation $K_n$ will be used to denote the complete graph on $n$ vertices. For two graphs $\Graph ,\Graph' \subset K_n$ the graph $\Graph\setminus \Graph'$ will be the graph induced by the edges $E(\Graph) \setminus E(\Graph')$. Moreover for $S \subset V(\Graph)$ we write $\Graph[S]$ to denote the subgraph indued by the edges with both end points belonging to $S$. For two disjoint subsets $U, V \subset V(\Graph)$ we further use $\Graph[U, V]$ to denote the bipartite subgraph induced by edges with one end point belonging to $U$ and the other in $V$. We will use the shorthands $e(S):= e(\Graph[S])$ and $e(U, V) := e(\Graph[U, V])$. For two vertices $u, v \in V(\Graph)$ we write $u \sim v$ to denote that they are neighbors. 

{For a graph $\Graph$ on $\gn$ vertices we let $\lambda_1(\Graph) \geq \lambda_2(\Graph) \ge \cdots \ge \lambda_\gn(\Graph)$ be the eigenvalues of its adjacency matrix arranged in a non-increasing order. To ease up the notation, whenever there is no scope of confusion, we will write $\lambda(\Graph)$ instead of $\lambda_1(\Graph)$. By a slight abuse of terminology we will often refer $\lambda(\Graph)$ as the top/largest eigenvalue of $\Graph$.} 

For two random variables (or vectors or matrices) $Y_1$ and $Y_2$ we write $Y_1 \stackrel{d}{=} Y_2$ to denote that they have the same distribution. For an event $\cA$ the notation $\cA^\complement$ will denote its complement. We will use the notation $\|B\|$ and $\|B\|_{\rm HS}$ to denote the operator norm and the Hilbert-Schmidt norm of a matrix $B$, respectively.

\subsection*{Organization of the paper} 
In Section \ref{sec:pf-thm-eiga} we prove Theorem \ref{thm:eig-main} \corA{for $\al >0$} and Corollary \ref{cor:typ-strc}. Proofs of all auxiliary estimates are in Section \ref{sec:prooflemmas}.  In Section \ref{sec:pf-thm} we \corA{first prove the large deviation lower bound for ${\rm UT}_t$. In Section \ref{sec:seed} we introduce a few} necessary definitions, e.g.~pre-seed graphs, core graphs, etc, and state a few relevant lemmas. Using these results we complete the proofs of Theorems \ref{thm:hom-main} and \ref{thm:eig-main} \corA{for $\al >0$}. Section \ref{sec:ps2c} is devoted to the proof of Lemma \ref{lem:ps2c} which allows us to move from pre-seed graphs to core graphs. In Sections \ref{sec:core-too-many} and \ref{sec:strong-core} we deal with core graphs with edges $\gg \corA{d}^2$ and $O(\corA{d}^2)$, respectively. 
Sections \ref{sec:pf-cor} and \ref{sec:mf-lbd} provide proofs of Corollary \ref{cor:hom-main} and Theorem \ref{cor:mf-lbd}, respectively. In Appendix \ref{sec:graph-eig-pf} we give the proofs of some non-standard bounds on the spectral radius of a graph. Finally, in Appendix \ref{app:loc-hom} we provide some results on core graphs that are used in Section \ref{sec:pf-thm}. 

\subsection*{Acknowledgments} 
The author thanks Fanny Augeri for her input on Section \ref{sec:mf} and Bhaswar Bhattacharya for helpful suggestions. The author also thanks the anonymous referee for numerous helpful comments and suggestions that helped enhancing the clarity of the presentation. 
This research was carried out in part as a member of the Infosys-Chandrasekharan Virtual Center for Random Geometry, supported by a grant from the Infosys Foundation, and was partially supported by a MATRICS Grant (MTR/2019/ 001105) from Science and Engineering Research Board of Govt.~of India, and an Infosys–ICTS Excellence Grant.
The author also acknowledges the support of the Department of Atomic Energy, Govt.~of India, under project no.~RTI4001.

\section{Proofs of Theorem \ref{thm:eig-main}\corA{(\normalfont{a})} \corA{for $\al=0$} and Corollary \ref{cor:typ-strc}}\label{sec:pf-thm-eiga}

First we 
derive the large deviation lower bound. This will be straightforward. To derive that bound we will need estimates on the binomial upper tail probability. We use the following result.

\begin{lem}[{\cite[Lemma 4.7.2]{Ash}}]\label{lem:ash}
Let $X \stackrel{d}{=} \dBin(N,s)$ for some $N \in \N$ and $s \in (0,1)$. Fix any $\uplambda \in (s, 1]$. Then  
\[
\frac{\exp(- N I_s(\uplambda))}{\sqrt{ 8 N \uplambda (1-\uplambda)}} \le \P(X \ge \uplambda N) \le \exp(- N I_s(\uplambda)),
\]
where $I_s(\cdot)$ is the binary relative entropy function. That is,
\beq\label{eq:bin-entropy}
I_s(\uplambda) := \uplambda \log \frac{\uplambda}{s} +  (1-\uplambda) \log \frac{1-\uplambda}{1-s}. 
\eeq
\end{lem}

For any $\gamma >0$ such that $(1+\gamma) s \le 1$ one can also derive from above the following well known Chernoff bound:
\beq\label{eq:chernoff}
\P(\dBin(N,s) \ge (1+\gamma) N s) \le \exp\left(-\f{\gamma^2}{2+\gamma} Ns\right). 
\eeq
The next lemma provides bounds on the top eigenvalue of a graph in terms of various graph parameters. It will be used in the proof of Theorem \ref{thm:eig-main}, for the lower bound and later more extensively for the upper bound.


\begin{lem}\label{lem:graph-eig}
Let $\Graph$ be a graph with ${\bm e}$ edges, maximum degree $\Delta$, and minimum degree $\updelta$. Then the following bounds hold for the top eigenvalue $\lambda(\Graph)$ of the adjacency matrix of $\Graph$.
\begin{enumerate}
\item[(i)]
\(
\sqrt{\Delta} \le \lambda(\Graph) \le \min \{\Delta, \sqrt{2 {\bm e}}\}.
\)

\item[(ii)] If $E(\Graph) = \cup_i E(\Graph_i)$ then $\lambda(\Graph) \le \sum_i \lambda(\Graph_i)$. Additionally, if the collection of graphs $\{\Graph_i\}$ are vertex disjoint then $\lambda(\Graph)= \max_i \lambda (\Graph_i)$.

\item[(iii)] If $\Graph$ is a forest then $\lambda(\Graph) \le 2\sqrt{\Delta -1}$. 

\item[(iv)] If $\Graph$ is a star then $\lambda(\Graph) = \sqrt{\Delta}$. 

\item[(v)] If $\Graph$ is a bipartite graph such that the degrees on both sides of the bipartition are bounded by $\Delta_1$ and $\Delta_2$ respectively, then $\lambda(\Graph) \le \sqrt{\Delta_1 \Delta_2}$.


\item[(vi)] Let $\Graph' \subset \Graph$ then $\lambda(\Graph') \le \lambda(\Graph)$. 

\item[(vii)] Let $\Graph$ be a bipartite graph with vertex bipartition $V_1$ and $V_2$. That is, all its edges are between $V_1$ and $V_2$. Then 
\[
\lambda(\Graph) \le \min_{i\in\{1,2\}} \max_{v \in V_i} \sqrt{\sum_{u: u \sim v} \deg_\Graph(u)}.
\]
\item[(viii)] If $\Graph$ is a graph with $\updelta \ge 2$ then we have
\[
\lambda(\Graph) \le \sqrt{2({\bm e} - {\bm v}) + \Delta +2},
\]
where ${\bm v}$ is the number of vertices in $\Graph$. 
\end{enumerate}
\end{lem}

Most of the bounds in Lemma \ref{lem:graph-eig} are well known. The proofs of the rest are provided in Appendix \ref{sec:graph-eig-pf}.

\begin{proof}[Proof of Theorem \ref{thm:eig-main}(a) \corA{for $\al >0$} (lower bound)] 
By Lemma \ref{lem:graph-eig}(i) we have that
\beq\label{eq:eig-lbd}
\P({\rm UT}_\lambda(\delta)) \ge \P({\rm UT}_\Delta((1+\delta)^2)) \ge \P(\deg_{\G(n,p)}(v) \ge (1+\delta)^2 \corA{d^2}),
\eeq
where $v \in [n]$ is an arbitrarily chosen vertex. Now the desired lower bound follows upon noting that
\beq\label{eq:entropy-large-lambda}
I_s(\uplambda) = \uplambda \log \frac{\uplambda}{s} \cdot (1+o(1)), \quad \text{ for } 0< s \ll \uplambda \le 1,
\eeq
where $I_s(\uplambda)$ is as in \eqref{eq:bin-entropy}. 
\end{proof}

The rest of this section is devoted to the proof of the upper bound. {As outlined in Section \ref{sec:outline}} this will require us to split the graph appropriately and show that the spectral radius of the random subgraph obtained upon removing the one neighborhoods of high degree vertices is small in the large deviations scale. \corA{The following definition provides the necessary splitting of a graph $\Graph$.} 

\begin{dfn}[Decomposition of the graph]\label{dfn:graph-decompose}
Fix $\eta, \vep \in (0,1/4)$. Given any graph $\Graph$ (possibly random) with vertex set $V(\Graph) \subset [n]$ we define the vertices of high, low, and moderate degrees, to be denoted by $V_{H}(\Graph), V_{L}(\Graph)$, and $V_{M}(\Graph)$, respectively, as follows:
\[
V_{H}(\Graph):= \left\{v \in V(\Graph): \deg_\Graph(v) \ge \varpi_n \right\},
\]
where
\beq\label{eq:varpi}
\varpi_n:= \left\{\begin{array}{ll}
\corA{d} \cdot \sqrt{\f{\log n}{\log \log n}}, & \mbox{ if } \sqrt{\f{\log n}{\log \log n}} \ll \corA{d} \le (\log n)^{1-\eta/2},\\
\corA{d}^{1+\eta} & \mbox{ if } (\log n)^{1-\eta/2} \le \corA{d} \mbox{ and } \log \corA{d} \ll \log n,
\end{array}
\right.
\eeq
\[
V_{L}(\Graph):= \left\{v \in V(\Graph): \deg_\Graph(v) \le (1+\delta (1-\vep)) \corA{d} \right\},
\]
and $V_{M}(\Graph):= V(\Graph) \setminus (V_{H} \cup V_{L})$. When the graph is clear from the context we suppress the dependence in $\Graph$ and write $V_{H}, V_{M}$, and $V_{L}$. We then let $V_{{L}_1}$ to be the subset of $V_{L}$ that are connected to $V_{H}$. That is,
\[
V_{{L}_1} := \left\{ v \in V_{L}: \exists \, u \in V_{H} \text{ such that } (u,v) \in E(\Graph)\right\}.
\]
Denote $V_{{L}_2}:= V_{L}\setminus V_{{L}_1}$. Now we describe the necessary decomposition of the graph $\Graph$:
\begin{itemize}
\item We let $\Graph_{H}:=\Graph[V_H]$, $\Graph_{M}:=\Graph[V_M]$, $\Graph_{{L}_1}:=\Graph[V_{L_1}]$, and $\Graph_{{L}_2}:=\Graph[V_{L_2}]$. 

\item Define $\Graph_{{HM}}:=\Graph[V_H, V_M]$. 
Similarly define $\Graph_{{L}_1 {L}_2}$ and $\Graph_{ML}$. 

\item Denote $\wh \Graph_{{L_1L_2}}$ and $\wh \Graph_{L_1}$ be the two cores of $\Graph_{L_1L_2}$ and $\Graph_{L_1}$, respectively. Set ${\sf F}_{{L_1L_2}} := \Graph_{{L_1L_2}}\setminus \wh \Graph_{{L_1L_2}}$ and ${\sf F}_{L_1}:= \Graph_{L_1}\setminus \wh \Graph_{L_1}$. Note that by definition ${F}_{{L_1L_2}}$ and ${\sf F}_{L_1}$ are forests. 

\item Similarly as above we define $\Graph_{{H}{L}_1}$, $\wh \Graph_{{H}{L}_1}$, and ${\sf F}_{{H}{L}_1}$. 

\item We further decompose the forest ${\sf F}_{{H}{L}_1}$ into two subgraphs. We let $\wt {\sf F}_{{H}{L}_1}$ to be the subgraph spanned by the edges that are incident to vertices $u \in V_{L_1}$ that have degree one in ${\sf F}_{HL_1}$. 
Note that by construction $\wt {\sf F}_{HL_1}$ is a vertex disjoint union of stars with center vertices belonging to $V_H$. Finally set $\wh {\sf F}_{HL_1}:= {\sf F}_{HL} \setminus \wt {\sf F}_{HL_1}$. 
\end{itemize}
\end{dfn}


\corA{Our task is to show that 
\[
\max\{\lambda(\Graph_H), \lambda(\Graph_{HM}), \lambda(\Graph_M), \lambda(G_{ML}), \lambda(G_{L_1}), \lambda(G_{L_1 L_2}), \lambda(\wh F_{HL_1})\} =o(d)
\] 
at the large deviations scale for $\Graph=\G(n,p)$. This will then imply that on ${\rm UT}_\lambda$ the maximum eigenvalue of the rest of the graph must be large which will yield the desired probability upper bound.}  

\corA{To carry out this task we need a few lemmas}. The first lemma yields a bound on the number of moderate and high degree vertices in the two neighborhood of any given vertex. To state this lemma let us introduce the following notation.

For any $v \in [n]$ we let 
\[
\cN_{M}(v, \Graph):= \{u \in [n]: (u, v) \in E(\Graph) \mbox{ and } u \in V_M(\Graph)\}
\]
and
\[
\cN_M^{(2)}(v, \Graph):= \{u \in [n]: \dist_\Graph(u,v) \le 2 \text{ and } u \in V_M(\Graph)\},
\]
where $\dist_\Graph(\cdot,\cdot)$ denotes the graph distance.  
We similarly define $\cN_H(v, \Graph)$ and $\cN_H^{(2)}(v,\Graph)$. When $\Graph= \G(n,p)$ (which will be the case for most of this section that follows), to ease the notation, we will drop the dependence on $\Graph$ and write $\cN_M(v)$, $\cN_H(v)$ etc. This convention will be adopted also for other notation throughout this section. 

\corA{In the rest of the section we will always assume that 
\[
\al=0 \quad \text{ and } \quad d \gg \sqrt{\f{\log n}{\log \log n}}.
\] 
To avoid repetition, we have chosen not to include the above assumption in the statements of the rest of the lemmas in this section.}

\begin{lem}\label{lem:nbh-in-MH}
For any $\delta >0$ there exists $C_1=C_1(\delta)<\infty$ 
such that 
\beq\label{eq:cNM-1}
\P(\exists v \in [n]: |\cN_M^{(2)}(v)| \ge C_1 \corA{d} \log \corA{d}) \le \exp(-2(1+\delta)^2\corA{d^2} \log \corA{d})
\eeq
and
\beq\label{eq:cNH-1}
\lim_{n \to \infty} \f{\log\P(\exists v \in [n]: |\cN_H^{(2)}(v)| \ge \corA{d}^{1-\eta/2})}{\corA{d^2} \log \corA{d}} = -\infty.
\eeq
\end{lem}

\begin{rmk}\label{rmk:nbh-in-MH}
Note that by definition $\cN_M(v) \subset \cN_M^{(2)}(v)$ and $\cN_H(v) \subset \cN_H^{(2)}(v)$. Therefore, the bounds \eqref{eq:cNM-1} and \eqref{eq:cNH-1} continue to hold when $\cN_M^{(2)}(v)$ and $\cN_H^{(2)}(v)$ are replaced by $\cN_M(v)$ and $\cN_H(v)$, respectively. 
\end{rmk}

\corA{Lemma \ref{lem:nbh-in-MH} and Remark \ref{rmk:nbh-in-MH} are used below to obtain a bound on $\lambda(\Graph_H), \lambda(\Graph_{HM})$, and $\lambda(\wh F_{H L_1})$.} 
%
%
%
Next, fix $C_2 < \infty$ and $\gamma >0$. Let
\[
\cA_1:= \left\{\exists S \subset [n]: |S| \le \corA{d}^2,  \text{ and } e(S) \ge |S| + \gamma \corA{d}^2\right\}
\]
and
\[
\cA_2:= \left\{ \exists S \subset [n]: \corA{d}^2 \le |S| \le \corA{d}^{C_2} \text{ and } e(S) \ge (1+\gamma) |S| \right\}.
\]
The following lemma shows that both $\cA_1$ and $\cA_2$ are unlikely at the large deviations scale. 

\begin{lem}\label{lem:not-many-edges}
For any $\gamma >0$ and $C_2 < \infty$ we have that
\[
\max\{\P(\cA_1), \P(\cA_2)\} \le \exp\left(-\f\gamma2 \corA{d^2} \log n\right),
\]
for all $n$ sufficiently large. 
\end{lem}


The next lemma provides a bound on the number of high degree vertices \corA{which together with the lemma above will be used in the proof of $\lambda(\Graph_{L_1} \cup \Graph_{L_1L_2})=o(d)$ at the large deviations scale.}

\begin{lem}\label{lem:V-H-bd}
The following probability bound holds:
\[
\lim_{n \to \infty}\f{\log\P\left(|V_{H}| \ge  \corA{d}^{1-\eta/2} \right)}{\corA{d}^2 \log \corA{d}} = -\infty.  
\]
\end{lem}



In the following lemma we will show that any connected component of $\Graph_M$ (the underlying graph is again $\G(n,p)$) cannot have too many vertices. \corA{This is needed to obtain a desired bound on $\lambda(\Graph_M)$.} Its statement requires some further notation. Fix $C_3 < \infty$ and let
\[
\cM:= \left\{\exists S \subset V_M: |S| \ge C_3 \corA{d} \log \corA{d} \text{ and } \Graph[S] \text{ is connected}\right\}.
\]
\begin{lem}\label{lem:M-connected}
There exists $C_3 =C_3(\delta)< \infty$ 
such that
\[
\P(\cM) \le \exp(-2(1+\delta)^2 \corA{d}^2 \log \corA{d}).
\]
\end{lem}

\corA{The goal of the next lemma is to carry out iterative pruning procedure that will be needed to tackle $\lambda(\Graph_{L_1} \cup \Graph_{L_1L_2})$}. 
To state this result we need to introduce a couple more notation.

Fix $L \in \N$ and $\alpha_0 \in (0,1)$. For $\ell \in [L-1]$ set 
\[
\mathscr{C}_\ell:= \{\Graph \subset K_n: \updelta(\Graph) \ge 2, \Delta(\Graph) \le (1+\delta) \corA{d}, v(\Graph) \le \corA{d}^{L-\ell+1}\}
\]
and 
\[
\cC_\ell:= \left\{\exists \Graph \subset \G(n,p): \Graph \in \mathscr{C}_\ell \text{ and } \lambda(\Graph) \ge \f{\alpha_0}{2^{\ell-1}} \corA{d}\right\}.
\]

\begin{lem}\label{lem:G-low-lambda}
For any $\ell \in [L-2]$
\[
\P(\cC_\ell) \le \P(\cC_{\ell+1}) + \exp\left( - \f{\alpha_0^2}{C(1+\delta) 2^{2\ell}} \corA{d}^2 \log n\right), 
\]
where $C<\infty$ is some absolute constant.
\end{lem}

The proofs of these lemmas are postponed to Section \ref{sec:prooflemmas}. Below, using these lemmas, we complete the proof of the large deviations upper bound.

\begin{proof}[Proof of Theorem \ref{thm:eig-main} \corA{for $\al >0$} (upper bound)] 
Fix $\vep>0$ sufficiently small. 
We will show that \\
$\lambda (\G(n,p)\setminus \wt {\sf F}_{HL_1}(\G(n,p))) \ge \vep \corA{d}$ 
with a negligible probability at the large deviations scale. This will yield the upper bound. 

\textbf{Step 1.} We claim that
\beq\label{eq:step1}
\P\left(\max\{\lambda(\Graph_H), \lambda(\Graph_{HM})\} \ge \f{\vep\delta \corA{d}}{15} \right) \le 2 \exp\left( -2(1+\delta)^2 \corA{d}^2 \log \corA{d}\right). 
\eeq

By Lemma \ref{lem:graph-eig}(i) and (v) 
\[
\lambda(\Graph_H) \le \max_{v \in V_H} |\cN_H(v)| \quad \text{ and } \quad \lambda(\Graph_{HM}) \le \sqrt{\max_{v \in V_M} |\cN_H(v)| \cdot \max_{u \in V_H} |\cN_M(\corA{u})|}. 
\]
Therefore, \eqref{eq:step1} is now immediate from Lemma \ref{lem:nbh-in-MH} (see also Remark \ref{rmk:nbh-in-MH}). 

\textbf{Step 2.} We next claim that 
\beq\label{eq:step2}
\P\left(\lambda(\Graph_M) \ge \f{\vep \delta \corA{d}}{15} \right) \le 2 \exp\left( -2(1+\delta)^2 \corA{d}^2 \log \corA{d}\right).
\eeq

Since the top eigenvalue of a graph is the maximum of the top eigenvalues of its connected components (see Lemma \ref{lem:graph-eig}) we obtain that
{\allowdisplaybreaks
\begin{multline}\label{eq:step2a}
\P\left(\lambda(\Graph_M) \ge \f{\vep \delta \corA{d}}{15} \right) \le \P\left(\exists S \subset V_M: \Graph[S] \text{ is connected  and } \lambda (\Graph[S]) \ge \f{\vep \delta \corA{d}}{15}\right)\\
\le \P\left(\exists S \subset [n]: |S| \le C_3 \corA{d} \log \corA{d} \text{ and } \lambda (\Graph[S]) \ge \f{\vep \delta \corA{d}}{15} \right) + \P(\cM). 
\end{multline}
}
\corA{The bound on the second term in the \abbr{RHS} of \eqref{eq:step2a} follows from Lemma \ref{lem:M-connected}. To bound the first term we note that, for any $c_\star >0$, upon taking a union bound,}
{\allowdisplaybreaks
\begin{multline}\label{eq:step2b}
 \P\left(\exists S \subset [n]: |S| \le C_3 \corA{d} \log \corA{d} \text{ and } e(S) \ge c_\star \corA{d}^2\right) 
\le \sum_{w \ge c_\star \corA{d}^2} \sum_{s \le C_3 \corA{d} \log \corA{d}} \binom{n}{s}\binom{s^2}{w} p^{w} \\
\le \sum_{w \ge c_\star \corA{d}^2} \sum_{s \le C_3 \corA{d} \log \corA{d}} \exp\left(s \log n +w \log (2es) - w \log (1/p)\right)\\
\le O(\corA{d} \log \corA{d}) \sum_{w \ge c_\star \corA{d}^2} \exp \left(-\f{w}{2} \log n\right) \le n^3 \exp\left(-\f{c_\star}{2} \corA{d}^2 \log n\right) \le \exp\left(-\f{c_\star}{4} \corA{d}^2 \log n\right),
\end{multline}
}
where in the second inequality we used Stirling's approximation, and the third inequality is due to the facts that $s = O(\corA{d} \log \corA{d})$, $\log \corA{d} \ll \log n$, $\corA{d}=np \gg 1$, and $w \ge c_\star \corA{d}^2$.  

Upon setting $c_\star = \vep \delta/450$, we observe that \eqref{eq:step2} follows from the fact that $\lambda(\Graph) \le \sqrt{2 e(\Graph)}$ (see the upper bound in Lemma \ref{lem:graph-eig}(i)), \eqref{eq:step2a}-\eqref{eq:step2b}, and Lemma \ref{lem:M-connected}. 

\textbf{Step 3.} We aim to show that, for all large $n$, 
\beq\label{eq:step3}
\P\left(\lambda(\Graph_{L_1} \cup\Graph_{L_1L_2}) \ge \f{\vep \delta \corA{d}}{15} \right) \le  \exp\left( - 2(1+\delta)^2   \corA{d}^2 \log \corA{d} \right).
\eeq

\corA{We will first find a bound for $\lambda(\wh\Graph_{L_1} \cup \wh\Graph_{L_1L_2})$. As the rest of the subgraph $\Graph_{L_1} \cup\Graph_{L_1L_2}$ is a forest, the rest will have a small spectral radius. Putting these two pieces together we will get \eqref{eq:step3}.}

To this end, we begin by noting that 
\beq\label{eq:step3a}
\P(\Delta(\G(n,p)) \ge \corA{d}^{2+\eta/3}) \le n \P(\dBin(n-1,p) \ge \corA{d}^{2+\eta/3}) \le \exp\left(-\f12 \corA{d}^{2+\eta/3} \log \corA{d} \right),
\eeq
which is negligible at the large deviations scale. As every vertex in $V_{L_1}$ must be connected to some vertex in $V_H$ we observe that on the event
\[
\Omega_0:=\left\{\Delta:=\Delta(\G(n,p)) \le \corA{d}^{2+\eta/3}\right\} \cap \left\{ |V_H| \le \corA{d}^{1-\eta/2}\right\}
\]
we have that $|V_{L_1}| \le \Delta \cdot |V_H| \le \corA{d}^{3-\eta/6}$. By a same reasoning, on the event $\Omega_0$, we also have that  
\[
\left|\left\{u \in V_{L_2}: (u,v) \in E(\G(n,p)) \text{ for some } v \in V_{L_1}\right\}\right| \le \left[\max_{v \in V_{L_1}} \deg(v)\right] \cdot |V_{L_1}|\le \corA{d}^{4-\eta/12}. 
\]
Thus, on the event $\Omega_0$, we have that $v(\Graph_{L_1} \cup \Graph_{L_1L_2}) \le \corA{d}^4$. Therefore, setting $L=4$ and $\alpha = \vep \delta/30$, and applying Lemma \ref{lem:G-low-lambda} we derive that 
\beq\label{eq:step3b}
\P\left(\lambda(\wh\Graph_{L_1} \cup\wh \Graph_{L_1L_2}) \ge \f{\vep \delta \corA{d}}{30} \right) \le \P(\Omega_0^c) + \P(\cC_1) \le \P(\Omega_0^c) + \P(\cC_3) + 2 \exp(-c_1\vep^2 \corA{d}^2 \log n),
\eeq
for some constant $c_1 >0$, depending only on $\delta$. Observe that, by \eqref{eq:step3a} and Lemma \ref{lem:V-H-bd}
\beq\label{eq:step3c}
{-\log \P(\Omega_0^c)} \gg {\corA{d}^2 \log \corA{d}}. 
\eeq 

On the other hand, by Lemma \ref{lem:graph-eig}(viii) for any graph $\Graph$ with $\updelta(\Graph) \ge 2$ and $\Delta(\Graph) \lesssim \corA{d}$ we have that 
\[
\lambda(\Graph) \gtrsim \corA{d} \Rightarrow e(\Graph) - v(\Graph) \gtrsim \corA{d}^2. 
\]
Therefore, we find that $\cC_3 \subset \cA_1$ for a suitably chosen $\gamma$, and hence by Lemma \ref{lem:not-many-edges} we deduce that
\beq\label{eq:step3d}
-\log \P(\cC_3) \gg \corA{d}^2 \log \corA{d}.
\eeq

Recall that ${\sf F}_{L_1}=\Graph_{L_1}\setminus \wh \Graph_{L_1}$ and ${\sf F}_{L_1L_2}=\Graph_{L_1L_2}\setminus \wh \Graph_{L_1L_2}$ are forests. Hence, by Lemma \ref{lem:graph-eig}(iii)
\[
\lambda({\sf F}_{L_1L_2} \cup {\sf F}_{L_1}) = O\left(\sqrt{\max_{v \in V_L} \deg(v)}\right) = O(\sqrt{\corA{d}}).
\]
This observation together with \eqref{eq:step3b}-\eqref{eq:step3d} now yield \eqref{eq:step3}.  

\textbf{Step 4.} Our next goal is to derive that 
\beq\label{eq:step4}
\P\left(\lambda(\Graph_{ML}) \ge \f{\vep \delta \corA{d}}{15} \right) \le  \exp\left( -2(1+\delta)^2 \corA{d}^2 \log \corA{d}\right).
\eeq

Fix any $\wt c >0$. As $\Graph_{ML}$ is a bipartite graph with vertex bipartition $V_M$ and $V_L$, Lemma \ref{lem:graph-eig}(vii) implies that
\beq\label{eq:step4a}
\left\{\lambda(\Graph_{ML}) \ge \wt c \corA{d}\right\} \subset \left\{\exists v \in V_M, \cW \subset  V_L: w \sim v \, \forall \,  w \in \cW \text{ and } \sum_{w \in \cW} \deg_{\Graph_{ML}}(w) \ge \wt c^2 \corA{d}^2 \right\} =:\wt \Omega_0. 
\eeq
Using again that $\Graph_{ML}$ is a bipartite graph we also observe that on the event $\wt \Omega_0$ there exists $\cU \subset V_M$ (namely the set of neighbors in $V_M$ of the vertices in $\cW$) such that $e(\cW, \cU \cup\{v\}) \ge \wt c^2 \corA{d}^2$. Notice that $\cU \subset \cN_M^{(2)}(v)$. Therefore, noting that $\deg(v) \le \varpi_n$ for any $v \in V_M$ (see Definition \ref{dfn:graph-decompose}) we deduce that
{\allowdisplaybreaks
\begin{multline}\label{eq:step4b}
\wt \Omega_0 \setminus \left\{ \exists v' \in [n] : |\cN_M^{(2)}(v')| \ge C_1 \corA{d} \log \corA{d} \right\} \\
\subset \left\{\exists v \in [n], \cU \subset V_M, \cW \subset V_L: |\cU| \le C_1 \corA{d} \log \corA{d}, |\cW| \le \varpi_n, e(\cU, \cW \cup\{v\}) \ge \wt c^2 \corA{d}^2\right\}
=: \wh \Omega_0,  
\end{multline}
}
where $C_1$ is as in Lemma \ref{lem:nbh-in-MH}. By the union bound and Lemma \ref{lem:ash} (see also \eqref{eq:entropy-large-lambda}) we now obtain that 
\begin{multline}\label{eq:step4c}
\P(\wh \Omega_0) \le \sum_{\substack{\gu \le \varpi_n \\ \gw \le C_1 \corA{d} \log \corA{d}}}  \sum_{v \in [n]} \sum_{\substack{\cU,  \cW \subset[n]\\ |\cU|=\gu, |\cW|=\gw}}  \P(e(\cU, \cW \cup\{u\}) \ge \wt c^2 \corA{d}^2)\\
\le n^2 \cdot \exp\left( \varpi_n \log n + O(\corA{d} \log \corA{d}) \log n - \wt c^2 \corA{d}^2 \log n \cdot (1+o(1))\right) \\
\le n^2 \exp \left(-\frac{\wt c^2}{2} \corA{d}^2 \log n\right) \le \exp \left(-\frac{\wt c^2}{4} \corA{d}^2 \log n\right),
\end{multline}
where in the second and third steps we used that $\varpi_n, \corA{d} \log \corA{d} \ll \corA{d}^2 \le n^{o(1)}$, and in the last step we used that $\corA{d} \gg 1$. Now, upon choosing $\wt c= \vep \delta /15$,  \eqref{eq:step4} follows from \eqref{eq:step4a}-\eqref{eq:step4c} and Lemma \ref{lem:nbh-in-MH}. 

\textbf{Step 5.} We will show that 
\beq\label{eq:step5}
\P\left(\lambda(\wh {\sf F}_{HL_1}) \ge \f{\vep \delta \corA{d}}{15} \right) \le \exp\left( -2(1+\delta)^2 \corA{d}^2 \log \corA{d} \right).
\eeq

We claim that 
\beq\label{eq:step5a}
\lim_{n \to \infty}\f{\log\P(\Delta(\wh{\sf F}_{HL_1}) \ge (1+\delta) \corA{d})}{\corA{d}^2 \log \corA{d}} = -\infty.   
\eeq
Since $\wh {\sf F}_{HL_1}$ is a star, upon applying Lemma \ref{lem:graph-eig}(iii), \eqref{eq:step5} is immediate from \eqref{eq:step5a}. By definition $\max_{v \in V_L} \deg(v) < (1+\delta) \corA{d}$. Therefore, \corA{as $\wh{\sf F}_{HL_1}$ is a forest, and as it does not contain any star that has its center vertex belonging to $V_H$}, we find that
\[
\left\{\Delta(\wh{\sf F}_{HL_1}) \ge (1+\delta) \corA{d} \right\} \subset \left\{ \exists v \in [n]: |\cN_H^{(2)}(v)| \ge (1+\delta) \corA{d} \right\}.
\]
Thus \eqref{eq:step5a} follows from Lemma \ref{lem:nbh-in-MH}. 

\textbf{Step 6.} Finally we note that
\beq\label{eq:step6}
\lambda(\Graph_{L_2} \cup \wt {\sf F}_{HL_1}) \le \max\left\{\sqrt{\Delta(\G(n,p))}, (1+\delta(1-\vep)) \corA{d}\right\}.
\eeq
This is immediate since $\Graph_{L_2}$ and $\wt {\sf F}_{HL_1}$ are vertex disjoint, and $\wt {\sf F}_{HL_1}$ is a vertex disjoint union of stars (apply Lemma \ref{lem:graph-eig}(ii) and (iv)). 

We may now complete the proof of the theorem. Indeed, by Lemma \ref{lem:graph-eig}(ii)
\beq\label{eq:stepfin}
{\rm UT}_{\lambda}( \delta) \subset \left\{\lambda(\Graph_{L_2} \cup \wt {\sf F}_{HL_1}) \ge (1+\delta(1-\vep/2)) \corA{d}\right\} \cup \Omega_\star 
\subset {\rm UT}_\Delta((1+\delta(1-\vep/2))^2) \cup \Omega_\star,
\eeq
where
\[
\Omega_\star:= \left\{ \lambda(\Graph_H) + \lambda(\Graph_{HM}) + \lambda(\Graph_M) +\lambda(\Graph_{ML}) +\lambda(\wh{\sf F}_{HL_1}) +\lambda(\Graph_{L_1} \cup \Graph_{L_1L_2}) \ge \f{\vep \delta \corA{d}}{2}\right\},
\]
and the last step in \eqref{eq:stepfin} follows from \eqref{eq:step6}. By \eqref{eq:step1}, \eqref{eq:step2}, \eqref{eq:step3}, \eqref{eq:step4}, and \eqref{eq:step5} we have that
\beq\label{eq:stepfin1}
\P(\Omega_\star) \lesssim \exp(-2(1+\delta)^2 \corA{d}^2 \log \corA{d}). 
\eeq
This together with Lemma \ref{lem:ash} and \eqref{eq:stepfin} now yield that
\[
\limsup_{n \to \infty} \f{\log \P({\rm UT}_\lambda(\delta))}{\corA{d}^2 \log \corA{d}} \le - (1+\delta(1-\vep/2))^2.
\]
Since $\vep>0$ is arbitrary this completes the proof. 
\end{proof}


\begin{proof}[Proof of Corollary \ref{cor:typ-strc}]
We see from \eqref{eq:stepfin} that
\[
{\rm UT}_\lambda(\delta)  \cap {\rm UT}_\Delta((1+(1-\chi)\delta)^2)^\complement \subset \Omega_\star,
\]
for any $\chi > \vep /2$. Therefore, using \eqref{eq:stepfin1} for the numerator below, and \eqref{eq:eig-lbd} and the lower bound from Lemma \ref{lem:ash} for the denominator below we deduce that
\[
\limsup_{n \to \infty} \P\left({\rm UT}_\Delta((1+(1-\chi)\delta)^2)^\complement \big| {\rm UT}_\lambda(\delta)\right) \le \limsup_{n \to \infty} \f{\P(\Omega_\star)}{\P({\rm UT}_\lambda(\delta))}=0. 
\]
Since $\vep>0$ is arbitrary this completes the proof. 
\end{proof}


\subsection{Proofs of Lemmas \ref{lem:nbh-in-MH} and \ref{lem:not-many-edges}-\ref{lem:G-low-lambda}}\label{sec:prooflemmas}
We start with the proof of Lemma \ref{lem:V-H-bd}. 

\begin{proof}[Proof of Lemma \ref{lem:V-H-bd}]
This is a simple application of the binomial tail probability bound. Indeed, fix a set $\cU \subset [n]$ of cardinality $t$. \corA{Recall $\varpi_n$ from \eqref{eq:varpi}}. Observe that  
\begin{multline}\label{eq:VH-bd-1}
\P\left(\sum_{u \in \cU} \deg(u) \ge \varpi_n t \right) = \P(2 e(\cU) + e(\cU, [n]\setminus \cU) \ge \varpi_n t) \\
\le \P(e(\cU) \ge \varpi_n t /3) + \P(e(\cU, [n]\setminus \cU) \ge \varpi_n t /3).
\end{multline}
As $\varpi_n \gg \corA{d}$, and 
\[
e(\cU, [n]\setminus \cU ) \stackrel{d}{=} \dBin(t(n-t), p) \quad \text{ and } \quad e(\cU) \stackrel{d}{=} \dBin\left(\binom{t}{2}, p\right),
\]
using Lemma \ref{lem:ash} (see also \eqref{eq:entropy-large-lambda}) we obtain that, there exists some $c >0$, depending only on $\eta$, such that the \abbr{LHS} of \eqref{eq:VH-bd-1} is bounded above by $\exp(-c \varpi_n t \log \log n)$. Hence, applying the union bound we now deduce that
{\allowdisplaybreaks
\begin{multline*}
\P(|V_H| \ge \corA{d}^{1-\eta/2}) \le \sum_{\substack{\cU \subset[n]\\ |\cU| = \corA{d}^{1-\eta/2}}} \P\left(\sum_{u \in \cU} \deg(u) \ge \varpi_n \corA{d}^{1-\eta/2} \right)\\
\le \exp\left(-\corA{d}^{1-\eta/2} (c \varpi_n \log \log n - \log n)\right) \le \exp\left(-\f{c}{2}\corA{d}^{1-\eta/2} \varpi_n \log \log n\right)\\
\le \exp\left(-\f{c}{2}\corA{d}^{2+\eta/2} \log \log n\right),
\end{multline*}
}
where in the penultimate step we have used that $\varpi_n \log \log n \gg \log n$ (follows from \eqref{eq:varpi} and that $\corA{d} \gg \sqrt{\log n/\log \log n}$) and in the last step we used that $\varpi_n \ge \corA{d}^{1+\eta}$. This completes the proof. 
\end{proof}

Next we prove Lemma \ref{lem:not-many-edges}. \corA{This also uses the Binomial tail bound of Lemma \ref{lem:ash}.}

\begin{proof}[Proof of Lemma \ref{lem:not-many-edges}]
Fix a set $S \subset [n]$ with $|S| = s \le \corA{d}^{C_2}$. Then $e(S) \stackrel{d}{=}\dBin(\binom{s}{2}, p)$. Therefore, by Lemma \ref{lem:ash} (see also \eqref{eq:entropy-large-lambda}), as $\log(np) \ll \log n$, we have that
{\allowdisplaybreaks
\begin{multline*}
\P(e(S) \ge (1+\gamma)s) \le \exp\left(-(1+\gamma)s \left\{ \log n - \log(\corA{d} s) \right\} \cdot(1+o(1)) \right) \\
\le \exp\left(-(1+\gamma) s \log n (1+o(1))\right),
\end{multline*}
}where in the last inequality we used the upper bound on $s$.
Now the uper bound on $\P(\cA_2)$ follows after taking a union over the choices of $S \subset [n]$ with $\corA{d}^2 \le |S| \le \corA{d}^{C_2}$. 

The proof of the upper bound on $\P(\cA_1)$ is similar. Indeed, fixing $S \subset [n]$ with $|S| =s \le \corA{d}^2$, using the binomial tail probability bound and that $\log \corA{d} \ll \log n$ once more, we find that
\[
\P(e(S) \ge s + \gamma \corA{d}^2) \le \exp(-(s+\gamma \corA{d}^2) \cdot (1+o(1)) \log n).
\]
Now the proof follows by taking a union over $S \subset [n]$ such that $|S| \le \corA{d}^2$.  
\end{proof}

Now we proceed to prove Lemma \ref{lem:nbh-in-MH}. Let us add that a similar result was proved in \cite{KS} (see Lemma 2.4 there). 

\begin{proof}[Proof of Lemma \ref{lem:nbh-in-MH}]
Fix any $v \in [n]$. We will show that
\beq\label{eq:cN-M}
\P(|\cN_M^{(2)}(v)| \ge C_1 \corA{d} \log \corA{d}) \le \exp(-3(1+\delta)^2 \corA{d}^2 \log \corA{d}).
\eeq
Since $\corA{d}^2 \log \corA{d} \gg \log n$, an application of the union bound then yields \eqref{eq:cNM-1}. 

Turning to prove \eqref{eq:cN-M} let us fix a set $\cU\subset [n] \setminus \{v\}$ of cardinality $\ell$ and $\cW \supset \cU \cup\{v\}$ such that $\ell +1 \le |\cW| \le 2\ell+1$. Fix a collection of edges $\cE$ such that at least one end point of the edges in $\cE$ \corA{belongs} to $\cU$, and $|\cE| = |\cW| -1 \le 2\ell$. 

\corA{Later in the proof the set $\cU$ will be the (random) set of the vertices that are of distance at most two from $v$ and have moderately high degrees. The set $\cW$ will be the smallest connected subgraph containing $v$ and $\cU$, and $\cE$ will be the edge set of that graph. Below we find bounds on the probabilities of certain events involving deterministic choices of $\cU, \cW$, and $\cE$ with the aforementioned properties which will then allow us to take a union bound over the allowable ranges of these sets.} 

Let $\{a_{i,j}\}_{i < j}$ be i.i.d.~$\dBer(p)$ and for $i >j$ set $a_{i,j}=a_{j,i}$. 
We claim that for any $\ell \le \corA{d}^2$
{\allowdisplaybreaks
\begin{multline}\label{eq:cN-M2}
\P\left(\sum_{u \in \cU} \deg(u) \ge (1+\delta(1-\vep)) \corA{d}\ell \Big| a_{i,j} =1 \, \forall \, \{i,j\} \in \cE \right) \\
\le \P\left(\sum_{\substack{i,j \in \cU\\ \{i,j\} \notin \cE}} a_{i,j}+ \sum_{\substack{i \in \cU, j \notin \cU\\ \{i,j\} \notin \cE}} a_{i,j} \ge (1+\delta(1-2\vep)) \corA{d}\ell \Bigg| a_{i,j} =1 \, \forall \, \{i,j\} \in \cE \right) \\
\le \P\left(2 e(\cU) + e(\cU, [n]\setminus \cU ) \ge (1+\delta(1-2\vep))  \corA{d}\ell\right) \\
\le \P(\dBin(n\ell, p) \ge (1+\delta(1-3\vep))\corA{d}\ell) + \P(\dBin(\ell^2/2, p) \ge \delta\vep \corA{d}\ell /2) 
\le \exp(-c_\delta \corA{d} \ell),
\end{multline}
}
for some $c_\delta >0$. The first inequality above follows from that $|\cE| \le 2\ell$ and $\corA{d} \gg 1$.  The second inequality is a consequence of the fact that the edges in $\G(n,p)$ are independent and stochastic domination of binomial random variables. The penultimate inequality follows from a union bound. As $\log \corA{d} \ll \log (1/p)$, the last inequality in \eqref{eq:cN-M2} follows from Lemma \ref{lem:ash} (see also \eqref{eq:chernoff} and \eqref{eq:entropy-large-lambda}). 


\corA{Recall the definition of $V_M$ from Definition \ref{dfn:graph-decompose}}. Now note that the event $|\cN_M^{(2)}(v)| \ge \ell$ implies that there exists a set $\cU \subset V_M$ with $|\cU| =\ell$ such that
\beq\label{eq:VM}
\sum_{i \in \cU, j \in [n]} a_{i,j}= \sum_{u \in \cU} \deg(u) \ge (1+\delta(1-\vep)) \corA{d} \ell.
\eeq
Let $\cW$ be the set of vertices of the smallest subgraph of $\G(n,p)$ containing $\cU \cup\{v\}$ that is connected. Let $\cE$ be the edge set of that subgraph. Since the vertices of $\cU$ are at a distance at most two from $v$ it follows that $\ell +1 \le |\cW| \le 2 \ell +1$. Furthermore, note that the smallest connected subgraph being a tree we also have that $|\cE| = |\cW|-1$. 

Equipped with these observations, to bound the probability that $|\cN_M^{(2)}(v)|$ exceeds $\ell$ we first fix $\cU, \cW$, and $\cE$. Then we find a bound on the probability that \eqref{eq:VM} and the event that $a_{i,j} =1$ for all $\{i,j\} \in \cE$ hold. Finally we take a union bound over the choices of $\cU$, $\cW$, and $\cE$. Carrying out these steps, by \eqref{eq:cN-M2} and Stirling's approximation, we obtain that
{\allowdisplaybreaks
\begin{multline*}
\P(|\cN_M^{(2)}(v)| \ge \ell) \le \sum_{w=\ell+1}^{2\ell+1} \sum_{\substack{\cU \subset [n]\\ |\cU| =\ell}} \sum_{\substack{\cW \subset [n]\\ \cW \supset \cU \cup \{v\}\\|\cW|=w}} \sum_{\substack{\cE: \cE \subset \cU \times \cW\\ |\cE| =w-1}}\P\left(a_{i,j} =1 \, \forall \, \{i,j\} \in \cE \right) \\
\qquad \qquad \qquad \cdot \P\left(\sum_{u \in \cU} \deg(u) \ge (1+\delta(1-\vep)) \corA{d} \ell \Big| a_{i,j} =1 \, \forall \, \{i,j\} \in \cE \right)\\
\le \sum_{w=\ell+1}^{2\ell+1} \binom{n}{\ell} \binom{n-\ell-1}{w-\ell-1} \binom{\ell w}{w-1} \cdot p^{w-1} \exp(-c_\delta \corA{d} \ell)\\
\le \sum_{w=\ell+1}^{2\ell+1} \corA{d}^{w} (e \ell)^w \cdot \frac{\binom{w}{\ell}}{w!} \cdot p^{-1} \exp(-c_\delta \corA{d} \ell) \\
\le 2\ell n \cdot (3 e^2 \corA{d})^{3\ell} \exp(-c_\delta \corA{d} \ell) \le \exp\left(\log n - \f{c_\delta}{2} \corA{d} \ell\right),
\end{multline*}}
where in the penultimate step we used that $p \ge n^{-1}$ and in the last step we used that $\corA{d} \gg 1$. Now, upon choosing $\ell= C_1 \corA{d} \log \corA{d}$, for some large $C_1< \infty$, depending only on $\delta$, as $\corA{d}^2 \log \corA{d} \gg \log n$, we obtain \eqref{eq:cN-M}. 

The proof of \eqref{eq:cNH-1} being similar to that of \eqref{eq:cNM-1} is omitted. This completes the proof of the lemma. 
\end{proof}

Next we prove Lemma \ref{lem:M-connected}. 

\begin{proof}[Proof of Lemma \ref{lem:M-connected}]
\corA{The proof relies on the following observation: If a subset $S \subset V_M$ of large size is connected then the difference between $\sum_{v \in S} \deg(v)$ and the number of edges in any of its spanning tree is large with high probability. Since $S$ can be a random set we first fix a set of vertices $\cT$ and a spanning tree $\T$ on those vertices. We then find the probability of the event that the difference between the sum of degrees of vertices in $\cT$ and the number of edges in $\T$ is large. Then we do a union over the allowable choices of $\cT$ and $\T$. Since $\T$ is a tree on $\cT$, the number of possible choices for $\T$ for any given $\cT$ is not too large, which helps us in the union bound. Below we carry out the details.}  

Fix a subset of vertices $\cT \subset [n]$ of cardinality $t$, and a subset $\cE \in \binom{\cT}{2}$ such that $|\cE|=t-1$. 

Arguing similarly as in \eqref{eq:cN-M2} we observe that 
{\allowdisplaybreaks
\begin{multline}\label{eq:cond-M}
\P\left(\sum_{u \in \cT} \deg(u) \ge (1+\delta(1-\vep)) \corA{d} t \Big| a_{i,j} =1 \, \forall \, (i,j) \in \cE \right) \\
\le \P\left(\sum_{\substack{i,j \in \cT\\ \{i,j\} \notin \cE}} a_{i,j}+ \sum_{i \in \cT, j \notin \cT} a_{i,j} \ge (1+\delta(1-2\vep)) \corA{d} t \Bigg| a_{i,j} =1 \, \forall \, (i,j) \in \cE \right) \\
\le\P\left(2 e(\cT) + e (\cT, [n]\setminus \cT) \ge (1+\delta (1-2\vep)) \corA{d} t\right). 
\end{multline}
}

Now we claim that, there exists some constant $c_\delta >0$, depending only on $\delta$ such that 
\beq\label{eq:condM-rhs}
\P\left(2 e(\cT) + e (\cT, [n]\setminus \cT) \ge (1+\delta (1-2\vep)) \corA{d} t\right) \le \gp_t:=\left\{
\begin{array}{ll}
\exp(-c_\delta \corA{d} t) & \mbox{ if } t \le \corA{d}^2,\\
\exp( -\vep (\delta \wedge 1) \corA{d} t/2) & \mbox{ if } \corA{d}^2 \le t \le \vep (\delta \wedge 1) n,\\
\exp( -\wh c \corA{d} t) & \mbox{ if }  \vep (\delta \wedge 1) n \le t \le c_0 n, 
\end{array}
\right.
\eeq
where $c_0 \in [\vep (\delta \wedge 1), 1)$ to be determined below, and $\wh c >0$ is some constant depending on $\delta, \vep$, and $c_0$. To see this, for $t \le \vep(\delta \wedge 1) n$, by the triangle inequality, we observe that  
\begin{multline*}
\P\left(2 e(\cT) + e (\cT, [n]\setminus \cT) \ge (1+\delta (1-2\vep)) np t\right) \\
\le \P\left(\dBin(t(n-t), p) \ge (1+\delta (1-6\vep)) np t\right) + \P\left(\dBin\left(\binom{t}{2}, p\right) \ge \binom{t}{2} p + \vep \delta np t\right),
\end{multline*}
By \eqref{eq:chernoff} it follows that the first term in the \abbr{RHS} is bounded above by $\exp(-2 c_\delta np t)$. On the other hand, by Lemma \ref{lem:ash} and \eqref{eq:entropy-large-lambda} we obtain that the second term in the \abbr{RHS} is bounded above by $\exp(-\vep \delta np t \log n/2)$ for $t \le n^2p^2$, while applying the Chernoff bound (see \eqref{eq:chernoff}) with $\gamma=2 \vep (\delta \wedge 1) n/t$ we find that the same term is bounded by $\exp(-\vep (\delta \wedge 1)npt/2)$ for $ t \le \vep (\delta \wedge 1) n$. Combining these estimates we obtain \eqref{eq:condM-rhs} for all $t \le  \vep (\delta \wedge 1) n$. To prove the remaining range of $t$ we use triangle inequality again to obtain that
{\allowdisplaybreaks
\begin{multline*}
\P\left(2 e(\cT) + e (\cT, [n]\setminus \cT) \ge (1+\delta (1-2\vep)) np t\right) \\
\le \P\left(\dBin(t(n-t), p) \ge (1+\delta (1-\vep))  t(n-t)p\right) + \P\left(\dBin\left(\binom{t}{2}, p\right) \ge (1+\delta (1-\vep))\binom{t}{2} p \right),
\end{multline*}
}
and then use the Chernoff bound and the fact that $ \vep(\delta \wedge 1) n \le t \le c_0 n$. This proves \eqref{eq:condM-rhs}. 

\corA{Next we argue that size of $V_M$ is not too large. This will later allow us to restrict the size of $\cT$ in the union bound so that we can use the bound \eqref{eq:condM-rhs}. To this end, we note that} if $S \subset V_M$ then $2e(\G(n,p)) \ge \sum_{u \in S}\deg(u) \ge (1+\delta(1-\vep)) \corA{d}|S|$. Therefore, for $c_0=c_0(\delta, \vep) \in (0,1)$ such that $c_0(1+\delta(1-\vep)) \ge (1+\delta/2)$, by Chernoff bound, we find that 
\beq\label{eq:c-0}
\P(\Omega_M) \le \P\left(e(\G(n,p)) \ge (1+\delta/2) \cdot \binom{n}{2} p\right) \le \exp(-\wt c_\delta n^2 p),
\eeq
for some $\wt c_\delta >0$, depending only on $\delta$, where $\Omega_M:= \{\exists S \subset V_M: |S| \ge c_0 n\}$. For the rest of the proof we will work with this choice of $c_0$. 

Equipped with all necessary bounds we now complete the proof of the lemma. To this end, we notice that on the event $\cM \setminus \Omega_M$ \corA{there must exist} a set of vertices $\cT \subset [n]$ such that $|\cT| =t$, where $t \in [C_3 \corA{d} \log \corA{d}, c_0 n]$, and $\Graph[\cT]$ is connected. This further implies \corA{the existence of a spanning tree} $\T$ of $\cT$ with edge set $\cE \subset \binom{\cT}{2}$ such that $a_{i,j}=1$ for $(i,j) \in \cE$, where $a_{i,j}$'s are as in \eqref{eq:cond-M}. 

Hence, using that the number of trees on $t$ vertices is bounded by $t^{t+2}$, applying \eqref{eq:cond-M}-\eqref{eq:condM-rhs}, and a union bound we find that
{\allowdisplaybreaks
\begin{multline*}
\P(\cM) -\P(\Omega_M) \le \sum _{t \ge C_3 \corA{d} \log \corA{d}}^{c_0 n} \sum_{\cT: |\cT| =t} \sum_{\cE} p^{t-1} \P\left(\sum_{u \in \cT} \deg(u) \ge (1+\delta(1-\vep)) \corA{d} t \Big| a_{i,j} =1 \, \forall \, (i,j) \in \cE \right) \\
\le \sum _{t = C_3 \corA{d} \log \corA{d}}^{c_0 n} \binom{n}{t} t^{t+2} p^{t-1} \gp_t 
\le n^3 \sum _{t = C_3 \corA{d} \log \corA{d}}^{c_0 n} (e \corA{d})^t \gp_t 
\\
\le n^3 \left( \exp\left(-\f{c_\delta C_3}{2}\corA{d}^2 \log \corA{d}\right)+ \exp(-c_\star \corA{d}^3)\right),
\end{multline*}}
for some constant $c_\star >0$. Now recalling that $\corA{d}^2 \log \corA{d} \gg \log n$ and $n \gg \corA{d} \log \corA{d}$, the proof completes upon applying \eqref{eq:c-0}, and choosing $C_3< \infty$ sufficiently large, depending only $\delta$. 
\end{proof}

We end this section with the proof of Lemma \ref{lem:G-low-lambda}. 

\begin{proof}[Proof of Lemma \ref{lem:G-low-lambda}]
We begin by noting that, by Lemma \ref{lem:graph-eig}(viii), for any graph $\Graph$ with $\updelta(\Graph) \ge 2$ and $\Delta(\Graph) \le (1+\delta) \corA{d}$, as $\corA{d} \gg 1$, the lower bound $\lambda(\Graph) \ge \f{\alpha}{2^{\ell-1}} \corA{d}$ implies that
\[
e(\Graph) - v(\Graph) \ge \f{\alpha^2}{2^{2\ell}} \corA{d}^2,
\]
for all large $n$. 

Therefore, by Lemma \ref{lem:not-many-edges} we have that
\begin{multline}
\P(\cC_\ell) \\
\le \P\left(\exists \Graph \subset \G(n,p): \Graph \in \mathscr{C}_\ell, \lambda(\Graph) \ge \f{\alpha}{2^{\ell-1}}\corA{d}, \text{ and }v(\Graph) \ge \corA{d}^2\right) + \exp\left(-\f{\alpha^2}{2^{2\ell+1}}\corA{d}^2 \log n\right)\\
 \le \P(\wt{{\cC}}_\ell)+ 2\exp\left(-\f{\gamma_\star}{8}\corA{d}^2 \log n\right),
\end{multline}
where
\begin{equation*}
\wt{{\cC}}_\ell:= \left\{\exists \Graph \subset \G(n,p): \Graph \in \mathscr{C}_\ell, \lambda(\Graph) \ge \f{\alpha}{2^{\ell-1}}\corA{d}, e(\Graph) \le \left(1+\f{\gamma_\star}{4}\right) v(\Graph),  \text{ and }  v(\Graph) \ge \corA{d}^2\right\}
\end{equation*}
and $\gamma_\star:= \f{\alpha^2}{(1+\delta) \cdot 9 \cdot 2^{2\ell}}$. It now remains to show that $\wt \cC_\ell \subset \cC_{\ell+1}$. This will complete the proof. 

Turning to do this task let us decompose the vertices of $\Graph$ as follows:
\[
V_1:=\{v \in V(\Graph): \deg_\Graph(v) \ge \gamma_\star \corA{d}\} \quad \text{ and } \quad V_2:= V(\Graph)\setminus V_1. 
\]
Define
\[
\Graph_1:= \Graph[V_1], \quad \Graph_2:= \Graph[V_2], \quad \text{ and } \Graph_{12}:= \Graph[V_1, V_2]. 
\]
Further let $\wh \Graph_1$ be the two core of $\Graph_1$ and set ${\sf F}_1:=\Graph_1\setminus \wh \Graph_1$. Note that ${\sf F}_1$ is a forest. Also observe that for any $\Graph$ with $\Delta(\Graph) \le (1+\delta)\corA{d}$, due to our choice of $\gamma_\star$, by Lemma \ref{lem:graph-eig}(i), (iii), and (v) we have that
\[
\max\{\lambda({\sf F}_1), \lambda(\Graph_{12}), \lambda(\Graph_2)\} \le \f{\alpha}{3 \cdot 2^\ell} \corA{d}.  
\]
Thus, by Lemma \ref{lem:graph-eig}(ii)
\[
\lambda(\Graph) \ge \f{\alpha}{2^{\ell-1}} \corA{d} \Longrightarrow \lambda(\wh \Graph_1) \ge \f{\alpha}{2^{\ell}} \corA{d}.
\]
Moreover, for $\Graph$ satisfying the hypotheses of the event $\wt \cC_\ell$, we notice that
\[
\f{\gamma_\star}{2} \corA{d} \cdot |V_1| \le \sum_{v \in V(\Graph)} (\deg_\Graph(v) -2) = 2 (e(\Graph) - v(\Graph)) \le \f{\gamma_\star}{2} v(\Graph),
\]
yielding that $v(\wh \Graph_1) \le \corA{d}^{L-\ell}$. Hence, the subgraph $\wh \Graph_1 \subset \Graph$ satisfies the hypotheses of the event $\cC_{\ell+1}$. Therefore $\wt \cC_\ell \subset \cC_{\ell+1}$, and the proof of the lemma is now complete. 
\end{proof}

\section{Proofs of Theorem \ref{thm:eig-main} for \corA{$\al >0$ and Theorem} \ref{thm:hom-main}}\label{sec:pf-thm}



\corA{We begin with the proof of the large deviation lower bound for ${\rm UT}_t$. The proof follows upon computing the lower bound of the probability of the existence of a clique of an appropriate size or that of a vertex with sufficiently high degree in $\G(n,p)$.}

\begin{proof}[Proof of Theorem \ref{thm:hom-main} (lower bound)]
Fix $\beta >0$. For graphs $\Graph, \Graph', \Graph''$ such that $\Graph' \subset \Graph$ and $\Graph'' \subset \Graph\setminus \Graph'$ we observe that
\[
{\rm Hom}(C_{2t}, \Graph) \ge {\rm Hom}(C_{2t}, \Graph') + {\rm Hom}(C_{2t}, \Graph'').
\]
Therefore, denoting $\G_1$ to be the random subgraph of $\G(n,p)$ induced by the edges between $\{1\}$ and $[n]\setminus \{1\}$, $\check{\G}_1$ to be random subgraph on $[n]\setminus\{1\}$, $\G_2$ to be the random subgraph induced by $[\lceil (\deltx+\beta)^{1/(2t)}\corA{d}\rceil]$, and $\check{\G}_2$ to be the random subgraph induced by the rest of the vertices, we obtain that
\[
{\rm Hom}(C_{2t}, \G(n,p)) \ge {\rm Hom}(C_{2t}, \G_1) + {\rm Hom}(C_{2t}, \check{\G}_1)
\]
and 
\[
{\rm Hom}(C_{2t}, \G(n,p)) \ge {\rm Hom}(C_{2t}, \G_2) + {\rm Hom}(C_{2t}, \check{\G}_2).
\]
It is easy to note that 
\begin{multline}\label{eq:Utdel1}
\P({\rm Hom}(C_{2t}, \G_1) \ge (\deltx+\beta)\corA{d}^{2t}) \ge \P\left(\deg_{\G(n,p)}(1) \ge  (({\deltx+\beta})/{2})^{1/t} \corA{d}^2\right) \\
\ge \exp\left(- (1+o(1)) \cdot \left( \f{\deltx+\beta}{2}\right)^{1/t} \corA{d}^2\log \corA{d} \right) =: \gq_1,
\end{multline}
where the last step is due to Lemma \ref{lem:ash} and we have used that $np^2 \ll 1 \ll \corA{d}$ . Moreover,
\begin{multline}\label{eq:Utdel2}
\gq_2:= \exp\left(-\f12 \lceil (\deltx+\beta)^{1/(2t)}\corA{d}\rceil^2 \log(1/p)\right) \le \P\left( K_{[\lceil (\deltx+\beta)^{1/(2t)}\corA{d}\rceil]} \subset \G(n,p)\right)\\
\le \P\left({\rm Hom}(C_{2t}, \G_2) \ge (1+o(1)) \cdot (\deltx+\beta)\corA{d}^{2t}\right). 
\end{multline}
Therefore, by the independence of the edges of $\G(n,p)$ 
we further deduce that
\beq\label{eq:Utdel3}
\P({\rm UT}_{t}(\deltx)) 
\ge \max\left\{ \gq_1 \cdot \P({\rm Hom}(C_{2t}, \check{\G}_1) \ge (1-\beta) \corA{d}^{2t}), \gq_2 \cdot \P({\rm Hom}(C_{2t}, \check{\G}_2) \ge (1-\beta/2)\corA{d}^{2t})\right\}.
\eeq
Since $\check{\G}_1$ and $\check{\G}_2$ are distributed as $\G(n', p)$ and $\G(n'', p)$ with $n', n'' = n(1-o(1))$ and $\beta$ is arbitrary we observe from \eqref{eq:Utdel1}-\eqref{eq:Utdel3} that, writing $A_n$ for the adjacency matrix of $\G(n,p)$, it suffices to show that 
\beq\label{eq:tr-conc}
\P(\Tr (A_n^{2t}) \le (1-\beta/2) \corA{d}^{2t}) =o(1).
\eeq
Turning to prove \eqref{eq:tr-conc} we apply standard concentration inequalities (e.g.~\cite[Theorem 1.16]{BR}) and the interlacing inequality, upon noting $\lambda_1(\E A_n) = \corA{d} (1+o(1))$, to obtain that, for any $\vep' >0$,
\[
\P(\lambda_1(A_n) \le (1-\vep') \corA{d}) \le \P(\|A_n- \E A_n\| \ge \vep' \corA{d}/2) = o(1).
\]
From this the claim \eqref{eq:tr-conc} follows and thus the proof of the lower bound is now complete.
\end{proof}

The next few sections are devoted to the proof of the large deviation upper bound for the upper tail of ${\rm Hom}(C_{2t}, \G(n,p))$. \corA{In the next section we borrow a few terminologies from \cite{BB, hms} to define certain subgraphs of $\G(n,p)$ and state some required properties of those subgraphs.}

\subsection{Pre-seed, seed, core, and strong core graphs}\label{sec:seed}
\corA{The first notion is about pre-seed graphs.} 
For $\Graph \subset K_n$ and a function $f$ we define
\[
\E_\Graph[f(\G(n,p)] := \E[f(\G(n,p))\mid a_{i,j}^\G=1 \, \forall (i,j) \in E(\Graph)],
\]
where $\{a_{i,j}^\G\}_{i,j=1}^n$ is the adjacency matrix of $\G(n,p)$. 

\begin{dfn}[Pre-seed graph]\label{dfn:seed-0}
Fix $\vep >0$ sufficiently small and an integer $t \ge 2$. Let $\bar C:=\bar C(\deltx, t)$ be a sufficiently large constant. A graph $\Graph \subset K_n$ is said to be a pre-seed graph if the followings hold:
\begin{enumerate}
\item[(PS1)] $\E_\Graph[{\rm Hom}(C_{2t}, \G(n,p))] \ge (1+\deltx(1 -\vep)) \E[{\rm Hom} (C_{2t}, \G(n,p)]$.
\item[(PS2)] $e(\Graph) \le \bar C \corA{d}^2 \log(1/p)$.
\end{enumerate}
\end{dfn}
The choice of $\bar C$ will be made precise during the course of the proof. In the lemma below we show that probability of the upper tail event can be approximately bounded by that of the existence of pre-seed subgraphs of $\G(n,p)$, \corA{thereby allowing us to exclude a certain subspace of the configuration space that do not contribute to the large deviation event.}

\begin{lem}\label{lem:2ps}
Let $\corA{d}$ satisfy \eqref{eq:p-hom-gen}. Then 
\[
\P({\rm UT}_t(\deltx)) \le (1+o(1)) \cdot \P(\exists \Graph \subset \G(n,p): \Graph \text{ is a pre-seed graph}). 
\]
\end{lem}

The proof of Lemma \ref{lem:2ps} relies on the following result from \cite{hms}. \corA{To state the result we need the following notation:} For a function $F(Y)$, where $Y=(Y_1, Y_2, \ldots, Y_N)$ is a random vector, and $I \subset [N]$, we use the shorthand $\E_I[F]$ to denote the conditional expectation of $F$ given $\{Y_i, i \in I \}$. 

\begin{lem}[{\cite[Lemma 3.7]{hms}}]\label{lem:markov}
Let $Y:=(Y_1,Y_2, \ldots, Y_N)$ be a random vector taking values in $\{0,1\}^N$ and $F:=F(Y)$ be a nonzero polynomial with nonnegative coefficient of degree at most $\corA{\wh{d}}$. Then for every $\ell \in \N$, $\deltx >0$, and $\vep \in (0,1)$ we have that
\[
\P\left(F \ge (1+\deltx) \E[ F] \text{ and } Y_I=0 \text{ for all } I \in \cI\right) \le \left(1 - \f{\vep \deltx}{1+\deltx}\right)^\ell,
\]
where
\[
\cI:= \left\{I \subset [N]: |I| \le \corA{\wh{d}} \ell \text{ and } \E_I[F] \ge (1+\deltx(1-\vep)) \E[F] \right\} \quad \text{ and } \quad Y_I:= \prod_{i \in I} Y_i. 
\]
\end{lem}

\corA{Lemma \ref{lem:markov} essentially says that the upper tail event of a function $F$, a `low' degree polynomial of Boolean variables with nonnegative coefficients, when viewed as a subset of the hypercube $\{0,1\}^N$, excluding  a set of small probability, can be covered by a union over a collection of sub cubes of `small' codimension such that the (conditional) average of $F$ on each of those sub cubes is large. In the context of the Erd\H{o}s-R\'enyi graph this translates to the existence of a pre-seed of subgraph of $\G(n,p)$.
The proof of Lemma \ref{lem:markov} follows from a bound on high moments of $F(Y) {\bf 1}(Y_I=0 \text{ for all } I \in \cI)$ and Markov's inequality.} 

\begin{proof}[Proof of Lemma \ref{lem:2ps}]
We set $F= {\rm Hom}(C_{2t}, \G(n,p))$ and $N= \binom{n}{2}$. Identifying the set $\{(i,j), i < j \in [n]\}$ with $[N]$ we apply Lemma \ref{lem:markov} with $\ell = \bar C t^{-1} \corA{d}^2 \log(1/p)$ to deduce that 
\beq\label{eq:no-seed}
\P\left({\rm UT}_{t} (\deltx)\cap \left\{\exists \Graph \subset \G(n,p): \Graph \text{ a pre-seed graph}\right\}^\complement\right) \le \left(\f12\right)^{\bar C t^{-1}\corA{d}^2 \log(1/p)} \ll \P({\rm UT}_{t} (\deltx)),
\eeq
where the last step follows upon choosing $\bar C$ choosing sufficiently large, and from the lower bound on the probability of ${\rm UT}_{t} (\deltx)$ (proved above). This immediately implies that
\[
\P({\rm UT}_{t} (\deltx))\le o(1) \cdot \P({\rm UT}_{t} (\deltx)) + \P\left({\rm UT}_{t} (\deltx) \cap \left\{\exists \Graph \subset \G(n,p): \Graph \text{ a pre-seed graph}\right\}\right),
\]
which in turn yields the desired upper bound on the probability of ${\rm UT}_{t} (\deltx)$. 
\end{proof}


The condition (PS1) in Definition \ref{dfn:seed-0} is difficult to work with. Below we will show that, in the regime \eqref{eq:p-hom-gen}, graphs satisfying (PS1) admit a nicer description \corA{which we define below}.

\begin{dfn}[Seed graph]\label{dfn:seed}
Let $\vep, \deltx, t$, and $\bar C$ be as in Definition \ref{dfn:seed-0}. A graph $\Graph \subset K_n$ is said to be a seed graph if the followings hold:
\begin{enumerate}
\item[(S1)] ${\rm Hom}(C_{2t}, \Graph) \ge \deltx(1 -2\vep) \corA{d}^{2t}$.
\item[(S2)] $e(\Graph) \le \bar C \corA{d}^2 \log(1/p)$.
\end{enumerate}
\end{dfn}

\corA{By Lemma \ref{lem:2ps} and the discussion above it follows that to upper bound the probability of ${\rm UT}_t$ we need the same for the existence of seed subgraph of $\G(n,p)$. The latter probability can be na\"ively bounded by bounding the cardinality of the number of seed graphs of a given size, using that each edge in $\G(n,p)$ appears independent with probability $p$, and then taking a union bound over the set of possible sizes of seed graphs. However, such a na\"ive approach do not give a tight upper bound. Instead, in a very broad sense, we first find a suitable `net' for the set of all seed graphs and then carry out the union bound over those net elements. Such nets will be obtained in stages. In the first stage we obtain the net of all seed graphs by simply deleting those edges that {\em do not account for many} homomorphism counts. This motivates the following definition.} 



\begin{dfn}[Core graph]\label{dfn:core-graph}
With $\vep, \deltx, t$, and $\bar C$ as in Definition \ref{dfn:seed-0} we define a graph $\Graph \subset K_n$ to be a core graph if
\begin{enumerate}
\item[(C1)] $ {\rm Hom}(C_{2t}, \Graph) \ge \deltx  (1-3 \vep) \corA{d}^{2t}$,
\item[(C2)] $e(\Graph) \le \bar C \corA{d}^2 \log(1/p)$,
\end{enumerate}
and
\begin{enumerate}[resume]
\item[(C3)] \(\min_{{\bf e} \in E(\Graph)} {\rm Hom}(C_{2t}, \Graph, {\bf e}) \ge \deltx \vep \corA{d}^{2t} /(\bar C \corA{d}^2 \log(1/p))\),
\end{enumerate}
where for an edge ${\bf e} \in E(\Graph)$ the notation ${\rm Hom}(C_{2t}, \Graph,{\bf e})$ denotes the number of homomorphisms of $C_{2t}$ in $\Graph$ that contain the edge ${\bf e}$.
\end{dfn}

\begin{lem}\label{lem:ps2c}
Let $\corA{d}$ satisfy \eqref{eq:p-hom-gen}. Then, for all large $n$,
\[
\left\{\exists \Graph \subset \G(n,p) : \Graph \text{ is pre-seed}\right\} \subset \left\{\exists \Graph \subset \G(n,p) : \Graph \text{ is core}\right\}.
\]
\end{lem}

The proof of Lemma \ref{lem:ps2c} is postponed to Section \ref{sec:ps2c}. \corA{It needs combinatorial arguments}. Equipped with Lemma \ref{lem:ps2c} we observe that it suffices to bound the probability of the existence of a core graph in $\G(n,p)$. \corA{Again a na\"ive union bound does not suffice. So} we split  the set of core graphs into two subsets: (i) core graphs with a large number of edges and (ii) core graphs with $O(\corA{d}^2)$ many edges. \corA{By counting the number of core graphs with a given a number of edges and a union bound (where a lower bound on the number of edges becomes handy)} we show below that the existence of the first set of graphs is unlikely at the large deviations scale. \corA{To treat the second set of graphs we extract yet another net for core graphs by keeping only those edges  that participate in even a larger number of homomorphism counts. These latter set of graphs will be termed as strong-core graphs.} 


\begin{dfn}[Strong-core graph]\label{dfn:strong-core}
Let $\vep, \deltx$, and $t$ be as in Definition \ref{dfn:seed-0}, and $\bar C_\star:= \bar C_\star(\deltx, t) < \infty$ be a large constant, depending on $\deltx$ and $t$, and the ratio \corA{$\log d/\log n$}.  We define a graph $\Graph \subset K_n$ to be a strong-core graph if
\begin{enumerate}
\item[(SC1)] $ {\rm Hom}(C_{2t}, \Graph) \ge \deltx  (1-{6}\vep) \corA{d}^{2t}$,
\item[(SC2)] $e(\Graph) \le \bar C_\star \corA{d}^2$,
\end{enumerate}
and
\begin{enumerate}[resume]
\item[(SC3)] \(\min_{{\bf e} \in E(\Graph)} {\rm Hom}(C_{2t}, \Graph, {\bf e}) \ge (\deltx\vep/ \bar C_\star) \cdot \corA{d}^{2t-2}\).
\end{enumerate}
\end{dfn}
The choice of $\bar C_\star$ will also be made precise in the proof. Note the difference in the lower bounds in (C3) and (SC3) in Definitions \ref{dfn:core-graph} and \ref{dfn:strong-core}, respectively. 

\begin{lem}\label{lem:core-too-many}
Fix $\vep \in (0,1)$ and $t \in \N$. Then, for $\corA{d}$ satisfying \eqref{eq:p-hom-gen} and $\vep \le 1/(36t)$ we have that 
\[
\limsup_{n \to \infty} \f{\log \P\left(\exists \Graph \subset \G(n,p): \Graph \text{ is core and } e(\Graph) \ge \bar C_\star \corA{d}^2\right)}{\corA{d}^2 \log \corA{d}} \le - {\f{\bar C_\star}{20}}. 
\]
\end{lem}

\begin{lem}\label{lem:strong-core}
{Consider the same setup as in Lemma \ref{lem:core-too-many}. Then, there exists some absolute constant $\beta < \infty$ such that
\[
\log \P\left(\exists \Graph \subset \G(n,p): \Graph \text{ is strong-core}\right) \le \exp(- \phi_t(\deltx) (1- \beta t \vep)),
\]
for all large $n$, where
\beq\label{eq:phi-deltx}
\phi_t(\deltx):= \min\left\{\f12\deltx^{1/t}  \log(1/p), \left(\f\deltx2\right)^{1/t} \log \corA{d} \right\}\corA{d}^2.
\eeq
} 
\end{lem}

Proofs of Lemmas \ref{lem:core-too-many} and \ref{lem:strong-core} are postponed to Sections \ref{sec:core-too-many} and \ref{sec:strong-core}, respectively. 
Equipped with Lemmas \ref{lem:2ps}, \ref{lem:ps2c}, \ref{lem:core-too-many}, and \ref{lem:strong-core} we now complete the proof of the large deviation upper bound. 

\begin{proof}[Proof of Theorem \ref{thm:hom-main} (upper bound)]
Using Lemmas \ref{lem:2ps} and \ref{lem:ps2c} we derive that
\[
\P({\rm UT}_t(\deltx)) \le (1+o(1)) \P(\exists \Graph \subset \G(n,p): \Graph \text{ is core}).
\]
Therefore, once we show that 
\beq\label{eq:c2sc}
\{\exists \Graph \subset \G(n,p): \Graph \text{ is core and } e(\Graph) \le \bar C_\star \corA{d}^2\} \subset \{\exists \Graph \subset \G(n,p): \Graph \text{ strong-core} \},
\eeq
application of Lemmas \ref{lem:core-too-many} and \ref{lem:strong-core}, with $\bar C_\star = C \deltx^{1/t}$ for some large absolute constant $C< \infty$, would yield that
\beq\label{eq:UT-t-bd}
\P({\rm UT}_t(\deltx)) \le \exp(-\phi_t(\deltx) (1-2\beta t \vep)),
\eeq
for all large $n$. To prove \eqref{eq:c2sc} we consider the subgraph $\Graph'\subset \Graph$ obtained by iteratively deleting edges ${\bf e}$ of $\Graph$ such that ${\rm Hom}(C_{2t}, \Graph, {\bf e}) \le ( \deltx \vep/ \bar C_\star) \cdot \corA{d}^{t-2}$. Using triangle inequality it follows that $\Graph'$ is indeed a strong-core graph. 
Thus \eqref{eq:c2sc} holds and the proof of the theorem is complete. 
\end{proof}

\begin{rmk}\label{rmk:C-star}
For the rest of the paper we set $\bar C_\star = C \deltx^{1/t}$. 
\end{rmk}

Next, using Theorem \ref{thm:hom-main}, we prove Theorem \ref{thm:eig-main} \corA{for $\al >0$}. 

\begin{proof}[Proof of Theorem \ref{thm:eig-main} \corA{for $\al >0$}]
\corA{Recall that it suffices to prove \eqref{eq:eig-main1} under the assumption \eqref{eq:p-ass-nsparse11}.} 
%
First let us prove the large deviation lower bound. For $m \in \N$, let $K_{[m]}$ denote the clique on $[m]$. Since $\lambda (K_{[m]})= m-1$, by Lemma \ref{lem:graph-eig}(i) and (vi) we have that
\[
\left\{K_{[\lceil (1+\delta) \corA{d} \rceil+1]} \subset \G(n,p)\right\} \cup {\rm UT}_\Delta(\delta) \subset {\rm UT}_{\lambda}( \delta).
\]
Now the desired large deviation lower bound follows from Lemma \ref{lem:ash} and the fact that the probability that $\G(n,p)$ contains $K_{[m]}$ is bounded below by $\exp(-m^2/2 \log(1/p))$. 

Turning to prove the upper bound we note that for any graph $\Graph$
\beq\label{eq:lam2hom}
\lambda_1(\Graph)^{2t} \le \sum_i \lambda_i(\Graph)^{2t} = {\rm Hom}(C_{2t}, \Graph).
\eeq
Now, as $\log \corA{d} \gtrsim \log n$, there exists $c>0$ such that $\corA{d} \ge n^c$, for all large $n$. Hence, \eqref{eq:p-ass-nsparse2} holds for all $t >1/c$. This, together with \eqref{eq:lam2hom} and \eqref{eq:UT-t-bd}  now shows that
\beq\label{eq:hom2lam}
-\log \P({\rm UT}_{\lambda}( \delta))  \ge \phi_t(\deltx)(1-2t \beta \vep)
\eeq
for all $t > 1/c$ with $\deltx =\deltx(t) = (1+\delta)^{2t} -1$. As $\lim_{t \to \infty} \deltx^{1/t} = (1+\delta)^2$, given any $\upepsilon >0$, there exists some $t_0 > 1/c$ such that $(\deltx/2)^{1/t} \ge (1+\delta)^2(1-\upepsilon)$. Setting $\vep = \upepsilon t_0^{-1}\beta^{-1}$ we therefore find that 
\[
-\log \P({\rm UT}_{\lambda}( \delta)) \ge (1-2 \upepsilon) \min \left\{\f12 (1+\delta)^2 \corA{d}^2 \log(1/p), (1+\delta)^2 \corA{d}^2 \log \corA{d} \right\},
\]
for all large $n$. Since $\upepsilon$ is arbitrary the desired upper bound now follows. This completes the proof. 
\end{proof}

\subsection{Bounds on expected homomorphism counts}\label{sec:ps2c}In this section our goal is to derive Lemma \ref{lem:ps2c}. To prove Lemma \ref{lem:ps2c} we will need a couple of results. \corA{Comparing Definitions \ref{dfn:seed-0} and \ref{dfn:seed} we find that we need an asymptotic estimate on $\E[{\rm Hom}(C_{2t}, \G(n,p))]$. The next lemma provides that necessary bound.} 


To state the result we need the notion of quotient graphs. For $\pi$ a partition of $V(H)$ we let $H/\pi$ to be the {\em quotient graph} obtained from $H$ by identifying vertices within parts of the partition $\pi$ and deleting multiple edges (but keeping self-loops). Note that if some part of $\pi$ contain vertices that do not form an independent set of $H$ then the quotient graph $H/\pi$ possesses self-loops. By an abuse of notation, quotient graphs that do not possess any self-loops will be termed, for convenience, {\em simple quotient graphs}. 
\begin{lem}\label{lem:hom-exp}
Fix $t \ge 2$.  
\begin{enumerate}
\item[(a)] Let $H=C_{2t}/\pi$ be a simple quotient subgraph of $C_{2t}$. If $H$ is a tree then $v(H) \le t+1$. Furthermore, the number of simple quotient subgraphs of $C_{2t}$ that are trees and have (vertex) size $(t+1)$ is $\binom{2t}{t}/(t+1)$. 

\item[(b)] Let $p \in (0,1)$ such that $\corA{d} \gg 1$. Then
\[
\E[{\rm Hom}(C_{2t}, \G(n,p))] = (1+o(1)) \cdot \left(\corA{d}^{2t} + \f{\binom{2t}{t}}{t+1}n \corA{d}^{t}\right).
\]
\end{enumerate}
\end{lem}

\corA{The proof of Lemma \ref{lem:hom-exp} uses a counting argument that has some similarity with the counting argument used in the proof of Wigner's semicircle law. To prove Lemma \ref{lem:ps2c} we also need to show that $\E_\Graph[ {\rm Hom}(C_{2t}, \G(n,p))] = (1+o(1)){\rm Hom}(C_{2t}, \Graph)$ (recall Definitions \ref{dfn:seed-0} and \ref{dfn:seed} again).}


\corA{Using \eqref{eq:hom-dfn} one can write $\E_\Graph[ {\rm Hom}(C_{2t}, \G(n,p))]$ as a sum over subgraphs ${\sf H} \subset C_{2t}$ such that $E({\sf H})$ is mapped to $E(\Graph)$ and the rest of the edges of $C_{2t}$ are mapped to $K_n\setminus \Graph$ via maps $\varphi: V(C_{2t}) \mapsto [n]$. If $\varphi$ were an injective map then while computing $\E_\Graph(\cdot)$ {\em all} the edges of $C_{2t}$ that are not mapped to those of $\Graph$ retain their independence. Thus in the case of subgraph counts this step was almost a triviality.} 



\corA{However, a homomorphism $\varphi$ not necessarily being an injective map the edges of $C_{2t}$ that are not mapped to $E(\Graph)$ need not be mapped to distinct edges of $K_n\setminus  \Graph$, and hence all these edges may not be ``free'' and some may get ``frozen'' via $\varphi$. To tackle this new difficulty one needs an additional combinatorial analysis. We require a few definitions to carry out this step.} 

\begin{dfn}[Equivalence classes of maps, and generating and non-generating edges and vertices]\label{dfn:frozen}
Let $\varphi: V(C_{2t}) \mapsto [n]$ be a map such that no two adjacent vertices of $C_{2t}$ are mapped to the same element in $[n]$. It naturally induces a map from $E(C_{2t})$ to $\binom{[n]}{2}$, which, by a slight abuse of notation, will be continued to be denoted by $\varphi$. 

For $\Graph \subset K_n$ and $H \subset C_{2t}$ we let $\cS({H}, {\sf G})$ to be the set of all maps $\varphi$ such that $\varphi(E({H})) \subset E(\Graph)$ and $\varphi(E(\bar{H})) \subset E(\bar{\Graph})$, where $\bar{H}:= C_{2t} \setminus {H}$ and $\bar \Graph:= K_n \setminus \Graph$. 

Notice that $\varphi \in \cS({H}, {\sf G})$ further induces equivalence classes on $E(\bar{H})$. Namely, 
$e_i, e_j \in E(\bar{H})$ are said to be equivalent iff $\varphi(e_i)=\varphi(e_j)$. Inside each equivalence class an arbitrarily chosen edge, e.g.~the edge with the smallest index under the canonical labelling of $E(C_{2t})$ (see Definition \ref{dfn:e-type} below), is said to be {\em generating} (or {\em free}), while the others are said to be {\em non-generating} (or {\em frozen}).  

This also allows us to extend the notion of free/frozen for vertices that are incident to some edge in $E(\bar{H})$:~Both end points of a non-generating edge are set to be non-generating (or frozen), while the rest of the vertices are termed to be generating (or free). 
\end{dfn}

\begin{dfn}[Types of edges]\label{dfn:e-type}
For convenience, let us put the canonical labelling on $C_{2t}$. That is,  $V(C_{2t})= [2t]$ and $E(C_{2t})= \{e_i\}_{i=1}^{2t}$, where $e_i=(i, i+1)$ for  $i \in [2t-1]$ and $e_{2t}=(2t,1)$. For $i \in [2t]$, the vertex $i$ will be said to be the left end point of the edge $e_i$, while the vertex in the other end of $e_i$ will be its right end point. 

Fix $H \subset C_{2t}$. 
We classify the edges of $\bar{H}$ as follows. If both end points of the edge ${\bf e} \in E(\bar{H})$ are in $V({H})$ then we classify ${\bf e}$ as a {\em type I edge}. If ${\bf e}$ is such that only its left end point belongs to $V({H})$ then we classify it as a {\em type II edge}. When only the right end point of ${\bf e}$ belongs to $V({\sf H})$ we term that to be a {\em type III edge}.  The rest of the edges of $\bar H$ are {\em type IV} edges. 
\end{dfn}

\begin{dfn}\label{dfn:psi-H}
Fix $\Graph \subset K_n$ and $H \subset C_{2t}$. Define
\[
\psi(H)= \psi_\Graph(H):= \sum_{\varphi \in \cS(H, \Graph)} \E_\Graph \left[\prod_{{\bf e}= (x,y) \in E(C_{2t})} a^\G_{\varphi(x), \varphi(y)} \right].
\]
Let ${\bm f}:= (\gf_1, \gf_2, \gf_3, \gf_4)$, where $\gf_i \in \Z_\ge$ for $i=1,\ldots, 4$, and define $\cS(H, \Graph, {\bm f})$ to be the subset of all $\varphi \in \cS(H, \Graph)$ such that the number of frozen edges of Type $i$ equals $\gf_i$ for $i=1,\ldots, 4$. Set 
\beq\label{eq:psi-Hf}
\psi(H, {\bm f})= \psi_\Graph(H, {\bm f}):= \sum_{\varphi \in \cS(H, \Graph, {\bm f})} \E_\Graph \left[\prod_{{\bf e}= (x,y) \in E(C_{2t})} a^\G_{\varphi(x), \varphi(y)} \right].
\eeq
\end{dfn}

Equipped with the above definitions we have the following result.
\begin{lem}\label{lem:ps-neg}
Fix $\emptyset \ne H \subsetneq C_{2t}$, $\Graph \subset K_n$, and ${\bm f} \in \Z_\ge^4$ such that $H$ does not contain any isolated vertex and $e(\Graph) = O(\corA{d}^2 \log n)$. Then,  for $p \in (0,1)$ such that $(\log n)^{t}  \ll \corA{d} \ll n^{1/2}(\log n)^{-2t}$  we have 
\beq\label{eq:ps-neg}
\psi(H, {\bm f}) = o(\corA{d}^{2t}).
\eeq
\end{lem}

Let us postpone the proof of Lemma \ref{lem:ps-neg} to later and now provide the proof of Lemma \ref{lem:ps2c}.

\begin{proof}[Proof of Lemma \ref{lem:ps2c} (using Lemmas \ref{lem:hom-exp} and \ref{lem:ps-neg})]
\corA{As $\max_{i=1}^4 \gf_i \lesssim t$}, it follows from Definition \ref{dfn:psi-H} and Lemma \ref{lem:ps-neg} that 
\[
\E_\Graph[{\rm Hom}(C_{2t}, \G(n,p)] = \sum_{H: H \subset C_{2t}} \psi_\Graph(H) = \psi_\Graph(C_{2t}) + (1+o(1)) \corA{d}^{2t} = {\rm Hom}(C_{2t}, \Graph) + (1+o(1)) \corA{d}^{2t},
\]
where the sum over $H$ in the first step is restricted to those subgraphs of $C_{2t}$ that do not possess any isolated vertex. Therefore, by Lemma \ref{lem:hom-exp}(b) we deduce that any pre-seed graph $\Graph$ is a seed graph,  for $\corA{d}$ as in \eqref{eq:p-hom-gen} and all large $n$. Now upon deleting edges ${\bf e} \in E(\Graph)$ (iteratively) such that ${\rm Hom}(C_{2t}, \Graph, {\bf e}) \le \deltx \vep \corA{d}^{2t}/ (\bar C\corA{d}^2 \log(1/p))$ we obtain a subgraph $\Graph' \subset \Graph$ such that conditions (C2) and (C3) of Definition \ref{dfn:core-graph} holds for $\Graph'$. Since $e(\Graph) \le \bar C \corA{d}^2 \log(1/p)$, by triangle inequality we obtain that (C1) also holds for $\Graph'$. Thus, $\Graph'$ is a core graph.
This concludes the proof. 
\end{proof}

It remains to prove Lemmas \ref{lem:hom-exp} and \ref{lem:ps-neg}. First, we prove Lemma \ref{lem:ps-neg}. \corA{It uses a counting argument.}

\begin{proof}[Proof of Lemma \ref{lem:ps-neg}]
To derive \eqref{eq:ps-neg} we will split the sum over $\varphi$ into subsets of $\cS(H, \Graph, {\bm f})$ such that the equivalence classes, determined by $\varphi$, are same and show that the bound \eqref{eq:ps-neg} holds for all allowable choices of equivalence classes. 

To this end, we fix a partition $\cP$ of the edge set $E(\bar H)$. This automatically defines equivalence classes on $E(\bar H)$ and therefore it further determines the set of non-generating/frozen edges, to be denoted by $\cF= \cF(\cP) \subset E(\bar H)$. There is still an indeterminacy in the choice of which end of generating edge to be mapped to the left end of non-generating edges. So we fix ${\bm m} \in \{-1,1\}^\cF$. We let $\cS(H, \Graph, \cP, {\bm m}) \subset \cS(H, \Graph, {\bm f})$ to be the collection of all maps $\varphi$ such that the equivalence class induced by $\varphi$ is $\cP$, in particular the set of non-generating edges is $\cF$, and the end points of the free edges to be mapped (via $\varphi$) to the left end of the frozen edges  are determined by ${\bm m}$. 

We then let $\psi(H, \cP, {\bm m})$ to be the \abbr{RHS} of \eqref{eq:psi-Hf} when the sum there is taken over $\varphi \in \cS(H, \Graph, \cP, {\bm m})$. As there are only finitely many choices of ${\bm m}$ and $\cP$ (depending only on $t$) such that the number of frozen edges is ${\bm f}$ we notice that it suffices to prove the bound \eqref{eq:ps-neg} for $\psi(H, \cP, {\bm m})$ for any fixed $\cP$ and ${\bm m}$. 

We now turn to prove \eqref{eq:ps-neg} for $\psi(H, \cP, {\bm m})$. Observe that conditioned on the edges in $\Graph$, for any $\varphi \in \cS(H, \Graph, \cP, {\bm m})$, only the edges $\varphi(E(\bar H) \setminus \cF)$ are independent. Therefore
{\allowdisplaybreaks
\beq\label{eq:ps-neg1}
\psi(H, \cP, {\bm m}) = \sum_{\varphi \in \cS(H, \Graph, \cP, {\bm m})}  \E \left[\prod_{{\bf e}= (x,y) \in E(\bar H)} a^\G_{\varphi(x), \varphi(y)} \right]  = \wt \psi(H, \cP, {\bm m}) \cdot p^{|E(\bar H) \setminus \cF|},
 \eeq
}where $\wt \psi(H, \cP, {\bm m}):= |\cS(H, \Graph, \cP, {\bm m})|$. As
\beq\label{eq:ps-neg2}
|E(\bar H) \setminus \cF| = 2t - e(H) - \sum_{i=1}^4 \gf_i,
\eeq
it now remains to find an appropriate bound on $\wt \psi(H, \cP, {\bm m})$. To this end, we let $V_\star:= V(H) \setminus  V(\cF)$ and $\widehat{V}:= [2t]\setminus (V(H) \cup V(\cF))$,  where $V(\cF)$ is the set of frozen vertices. Recalling Definition \ref{dfn:frozen}, as $\varphi \in \cS(H, \Graph, \cP, {\bm m})$, we observe that $\varphi(V_\star \cup\wh V)$ determines the values of $\varphi(V(\cF))$. Therefore
\[
\cS(H, \Graph, \cP, {\bm m}) \subset \left\{ \varphi_\star\Big| \, \varphi_\star: V_\star \mapsto V(\Graph),  \varphi_\star(E(H_\star)) \subset E(\Graph)\right\} \times \left\{\widehat \varphi\Big| \, \widehat \varphi: \widehat V \mapsto [n] \right\},
\]
where $H_\star:= \Graph[V_\star]$. Hence
\beq\label{eq:ps-neg31}
\wt \psi(H, \cP, {\bm m}) \le n^{\widehat{v}} \cdot {\rm Hom}(H_\star, \Graph),
\eeq
where $\wh{v}:= |\wh{V}|$. 

Since $\emptyset \ne H \subsetneq C_{2t}$ and $H$ does not contain isolated vertices it is easy to see that $H$ must be a vertex disjoint union of paths and hence, by the definition of $H_\star$, so is $H_\star$. Let us assume that $H= \cup_{i=1}^k P_{s_i}$ for some $k \ge 1$ and $s_i \ge 1$ for $i \in [k]$, where $P_\ell$ denotes a path of length $\ell$ \corA{and the union is a vertex disjoint union}. Recall Definition \ref{dfn:e-type} and note that each frozen edge appearing immediately after a path $P_{s_i}$, for some $i \in [k]$ (when $C_{2t}$ is traversed from the lowest indexed edge to the highest indexed edge in its canonical labelling) must be either a type I or type II edge, while each frozen edge immediately before a path $P_{s_i}$ should be either a type I and or type III edge. This yields that
\beq\label{eq:gf123}
\gf_1 + \gf_2 \vee \gf_3 \le k.
\eeq
\corA{In fact $\gf_1$ in the \abbr{LHS} above can be replaced by the total number of type I edges.}
We next claim that
\beq\label{eq:ps-neg3}
\widehat v \le 2t - v(H) - \gf_2 \vee \gf_3 - \gf_4
\eeq
and
\beq\label{eq:ps-neg4}
{\rm Hom}(H_\star, \Graph) \lesssim n^{v(H)} p^{e(H)} \cdot (np^2)^{k} \cdot \corA{d}^{-(\gf_1+\gf_2\vee \gf_3)} \cdot (\log n)^{3t}.
\eeq
Before proving \eqref{eq:ps-neg3}-\eqref{eq:ps-neg4} let us use these bounds to derive the bound \eqref{eq:ps-neg} for $\psi(H, \cP, {\bm m})$. Using \eqref{eq:ps-neg1}-\eqref{eq:ps-neg31} and \eqref{eq:ps-neg3}-\eqref{eq:ps-neg4} we obtain that
\[
\psi(H, \cP, {\bm m}) \lesssim \corA{d}^{2t} \cdot (\log n)^{3t}  \cdot (np^2)^{k - \gf_1 - \gf_2 \wedge \gf_3} \cdot \corA{d}^{-( \gf_2 \vee\gf_3 + \gf_4)}. 
\]
%
\corA{If $\gf_2 \vee \gf_3 \vee \gf_4 >0$ then the lower bound on $\corA{d}$ yields the desired bound on $\psi(H, \cP, {\bm m})$. Consider remaining case, i.e.~when $\gf_2 \vee \gf_3 \vee \gf_4=0$. As there are only Type I edges, at least one of them must a generating edge. Thus, by \eqref{eq:gf123} we find $\gf_1 < k$ and therefore the upper bound $d$ yields the desired bound.} 


To complete the proof of the lemma it remains to derive \eqref{eq:ps-neg3}-\eqref{eq:ps-neg4}. To prove \eqref{eq:ps-neg3} pick any frozen edge that is of type either II or IV. Consider the vertex on its right end point. Observe that such vertices are distinct as we traverse over all type II and IV frozen edges. Furthermore, all these vertices are in $([2t]\setminus V(H)) \cap V(\cF)$. Therefore, $\wh v \le 2t - v(H) - \gf_2 - \gf_4$. Now considering type III and IV frozen edges and the vertices at the left end point of these edges, and repeating the same argument as above we arrive at the bound \eqref{eq:ps-neg3}. 

Turning to prove \eqref{eq:ps-neg4} we observe that, for $\ell \ge -1$,
\beq\label{eq:hom-folk}
{\rm Hom}(P_\ell, \Graph) \le (2 e(\Graph))^{\lceil \f{\ell+1}{2}\rceil} \qquad \text{ and } \qquad {\rm Hom}({\sf H}_1 \cup {\sf H}_2, \Graph) \le {\rm Hom}({\sf H}_1, \Graph) \cdot  {\rm Hom}({\sf H}_2, \Graph),
\eeq
where for notational convenience we set $P_0$ to be the graph containing a single isolated vertex and $P_{-1}$ to be the empty graph, and ${\sf H}_1$ and ${\sf H}_2$ vertex disjoint. We then recall the definitions of $V_\star$ and $H_\star$ from above to find that $H_\star= \cup_{i=1}^k P_{\wt s_i}$ for some sequence $\{\wt s_i\}_{i \in [k]}$ such that $-1 \le \wt s_i \le s_i$ for all $i \in [k]$. We claim that
\beq\label{eq:H-H-star}
\sum_{i=1}^k \wt s_i \le \sum_{i=1}^k s_i - \gf_1 - \gf_2 \vee \gf_3 = e(H) - \gf_1 - \gf_2 \vee \gf_3 = v(H) - k - \gf_1 - \gf_2 \vee \gf_3.
\eeq
The equalities in \eqref{eq:H-H-star} follow from the fact that $H$ is a vertex disjoint union of paths $\{P_{s_i}\}_{i=1}^k$ with $s_i \ge 1$. To prove the inequality in \eqref{eq:H-H-star} we consider type I or II frozen edges. Let $\cR$ be the vertices at the left end point of such edges. Clearly $\cR \subset V(H) \cap V(F)$, all elements of $\cR$ are distinct, and $|V(H) \cap V(F)| \ge |\cR| =\gf_1+\gf_2$. By a similar argument we also have that $|V(H)\cap V(F)| \ge \gf_1 + \gf_3$. Thus $v(H_\star) \le v(H) - \gf_1 -\gf_2 \vee \gf_3$. Then using that $v(P_\ell) = \ell+1$ for all $\ell \ge -1$ we obtain the inequality in \eqref{eq:H-H-star}. 

Equipped with \eqref{eq:H-H-star} we apply \eqref{eq:hom-folk} to deduce that
\begin{multline*}
{\rm Hom}(H_\star, \Graph) \lesssim e(\Graph)^{\sum_{i=1}^k (\f{\wt s_i}{2}+1)} \lesssim \corA{d}^{ \sum_{i=1}^k \wt s_i + 2k} \cdot (\log n)^{3t} \le \corA{d}^{v(H)+k - \gf_1 -\gf_2 \vee \gf_3} \cdot (\log n)^{3t}\\
= n^{v(H)} p^{e(H)} (np^2)^k \corA{d}^{-(\gf_1 + \gf_2 \vee \gf_3)} \cdot (\log n)^{3t},
\end{multline*}
where in the second step we used that $e(\Graph) = O(\corA{d}^2 \log n)$, $v(H) \le 2t$, and $k \le t$, and in the last step we have used $v(H)=e(H)+k$. This yields \eqref{eq:ps-neg4}. The proof of the lemma is now complete. 
\end{proof}

We end this section with the proof of Lemma \ref{lem:hom-exp}. To prove part (a) we will need to use some standard notion that are used in the proof of Wigner's semicircle law, such as words, and graphs associated to words. These are  borrowed from \cite[Chapter 2.1.3]{AGZ}. Readers familiar with these notions can skip the definition below and move straight to the proof of Lemma \ref{lem:hom-exp}. 

\begin{dfn}
Given a finite set $\mathscr{S}$, an $\mathscr{S}$-word $w$ is a finite sequence of letters, elements of $\mathscr{S}$, i.e.~$w=s_1s_2 \cdots s_k$ for some $k \in \N$. The length of $w$, $\ell(w)$, is defined to be $k$ and its weight, to be denoted by ${\rm wt}(w)$, is the number of distinct letters in ${\rm supp}(w):=\{s_i, i \in [k]\}$. The word $w$ is closed if $s_1 = s_k$. Two $\mathscr{S}$-words are equivalent if there is a bijection on $\mathscr{S}$ which maps one word to the other. If $\mathscr{S}$ is clear from the context, we refer to $w$ simply as a word.  

For a word $w$ as above we let $G(w)$ to be the graph associated with it whose vertex set is ${\rm supp}(w)$ and the edge set is $\{\{s_i, s_{i+1}\}, i \in [k]\}$. Note that for a closed word $w$ this defines a path starting and ending at the same vertex. For an edge ${\bf e} \in E(G(w))$ we define $\mathscr{N}_{\bf e}(w)$ to  be the number of times this path traverses the edge ${\bf e}$.  

We then let $\mathscr{V}_{k,r}$ to be a set of representatives for the equivalence classes of closed $[r]$-words $w$ with $\ell(w)=k+1$ and ${\rm wt}(w)=r$. Define $\mathscr{W}_{k,r} \subset \mathscr{V}_{k,r}$ be the collection of all $w$ such that $\mathscr{N}_{\bf e}(w) \ge 2$ for all ${\bf e} \in E(G(w))$. The set $\mathscr{W}_{k, k/2+1}$ ($k$ must be even) is said to be the set of all Wigner words.  
\end{dfn}

\begin{proof}[Proof of Lemma \ref{lem:hom-exp}]
Fix a partition $\pi$ so that the quotient subgraph $C_{2t}/\pi$ is simple. Our strategy would be to associate any such simple quotient graph to $G(w)$ for some appropriately chosen closed word $w$, and then use properties of $G(w)$. 

Towards this end, we observe that for any partition $\pi$ of $[2t]$ there is a natural choice of a closed word $w= w(\pi) =s_1s_2 \cdots s_{2t+1}$ such that ${\rm wt}(w)$ is the number of parts in the partition $\pi$, and $\{s_1, s_2, \ldots, s_i\}$ is a discrete interval in $\N$ for all $i \in [2t+1]$. Let us illustrate it through examples: If $\pi_1= \{\{1,3,5\}, \{2\}, \{4\}, \{6\}\}$ and $\pi_2=\{\{1,3\}, \{2,4\}, \{5\}, \{6\}\}$ then $w(\pi_1)=1213141$ and $w_2(\pi)=1212341$. It can be argued that that this map $\pi \mapsto w(\pi)$ is a bijection between the set of all $\pi$ such that its number of parts is $r$ \corA{and the set words $w \in \mathscr{V}_{2t, r}$ such that ${\rm supp}(w)=[r]$}. Furthermore, the graphs $G(w(\pi))$ (\corA{when one edge is kept for each set of parallel edges}) and $C_{2t}/\pi$ are isomorphic. 

Let $\pi$ be such that $\mathscr{N}_{\bf e}(w(\pi)) =1$ for some edge ${\bf e} \in G(w(\pi))$. We claim that $G(w(\pi))$ (or equivalently $C_{2t}/\pi$) must contain a cycle. To see this, for ease of writing assuming that ${\bf e}_0=(\llbracket 1 \rrbracket, \llbracket 2t \rrbracket)$ is traversed only once we observe that the path induced by the set of edges $\gW:=\{(\llbracket i \rrbracket, \llbracket i+1\rrbracket), i \in [2t-1]\}$, where $\llbracket j \rrbracket$ denotes the equivalence class containing the vertex $j$ induced by $\pi$, does not contain the edge ${\bf e}_0$. \corA{Since the vertices $\llbracket 1 \rrbracket$ and $\llbracket 2t \rrbracket$ are connected via the edges in $\gW$ we deduce that} there are two paths between $\llbracket 1 \rrbracket$ and $\llbracket 2t \rrbracket$ yielding the claim.   

On the other hand $\mathscr{N}_{\bf e}(w(\pi)) \ge 1$ for all edge ${\bf e}$, $e(G(w(\pi))) \ge {\rm wt}(w) -1$ and $\sum_{\bf e} \mathscr{N}_{\bf e}(w(\pi))= 2t$, where the lower bound on $e(G(w(\pi)))$ is due to the fact that $G(w(\pi))$ is connected. This shows that for any $r \ge t+2$ such that ${\rm wt} (w(\pi)) =r$ there exists some edge ${\bf e}_\star$ such that $\mathscr{N}_{{\bf e}_\star}(w(\pi)) = 1$. Hence, we conclude that any simple quotient graph $H=C_{2t}/\pi$ with $v(H) > t+1$ cannot be a tree.

It additionally follows from above that the set of $\pi$ for which that $C_{2t}/\pi$ is a tree on $(t+1)$ vertices has a bijection with $\mathscr{W}_{2t, t+1}$. So, the conclusion on the number of trees follows from the one-to-one correspondence between Wigner words and Dyck paths (cf.~\cite[Proof of Lemma 2.1.6]{AGZ}). This completes part (a).

The proof of (b) uses the following well known identity:
\beq\label{eq:hom-to-N}
{\rm Hom}({\sf H}, \Graph) = \sum_{\pi} N({\sf H}/\pi, \Graph),
\eeq
where the sum is over all partition $\pi$ of the vertex set $V({\sf H})$ (cf.~\cite[Chapter 5]{L}), and we recall that $N({\sf H}, {\sf G})$ denotes the number of labelled copies of ${\sf H}$ in ${\sf G}$. Note that if $\Graph$ does not possess any self-loop then the sum in \eqref{eq:hom-to-N} can be restricted to a sum over $\pi$ such that the graph ${\sf H}/\pi$ is a simple quotient graph. By part (a), for a simple quotient graph $H= C_{2t}/\pi$ either $e(H) \ge v(H)$ or $e(H)=v(H)-1$ only if $v(H) \le t+1$. Since for any $H$ one has that $\E N(H, \G(n,p))= n^{v(H)} p^{e(H)}(1+o(1))$, denoting $\mathscr{T}$ to be the set of all trees on $(t+1)$ vertices and $\upphi$ to be the trivial partition (i.e.~consisting of only singletons), and using that $np \gg 1$ we deduce that 
\[
\sum_{\pi: \pi \ne \upphi, C_{2t}/\pi \notin \mathscr{T}} \E[N(C_{2t}/\pi, \G(n,p)] = o(\corA{d}^{2t}) + o(n \corA{d}^t).
\]
Finally using the bound on the number of simple quotient graphs $C_{2t}/\pi \in \mathscr{T}$ and \eqref{eq:hom-to-N} the proof completes. 
\end{proof}

\subsection{Cores with many edges are unlikely}\label{sec:core-too-many}
In this section we prove Lemma \ref{lem:core-too-many}. \corA{This is based on a combinatorial argument that  requires the following result on the number of graphs with certain specified properties}. 
\begin{lem}\label{lem:core-bd}
Fix $\vep \in (0,1)$ and $t \in \N$. Let $\corA{d}$ be such that $(\log n)^{2t} \ll \corA{d} \lesssim n^{1/2}$, and $\gD:= \gD(\vep)$ and $D:=D(\vep):= \lceil 32/\vep\rceil$ be intergers such that $\gD \ge D$. Denote
\[
\cV_1:=\cV_1(\Graph):= \{v \in V(\Graph): \deg_\Graph(v) \le \gD \}
\] 
and $\ol{\cV}_1:= V(\Graph) \setminus \cV_1$. Let $\cN({\bm e}, {\bm v}, \gD)$ be the number of core graphs $\Graph$ with $e(\Graph)={\bm e}$ and $v({\cV}_1(\Graph))= {\bm v}$. Then, for all large $n$, we have
\[
\cN({\bm e}, {\bm v}, \gD) \le \binom{n}{{\bm v}} \cdot \exp(\vep {\bm e} \log(1/p)). 
\] 
\end{lem}

Proceeding similarly as in the proof of \cite[Lemma 4.7]{BB}, and using the lemma below instead of \cite[Lemma 3.6]{BB}, the proof of Lemma \ref{lem:core-bd} follows. Therefore, we spare the details. 

\begin{lem}\label{lem:prod-deg}
Let $\corA{d}$ satisfy \eqref{eq:p-hom-gen} and $\Graph$ be a core graph. Then, for every edge ${\bf e}=(u,v) \in E(\Graph)$
\[
\deg_\Graph(u) \deg_\Graph(v) \ge \f{\wt c_0(\vep, t) \corA{d}^2}{(\log n)^{2t}},
\]
for some constant $\wt c_0(\vep, t) >0$. 
\end{lem}
The proof of Lemma \ref{lem:prod-deg} is postponed to Appendix \ref{app:loc-hom}. We now proceed to prove Lemma \ref{lem:core-too-many}.

\begin{proof}[Proof of Lemma \ref{lem:core-too-many}]
Let $\cV_1(\Graph)$ be as in Lemma \ref{lem:core-bd} with $\gD=D$. Note that, by Lemma \ref{lem:prod-deg}, there are no edges $\Graph$ with both end points in $\cV_1$. Therefore, $e(\Graph) \ge e(\cV_1(\Graph), \ol{\cV}_1(\Graph)) \ge v(\cV_1(\Graph))$. 
We claim that 
\beq\label{eq:core-bd1}
\cN({\bm e}, {\bm v}, D) \le \left\{
\begin{array}{ll}
\exp\left( 2(1/3+\vep){\bm e}  \log (1/p)\right),& \mbox{ if } {\bm v} \le {\bm e}/3,\\
\exp\left( \left\{(1+2\vep)\log(1/p) - \log \corA{d}/3\right\}{\bm e}\right), & \mbox{ otherwise}.
\end{array}
\right.
\eeq
As the probability of any graph $\Graph$ with ${\bm e}$ edges is $p^{\bm e}$, and upon using that 
$t^{-1}  \log(1/p) \le \log \corA{d} \le  \log(1/p)$, for $\corA{d}$ satisfying \eqref{eq:p-hom-gen}, the proof of this lemma follows by taking a union over $({\bm e}_{11}, {\bm e}_{12}, {\bm v})$ such that ${\bm e} \ge \bar C_\star \corA{d}^2$, and ${\bm e}_{11}, {\bm e}_{12}$, and ${\bm v} \le n^2$. 

Thus, it remains to prove \eqref{eq:core-bd1}. Since $\corA{d} \ll n^{1/2}$ the upper bound in \eqref{eq:core-bd1} is immediate from Lemma \ref{lem:core-bd} when ${\bm v} \le {\bm e}/3$. Turning to prove the remaining case we apply Lemma \ref{lem:core-bd} again, together with Stirling's approximation, to obtain that
\begin{multline*}
\cN({\bm e}, {\bm v}, D) \le \left(\f{en}{{\bm v}}\right)^{\bm v} \cdot \exp(\vep {\bm e} \log(1/p)) \le \left(\f{3en}{\bar C_\star \corA{d}^2}\right)^{\bm v} \cdot \exp(\vep {\bm e} \log(1/p))\\
 \le \exp\left( \left\{(1+2\vep)\log(1/p) - \log \corA{d}/3\right\}{\bm e}\right),
\end{multline*}
where in the last two steps we have used that ${\bm v} \ge {\bm e}/3$, and in the penultimate and the final steps we have also used that ${\bm e} \ge \bar{C}_\star \corA{d}^2$ and $v(\cV_1(\Graph)) \le e(\Graph)$, respectively. This completes the proof. 
\end{proof}

\subsection{Probability upper bound on the existence of strong-core graphs}\label{sec:strong-core}
In this section we prove Lemma \ref{lem:strong-core}. This will require that for all but a few edges in a strong-core graph, the product of the degrees of two adjacent vertices satisfy a tight upper and lower bound \corA{in the leading order}. First we define the set of ``bad'' and ``good'' edges. 
\begin{dfn}\label{dfn:prod-deg-ubd}
Fix $C_0 < \infty$. Let $\Graph_{\rm high} \subset \Graph$ be the subgraph spanned by the edges ${\bf e}= (u,v) \in E(\Graph)$ for which
\beq\label{eq:prod-deg-ubd}
\deg_\Graph(u)  \deg_\Graph(v) \ge C_0 \corA{d}^2, 
\eeq
and set $\Graph_{\rm low}:= \Graph \setminus \Graph_{\rm high}$. 
\end{dfn}

The next lemma shows that a careful choice of $C_0$ in \eqref{eq:prod-deg-ubd} yields that ${\rm Hom}(C_{2t}, \Graph)$ \corA{is almost the same as} ${\rm Hom}(C_{2t}, \Graph_{\rm low})$, while keeping the number of edges in $\Graph_{\rm high}$ small. It also provides a lower bound on the product of the degree of two adjacent vertices. The proof is postponed to Appendix \ref{app:loc-hom}. 

\begin{lem}\label{lem:prod-deg-sc}
Consider the same setup as in Lemma \ref{lem:core-too-many}. Let $\Graph$ be a strong-core graph. There exist constants $0 < c_0(\vep,t), C_0(\vep,t) < \infty$ such that the followings hold:
\begin{enumerate}
\item[(a)] For every edge ${\bf e}=(u,v) \in E(\Graph)$
\[
\deg_\Graph(u) \deg_\Graph(v) \ge { c_0(\vep, t) \corA{d}^2}. 
\]
\item[(b)] 
Let $\Graph_{\rm high}= \Graph_{\rm high}(\vep, t)$ and $\Graph_{\rm low}= \Graph_{\rm low}(\vep, t)$ be as in Definition \ref{dfn:prod-deg-ubd} with $C_0=C_0(\vep, t)$ in \eqref{eq:prod-deg-ubd}. Then 
\[
{\rm Hom}(C_{2t}, \Graph_{\rm low}) \ge (1-\vep) {\rm Hom}(C_{2t}, \Graph) \qquad \text{ and } \qquad e(\Graph_{\rm high}) \le \vep e(\Graph). 
\]
\end{enumerate}
\end{lem}

We now prove Lemma \ref{lem:strong-core}.

\begin{proof}[Proof of Lemma \ref{lem:strong-core}]
\corA{In the first step we decompose a strong-core graph $\Graph$ into three subgraphs: one is a bipartite graph with edges only between vertices of low and high degrees (by Lemma \ref{lem:prod-deg-sc}), the second subgraph consisting of vertices of large degrees, and the third consisting of negligible many edges. Furthermore, almost all the homomorphisms of $C_{2t}$ into $\Graph$ are contained in the first two subgraphs. In the second step using that a large number of homomorphism implies a lower bound on the number of the edges of a graph (both for bipartite and non-bipartite graphs), the upper bound on the number of core graphs with a given number of edges, and a union bound we complete the proof.}

To carry out these steps we need to introduce several notation. 
Let $D=D(\vep)$ be as in Lemma \ref{lem:core-bd}, and $C_0=C_0(\vep, t)$ and $c_0=c_0(\vep, t)$ be as in Lemma \ref{lem:prod-deg-sc}. Set $D_0 := 0$,
\beq\label{eq:wt-C}
\wt C = \wt C(\vep, t):= (t-1) \left(\left\lceil \f{(2\bar C_\star)^t}{\vep \deltx}\right\rceil  +2\right), \qquad \text{ and }  \qquad D_i := D \cdot \left(\f{C_0}{c_0}\right)^{i-1}, \text{ for } i \in [\wt C].
\eeq
Define $\wt V_0 := \emptyset$. Then, for $i \in [\wt C]$ we (iteratively) define 
\[
V_i=V_i(\Graph):= \left\{ v \in V(\Graph_{\rm low}): D_{i-1}+1 \le \deg_\Graph(v) \le D_i\right\}
\]
and
\[
\wt V_i = \wt V_i(\Graph):= \left\{ v \in V(\Graph_{\rm low}): (u,v) \in E(\Graph_{\rm low}) \text{ for some } u \in V_i \right\}\setminus \wt V_{i-1}.
\]
We further define ${\sf H}_i$ to be the subgraph (of $\Graph_{\rm low}$) spanned by the edges that are incident to some vertex in $\cup_{j=1}^i V_j$, and $\bar{\sf H}_i = \Graph_{\rm low} \setminus {\sf H}_i$. Note that, by Lemma \ref{lem:prod-deg-sc}, $\{{\sf H}_i\}_{i=1}^{\wt C}$ are bipartite graphs. 

We now claim that there exists an $i \in [\wt C]$ such that 
\beq\label{eq:sc-claim1}
{\rm Hom}(C_{2t}, {\sf H}_i) + {\rm Hom}(C_{2t}, \bar{\sf H}_i) \ge \deltx(1-8 \vep) \corA{d}^{2t}. 
\eeq
To derive this claim we observe that there exists an $i \in [\wt C-t+1]$ such that
\beq\label{eq:sc-claim2}
{\rm Hom}(C_{2t}, {\sf H}_{i+(t-1)}) - {\rm Hom}(C_{2t}, {\sf H}_i) \le \vep \deltx \corA{d}^{2t}.
\eeq
Otherwise, as for a graph $\Graph$ and a bipartite graph $\wt \Graph$ one has that
\beq\label{eq:hom-e-lbd}
{\rm Hom}(C_{2t}, \Graph) \le (2e(\Graph))^t \qquad \text{ and } \qquad {\rm Hom}(C_{2t}, \wt \Graph) \le 2 (e(\wt \Graph))^t,
\eeq
it follows that
\begin{multline*}
{\rm Hom}(C_{2t}, \Graph) \ge \left[\sum_{k=1}^{\wt C/(t-1) -1} {\rm Hom}(C_{2t}, {\sf H}_{k(t-1)+1}) - {\rm Hom}(C_{2t}, {\sf H}_{(t-1)(k-1)+1})\right] + {\rm Hom}(C_{2t}, {\sf H}_1) \\
\ge \left(\f{\wt C}{t-1}-1\right) \vep \deltx \corA{d}^{2t}> (2 \bar C_\star)^t \corA{d}^{2t},
\end{multline*}
(recall the definition of $\wt C$) yielding a contradiction to the fact that $e(\Graph) \le \bar{C}_\star \corA{d}^2$. 

Recalling the definition of $\{D_k\}_{k=1}^{\wt C}$, and using the tight upper and lower bounds on the product of the degrees of adjacent vertices in $\Graph_{\rm low}$ given by Lemma \ref{lem:prod-deg-sc}, we further deduce that any edge in $\bar{\sf H}_k$ that is incident to some edge in ${\sf H}_k$ must be contained in ${\sf H}_{k+1}$. Therefore, as the distance between any two vertices of same parity in $C_{2t}$ is at most $(t-1)$, by induction we derive that any homomorphism of $C_{2t}$ that uses edges of both ${\sf H}_{i}$ and $\bar{\sf H}_{i}$ must be contained in ${\sf H}_{i+(t-1)}$, but not in ${\sf H}_{i}$. Hence, by \eqref{eq:sc-claim2}, we also obtain that the number of such homomorphisms is at most $\vep \deltx \corA{d}^{2t}$. As any homomorphism of $C_{2t}$ into $\Graph_{\rm low}$ must either be fully contained in ${\sf H}_{i}$ or $\bar{\sf H}_{i}$, or must use edges of both ${\sf H}_{i}$ and $\bar{\sf H}_{i}$, an application of Lemma \ref{lem:prod-deg-sc}(b) now yields the claim \eqref{eq:sc-claim1}.  
This concludes the first step. 

It now remains to find an upper bound on the probability of the existence of a strong-core subgraph $\G(n,p)$ such that \eqref{eq:sc-claim1} holds. This would be the second step of the proof. 
Turning to do this we introduce a few more notation. Let
\[
\cI_{i, {\bm e}}:= \left\{\Graph: \Graph \text{ strong-core}, e(\Graph)= {\bm e}, {\rm Hom}(C_{2t}, {\sf H}_i) + {\rm Hom}(C_{2t}, \bar{\sf H}_i) \ge \deltx (1- 8 \vep) \corA{d}^{2t} \right\},
\]
\[
\cI_{i,j, {\bm e}, {\bm e}_\sharp}:= \left\{  \Graph: \Graph \in \cI_{i, {\bm e}}, e({\sf H}_i) = {\bm e}_\sharp, (j-1) \vep \deltx \le \corA{d}^{-2t} {\rm Hom}(C_{2t}, {\sf H}_i) \le j \vep \deltx\right\}, \quad j \in [s_0],
\]
where $s_0:= \lfloor (1-10\vep)/\vep\rfloor$, and
\[
\cI_{i,s_0+1, {\bm e}, {\bm e}_\sharp}:= \left\{ \Graph: \Graph \in \cI_{i, {\bm e}}, e({\sf H}_i) = {\bm e}_\sharp, \corA{d}^{-2t} {\rm Hom}(C_{2t}, {\sf H}_i) \ge s_0 \vep \deltx\right\}.
\]
Denote
\[
\cJ_{i,j, {\bm e}, {\bm e}_\sharp}:=\left\{\exists \Graph \subset \G(n,p): \Graph \in \cI_{i,j, {\bm e}, {\bm e}_\sharp}\right\}, \quad j \in [s_0+1].
\]
Observe that it suffices to show the following bounds:
\beq\label{eq:pr-cI-0}
\P(\cup_{i, {\bm e}, {\bm e}_\sharp} \cJ_{i, 1, {\bm e}, {\bm e}_\sharp}) \le \exp\left( - \f12\deltx^{1/t}(1-\wt \beta \vep) \corA{d}^2 \log(1/p) \right),
\eeq
\beq\label{eq:pr-cI-1}
\P\left(\cup_{i, {\bm e}, {\bm e}_\sharp} \cup_{j=2}^{s_0} \cJ_{i,j, {\bm e}, {\bm e}_\sharp}\right)  \le \exp\left( - \phi_t(\deltx)(1-\wt \beta \vep)^{1/t}(1+\beta_\star \vep^{1/t})  +\wh \beta t \vep \phi_t(\deltx)\right),
\eeq
(recall \eqref{eq:phi-deltx}) and
\beq\label{eq:pr-cI-2}
\P(\cup_{i, {\bm e}, {\bm e}_\sharp} \cJ_{i, s_0+1, {\bm e}, {\bm e}_\sharp})  \le \exp\left( - \phi_t(\deltx)(1-\wt \beta t \vep) \right),
\eeq
where the unions in $i,j, {\bm e}$, and ${\bm e}_\sharp$ are taken over their allowable ranges,
and $\wt \beta, \wh \beta$, and $\beta_\star$ are some absolute constants. 

To prove \eqref{eq:pr-cI-0}-\eqref{eq:pr-cI-2} we let ${\bm v}_\sharp(\Graph):= |\cup_{j=1}^i V_j|$, and split into two cases: ${\bm v}_\sharp(\Graph) = {\bm v}_\sharp \le \vep e(\Graph) = {\bm e}$ and ${\bm v}_\sharp > \vep {\bm e}$. In the first case, applying Lemma \ref{lem:core-bd} with 
\[
\cV_1= \cup_{j=1}^i V_j \cup \left\{v \in V(\Graph_{\rm high}): \deg_\Graph(v) \le D_i\right\}
\]
and $\gD=D_i$ we find that the number of core graphs under consideration can be bounded by 
\[
\binom{n}{v} \left(\f1p\right)^{\vep {\bm e}} \le \exp\left(8\vep {\bm e} \log(1/p)\right), 
\]
where we also used $|\cV_1|={\bm v} \le {\bm v}_\sharp + 2 \vep e(\Graph)$ (follows from Lemma \ref{lem:prod-deg-sc}(b) and $\corA{d} \ll n^{1/2}$). Since, by \eqref{eq:hom-e-lbd}, for any strong-core graph $\Graph$ one has the bound $e(\Graph) \ge {\bm e}_0:=\f12 \deltx^{1/t}(1-8\vep)^{1/t} \corA{d}^2$ and the probability of a graph $\Graph$ with ${\bm e}$ edges is $p^{\bm e}$, using the union bound, we derive that 
\beq\label{eq:sc-pf-int-step0}
\P\left(\cup_{i, {\bm e}, {\bm v}, {\bm v}_\sharp \le \vep {\bm e}_\sharp} \left\{\exists \Graph \subset \G(n,p): \Graph \in \cI_{i,j, {\bm e}, {\bm e}_\sharp}, {\bm v}_\sharp(\Graph) = {\bm v}_\sharp\right\}\right) \le \exp\left(- \f12\deltx^{1/t} (1- \beta_0 \vep) \corA{d}^2 \log(1/p)\right),
\eeq
for some absolute constant $\beta_0 < \infty$ and all large $n$.

In the second case, using the lower bound ${\bm e} \ge {\bm e}_0$ we obtain that
\[
 \binom{n}{{\bm v}} \le \left(\f{4e}{\vep \deltx^{1/t}}\right)^{\bm v} \cdot \left(\f1p\right)^{{\bm v}} \cdot \left(\f{1}{\corA{d}}\right)^{{\bm v}_\sharp}.
\]
Therefore, applications of Lemmas \ref{lem:core-bd} and \ref{lem:prod-deg-sc}(b) now yield  that 
{\allowdisplaybreaks
\begin{multline}\label{eq:sc-pf-int-step}
\P\left(\exists \Graph \subset \G(n,p): \Graph \in \cI_{i,j, {\bm e}, {\bm e}_\sharp}, {\bm v}_\sharp(\Graph) = {\bm v}_\sharp > \vep {\bm e}\right)\\
 \le \exp\left( - \log(1/p) ({\bm e} - {\bm e}_\sharp) - \log(1/p)({\bm e}_\sharp - {\bm v}_\sharp) - \log \corA{d} \cdot  {\bm v}_\sharp+ 4 \vep {\bm e} \log(1/p)\right),
\end{multline}
}
for all large $n$. We now need an upper bound on the \abbr{RHS} of \eqref{eq:sc-pf-int-step}. To obtain such a bound we note that, if ${\rm Hom}(C_{2t}, \bar{\sf H}_i) \ge \eta_1\deltx \corA{d}^{2t}$ and ${\rm Hom}(C_{2t}, {\sf H}_i) \ge \eta_2\deltx \corA{d}^{2t}$ then by \eqref{eq:hom-e-lbd} and recalling that ${\sf H}_i$ is a bipartite graph one has that 
\beq\label{eq:hom-e-lbd1}
{\bm e} - {\bm e}_\sharp \ge \f12 \eta_1^{1/t} \deltx^{1/t} \corA{d}^2
\eeq
and
\beq\label{eq:hom-e-lbd2}
\log(1/p) ({\bm e}_\sharp  - {\bm v}_\sharp) + \corA{d} \cdot  {\bm v}_\sharp \ge \log \corA{d} \cdot {\bm e}_\sharp \ge \eta_2^{1/t} \left(\f{\deltx}{2}\right)^{1/t} \corA{d}^2 \log \corA{d}. 
\eeq
On the other hand, as $\corA{d}$ satisfies \eqref{eq:p-hom-gen} and ${\bm e} \le \bar C_\star \corA{d}^2$ (see also remark \ref{rmk:C-star}), we have
\beq\label{eq:err-bd}
{\bm e} \log(1/p) = O(\deltx^{1/t} \corA{d}^2 \log(1/p)) = O(t \phi_t(\deltx)).  
\eeq
Now, we apply \eqref{eq:hom-e-lbd1} with $\eta_1=1 -10 \vep$, and \eqref{eq:err-bd} to obtain a bound on the \abbr{RHS} of  \eqref{eq:sc-pf-int-step} for $j=1$. Taking a union over the allowable ranges of $i, {\bm v}, {\bm v}_\sharp, {\bm e}$, and ${\bm e}_\sharp$, and using the bound in \eqref{eq:sc-pf-int-step0} we then deduce \eqref{eq:pr-cI-0}. 

Next, we apply \eqref{eq:hom-e-lbd2} with $\eta_2=1 -11 \vep$ to obtain a bound on the \abbr{RHS} of  \eqref{eq:sc-pf-int-step} for $j=s_0+1$. Therefore, using \eqref{eq:err-bd} again and proceeding similarly as in the case of $j=1$ we derive \eqref{eq:pr-cI-2}.

To prove \eqref{eq:pr-cI-1} we fix $j_0 \in \{2, 3, \ldots, s_0\}$, and apply \eqref{eq:hom-e-lbd1} and \eqref{eq:hom-e-lbd2} with $\eta_1=1 -9 \vep - j_0 \vep$ and $\eta_2 = (j_0-1) \vep$, respectively. Since $1 \le j_0 -1 \le j_0  \le (1/\vep - 10)$, by the definition of $s_0$, we find that
\[
\eta_1^{1/t} + \eta_2^{1/t} \ge (1-10 \vep)^{1/t} \cdot \left(1+\f14 \vep^{1/t} \right). 
\]
Therefore,
\begin{multline}\label{eq:hom-e-lbd3}
 \log(1/p) ({\bm e} -{\bm e}_\sharp) + \log(1/p)({\bm e}_\sharp - {\bm v}_\sharp) + \log \corA{d} \cdot {\bm v}_\sharp \ge  (\eta_1^{1/t} + \eta_2^{1/t}) \phi_t(\deltx) \\
  \ge \phi_t(\deltx) \cdot (1- \beta_0 \vep)^{1/t} \left(1+\f {\vep^{1/t}}{4}\right). 
\end{multline}
Plugging this bound and \eqref{eq:err-bd} in \eqref{eq:sc-pf-int-step}, followed be a union over $i, j, {\bm v}, {\bm v}_\sharp, {\bm e}$, and ${\bm e}_\sharp$ over their respective allowable ranges, and using \eqref{eq:sc-pf-int-step0} we now establish \eqref{eq:pr-cI-1}. 
This completes the proof. 
\end{proof}

\section{Proof of Corollary \ref{cor:hom-main}}\label{sec:pf-cor}
%

We start with some auxiliary results that help us to identify and eliminate a few unlikely events at the large deviations scale. The first result shows that if the number of edges of a bipartite graph $\Graph$, with vertex bipartition $V(\Graph)=U_1 \cup U_2$, does not exceed much from the size of $U_1$ then almost all homomorphisms of $C_{2t}$ in $\Graph$ must be due to the stars centered at the vertices in $U_2$.

\begin{lem}\label{lem:hom-bd-bipartite}
Let $\Graph$ be a bipartite graph with vertex bipartition $U_1$ and $U_2$. Assume that $e(\Graph) - |U_1| \le \upalpha e(\Graph)$ 
for some $\upalpha >0$. Then 
\beq\label{eq:hom-bip0}
{\rm Hom}(C_{2t}, \Graph) \le 2 \upalpha t e(\Graph)^t + 2 \sum_{w \in U_2} \deg_\Graph(w)^t. 
\eeq
\end{lem}

\begin{proof}
Consider the canonical labelling of the vertices and edges of $C_{2t}$ (\corA{recall Definition \ref{dfn:e-type}}). Fix $u_1, u_2, \ldots, u_t \in U_1$ (not necessarily distinct) and let $\ul u := (u_1, u_2, \ldots, u_t)$. Denote ${\rm Hom}(C_{2t}, \Graph, \ul u)$ to be the cardinality of $\mathscr{H}_{\ul u}$, the set of homomorphisms $\varphi$ of $C_{2t}$ into $\Graph$ such that $\varphi(2i-1)=u_i$ for $i \in [t]$. Since $\Graph$ is bipartite it follows that
\beq\label{eq:hom-bip1}
{\rm Hom}(C_{2t}, \Graph) = 2 \sum_{\ul u \in U_1^t} {\rm Hom}(C_{2t}, \Graph, \ul u).
\eeq
We will show that
\beq\label{eq:hom-bip2}
{\rm Hom}(C_{2t}, \Graph, \ul u) \le \sum_{i=1}^t  (\deg_\Graph(u_i)-1) \cdot \left[\prod_{j \ne i} \deg_\Graph(u_j) \right]+ \sum_{w \in U_2} \prod_{i=1}^t a_{u_i, w}^\Graph,
\eeq
for $\ul u \in U_1^t$, where $\{a_{u,v}^\Graph\}$ denotes the adjacency matrix of $\Graph$. Sum the second term in the \abbr{RHS} of \eqref{eq:hom-bip2} over all $\ul u \in U_1^t$, and apply  \eqref{eq:hom-bip1} to obtain the second term in the \abbr{RHS} of \eqref{eq:hom-bip0}. On the other hand, summing the first term in the \abbr{RHS} of \eqref{eq:hom-bip2}, and using the assumption $e(\Graph) - |U_1| \le \upalpha e(\Graph)$  and \eqref{eq:hom-bip1} we obtain the first term in the \abbr{RHS} of \eqref{eq:hom-bip0}.

Turning to prove \eqref{eq:hom-bip2} we split $\mathscr{H}_{\ul u}$ into two further subsets $\wt {\mathscr{H}}_{\ul u}$ and $\wh{\mathscr{H}}_{\ul u}:= \mathscr{H}_{\ul u}\setminus \wt{\mathscr{H}}_{\ul u}$, where $\wt{\mathscr{H}}_{\ul u}$ is the set of homomorphisms $\varphi$ such that the cardinality of the set $\{\varphi(2i)\}_{i =1}^t$ is one. Thus, $\varphi \in \wt{\mathscr{H}}_{\ul u}$ implies that $\varphi(2i)=w$, for some $w \in U_2$, and $a_{u_i, w}^\Graph=1$ for all $i \in [t]$. This indeed shows that $|\wt{\mathscr{H}}_{\ul u}|$ is bounded by the second term in the \abbr{RHS} of \eqref{eq:hom-bip2}. 

To establish the bound on $\wh{\mathscr{H}}_{\ul u}$ we observe that $\wh{\mathscr{H}}_{\ul u} \subset \cup_{i=1}^t \wh{\mathscr{H}}_{{\ul u}, i}$, where $\wh{\mathscr{H}}_{{\ul u}, i}$ is the set of $\varphi \in {\mathscr{H}}_{\ul u}$ such that $\varphi(2i) \ne \varphi(2i-2)$ (the vertex $0$ is to be understood as the vertex $2t$). We claim that
\beq\label{eq:hom-bip3}
|\wh{\mathscr{H}}_{{\ul u},i}| \le (\deg_\Graph(u_i)-1) \cdot \left[\prod_{j \ne i} \deg_\Graph(u_j) \right].
\eeq

Let us prove this claim for $i=2$. The other cases are similar. The number of choices of $\varphi(2)$ is bounded by $\deg_\Graph(u_1)$. Pick one such choice and let $\varphi(2)=w$ for some $w \in U_2$. Now, by definition, for any $\varphi \in \wh{\mathscr{H}}_{{\ul u}, 2}$ we have that $\varphi(4) \in U_2 \setminus \{w\}$. So, the number of choices of $\varphi(4)$ is at most $\deg_\Graph(u_2)-1$. Continuing this argument we get the claim. Now summing \eqref{eq:hom-bip3} over $i \in [t]$ we obtain the first term in the \abbr{RHS} of \eqref{eq:hom-bip2}. This completes the proof. 
\end{proof}

The next result is a strengthening of some of the bounds in the proof of Lemma \ref{lem:strong-core}. To state this result we need a few more notation, and we will reuse some of the notation of proof of Lemma \ref{lem:strong-core}. For $\upalpha_1 >0$, we define
\[
\corA{\wh{\cI}_{i, {\bm e}, {\bm e}_\sharp, {\bm v}_\sharp}(\upalpha_1)} := \left\{ \Graph \subset K_n: \Graph \in \cI_{i,s_0+1, {\bm e}, {\bm e}_\sharp}, {\bm v}_\sharp(\Graph) = {\bm v}_\sharp,  e({\sf H}_i) - {\bm v}_\sharp(\Graph) \ge \upalpha_1  e({\sf H}_i)\right\},
\]
where $i, s_0, {\bm v}_\sharp, {\bm e}_\sharp$ etc are as in the proof of Lemma \ref{lem:strong-core}. Set
\[
\corA{\wh{\cJ}_{i, {\bm e}, {\bm e}_\sharp, {\bm v}_\sharp}(\upalpha_1):= \left\{\exists \Graph \subset \G(n,p): \Graph \in \wh{\cI}_{i, {\bm e}, {\bm e}_\sharp, {\bm v}_\sharp(\upalpha_1)}\right\}}.
\]
\begin{lem}\label{lem:hom-unlikely}
Let $\corA{d}$ satisfy \eqref{eq:p-hom-gen}. 
\begin{enumerate}

\item[(a)] Additionally assume that 
\beq\label{eq:p-wt-chi}
(1+\wt \chi) \log \corA{d} \le  \log(1/p),
\eeq
for some absolute constant $\wt \chi >0$. Fix $\upalpha_1 >0$. Then, for all large $n$,
\beq\label{eq:cJ-wh}
\P\left( \cup_{i, {\bm e}, {\bm e}_\sharp, {\bm v}_\sharp} \wh{\cJ}_{i, {\bm e}, {\bm e}_\sharp, {\bm v}_\sharp}\corA{(\upalpha_1)} \right) \le \exp\left( - (1-\wt \beta \vep)^{1/t}(1+\upalpha_1 \wt \chi) \phi_t(\deltx)+\wh \beta t \vep \phi_t(\deltx) \right),
\eeq
where the union over $i, {\bm e}, {\bm v}_\sharp$, and ${\bm e}_\sharp$ are over their respective allowable ranges. 
\item[(b)] Fix $\chi_\star >0$ and define ${\bm e}_\star:= {\bm e}_\star(\chi_\star, \deltx, t) = (1+ \chi_\star) \cdot (\deltx/2)^{1/t} \corA{d}^2$. Then, 
 for $\vep$ sufficiently small we have that
\[
\P\left(\cup_{i, {\bm e}} \cup_{{\bm e}_\sharp \ge {\bm e}_\star} \cJ_{i, s_0+1, {\bm e}, {\bm e}_\sharp} \right) \le \exp\left(-\left(1+\f{ \chi_\star}{2}\right) \phi_t(\deltx) + \wh \beta t \vep \phi_t(\deltx) \right). 
\]
\end{enumerate}
\end{lem}

\corA{Recall that in the proof of Lemma \ref{lem:strong-core} we showed that there exists a bipartite subgraph ${\sf H}$ of a strong-core graph $\Graph$ contained in $\G(n,p)$ such that almost all homomorphisms of $C_{2t}$ in $\Graph$ is either completeley contained in ${\sf H}$ or completely contained in $\Graph\setminus {\sf H}$. Lemma \ref{lem:hom-unlikely}(a) and (b) show that on the sub event that almost all homomorphisms of $C_{2t}$ of $\Graph$ are contained in ${\sf H}$, the probabilities of both the events that the number of excess edges and the total number of edges in ${\sf H}$ are large, are much smaller than the relevant upper tail large deviation probability.} 

\corA{The proof of Lemma \ref{lem:hom-unlikely} builds on some of the intermediate steps in the proof of Lemma \ref{lem:strong-core}.} 

\begin{proof}[Proof of Lemma \ref{lem:hom-unlikely}]
First let us prove part (a). 
%
Indeed, using \eqref{eq:p-wt-chi} and the lower bound $e({\sf H}_i)^t= {\bm e}_\sharp^t \ge (1-11\vep)\cdot (\deltx/2) \cdot \corA{d}^{2t}$ for any $\Graph \in \cI_{i,s_0+1, {\bm e}, {\bm e}_\sharp}$, we find that
\begin{multline*}
\log(1/p) ({\bm e}_\sharp  - {\bm v}_\sharp) + \log \corA{d} \cdot {\bm v}_\sharp \ge (1+\upalpha_1 \wt \chi)\log(np) {\bm e}_\sharp  
\ge \corA{(1-11\vep)^{1/t} \left(1  +  \upalpha_1 \wt \chi \right)} \left(\f{\deltx}{2}\right)^{1/t} \corA{d}^2 \log \corA{d}, 
\end{multline*}
for any $\Graph \in\wh{\cI}_{i, {\bm e}, {\bm e}_\sharp, {\bm v}_\sharp} (\corA{\upalpha_1})$. Thus, upon using \eqref{eq:sc-pf-int-step} and \eqref{eq:err-bd}, and applying union bounds \eqref{eq:cJ-wh} follows. 

Proof of (b) is also straightforward.  We recall \eqref{eq:sc-pf-int-step}. Since $\log \corA{d} \le \log(1/p)$, \corA{recalling the definition of $\phi_t(\deltx)$ and} using \eqref{eq:err-bd} we find that for any ${\bm e}_\sharp \ge {\bm e}_\star$ we have
\beq\label{eq:e-sharp-l1}
\P\left(\exists \Graph \subset \G(n,p): \Graph \in \cI_{i,s_0+1, {\bm e}, {\bm e}_\sharp}, {\bm v}_\sharp(\Graph)  > \vep {\bm e}\right)
 \le \exp\left( -  (1+\chi_\star) \phi_t(\deltx)+ \wh \beta \vep t \phi_t(\deltx) \right).
\eeq
Proceeding as in the proof of \eqref{eq:sc-pf-int-step0}, as $e(\Graph) \ge e({\sf H}_i) = {\bm e}_\sharp$, we also have that for any ${\bm e}_\sharp \ge {\bm e}_\star$
\beq\label{eq:e-sharp-l2}
\P\left(\exists \Graph \subset \G(n,p): \Graph \in \cI_{i,s_0+1, {\bm e}, {\bm e}_\sharp}, {\bm v}_\sharp(\Graph)  \le \vep {\bm e}\right)
 \le \exp\left( -  (1+ \chi_\star)(1-\beta_0 \vep) \phi_t(\deltx) \right),
\eeq
where $\beta_0$ is as in \eqref{eq:sc-pf-int-step0}. Now performing a union bound the desired bound follows from \eqref{eq:e-sharp-l1}-\eqref{eq:e-sharp-l2}. This completes the proof of the lemma. 
\end{proof}

During the proof of Corollary \ref{cor:hom-main}, upon using Lemmas \ref{lem:hom-bd-bipartite} and \ref{lem:hom-unlikely} we will be able to show that most of the excess homomorphism counts of $C_{2t}$ in $\G(n,p)$ must be due to the one neighborhood of vertices of degree $\gtrsim \corA{d}^2$. 
The next lemma essentially shows that on ${\rm UT}_\Delta(\deltx^{1/t}(1-\chi)^{1/t})^\complement$ the event described above is unlikely to happen. \corA{This will eventually lead us to conclude that on the upper tail event ${\rm UT}_t$ there must exist a single large degree vertex with probability approaching one.}

\begin{lem}\label{lem:deg-unlikely}
Fix $\xi, \varrho, \wt \chi >0$, and $t \ge 2$. Let $\cW \subset [n]$ be such that $|\cW|= {\bm w}=O(1)$. Define
\[
\cA=\cA_\cW(\xi):=\left\{\sum_{w \in \cW} \deg_{\G(n,p)}(w)^t \ge \xi \corA{d}^{2t}\right\} \, \text{ and } \, \cB= \cB_\cW(\varrho):= \left\{\min_{w \in \cW} \deg_{\G(n,p)}(w) \ge \varrho \corA{d}^2 \right\}.
\]
Then, there exist absolute constants $\beta$ and $c_0$ such that for $\wt \chi \le c_0$ and all large $n$, 
\beq\label{eq:deg-unlikely}
\P\left(\cA_\cW(\xi) \cap \cB_\cW(\varrho) \cap {\rm UT}_\Delta(\xi^{1/t}(1-2\wt\chi)^{1/t})^\complement\right) \lesssim (t {\bm w} \wt \chi^{-1})^{\bm w}  \exp\left( - \xi^{1/t} (1+ \beta \wt \chi^{1/t}) \corA{d}^2 \log \corA{d} \right). 
\eeq
\end{lem}

\corA{To prove Lemma \ref{lem:deg-unlikely} we first modify the definitions of $\cA$ and $\cB$ so that we can work with a collection of independent Binomial random variables. We discretize the allowable ranges of those independent random variables so that $\cA \cap \cB$ holds.  Then we apply Binomial tail bounds, and use the strict convexity of the associated rate function to derive \eqref{eq:deg-unlikely}.} 

\begin{proof}[Proof of Lemma \ref{lem:deg-unlikely}]
For $w \in \cW$ we define $X_w$ to be the degree of vertex $w$ in the subgraph of $\G(n,p)$ induced by the vertices \corA{$([n]\setminus \cW) \cup \{w\}$}. So $\{X_w\}_{w \in \cW}$ are independent and $X_w \stackrel{d}{=} \dBin(n', p)$, where $n'= (1-o(1))n$. Denoting
\[
\wt \cA = \wt \cA_\cW(\xi) := \left\{ \sum_{w \in \cW} X_w^t \ge \wt\xi \corA{d}^{2t} \right\}   \quad \text{ and } \quad \wt\cB= \wt\cB_\cW(\varrho):= \left\{\min_{w \in \cW} X_w  \ge \wt\varrho \corA{d}^2 \right\},
\]
where $\wt\xi:= \xi(1-o(1))$ and $\wt\varrho:= \varrho(1-o(1))$, we observe that
\beq\label{eq:cD}
\cA \cap \cB \cap {\rm UT}_\Delta(\xi^{1/t}(1-2\wt \chi)^{1/t})^\complement \subset \wt \cA \cap \wt \cB \cap \left\{ \max_{w \in \cW} X_w \le \xi^{1/t}(1-2\wt \chi)^{1/t} \corA{d}^2\right\} =:\wt \cD. 
\eeq
Thus, it suffices to obtain an upper bound on the probability of the \abbr{RHS} of \eqref{eq:cD}. Turning to prove such a bound we split the allowable range of $X_w$ for $w \in \cW$ into small subintervals and find probability bounds for each possible choice of collection of indices $w \in \cW$ that belong to any given subinterval. We then perform a union bound. 

To carry out this approach, we let 
\[
h: = \f{\xi^{1/t}( (1-\wt \chi)^{1/t} - (1-2 \wt\chi)^{1/t})}{2{\bm w}} \le \f{\xi^{1/t}\wt \chi}{t {\bm w}} \qquad \text{ and } \qquad K:= \left\lceil (\xi^{1/t}(1-2\wt \chi)^{1/t} - \wt\varrho)/h \right\rceil. 
\]
Let $\{\cN_k\}_{k \in [K]}$ be a partition of $\cW$. We allow $\cN_k$'s to be empty sets. For $k \in [K]$, set $R_k:=[q_k\corA{d}^2, r_k\corA{d}^2)$, where $q_k:=\wt \varrho +(k-1) h$ and $r_k:= \wt \varrho+kh$. Define
\[
\wh \cD_{\ul \cN}:= \left\{X_w \in R_k, \forall \, w \in \cN_k \text{ and } k \in [K]\right\} \subset \wt \cD_{\ul \cN}:= \left\{X_w \ge q_k \corA{d}^2, \forall \, w \in \cN_k \text{ and } k \in [K]\right\}.
\]
Letting $\ul n:= (n_1, n_2, \ldots, n_K) \in \Z_\ge^K$ it follows that 
\[
\wt \cB \cap \left\{ \max_{w \in \cW} X_w \le \xi^{1/t}(1-2\wt \chi)^{1/t} \corA{d}^2\right\} \subset \bigcup_{\ul n} \bigcup_{\ul \cN: |\cN_k| = n_k, k \in [K]} \wh\cD_{\ul \cN}.
\]
We also observe that on the event $\wt \cA \cap \wh \cD_{\ul \cN}$
\[
\wt \xi \corA{d}^{2t} \le \sum_{w \in \cW} X_w^t = \sum_{k=1}^K \sum_{w \in \cN_k} X_w^t \le \sum_{k=1}^K n_k q_k^t \corA{d}^{2t} +  \wt \chi  \xi (1- \wt \chi)^{1- 1/t} \corA{d}^{2t},
\]
where $|\cN_k|=n_k$ and the last two inequalities are due to the definitions of $h$, $K$, and $q_k$'s. Therefore, we deduce that
\[
\wt \cD \subset \cup_{\ul n} \cup_{\ul \cN: |\cN_k| = n_k, k \in [K]} \wt \cD_{\ul \cN},
\]
where the sum over $\ul n$ is such that 
\beq\label{eq:ul-n}
\sum_{k=1}^K n_k q_k^t \ge  \xi (1-\wt \chi -o(1)).
\eeq 
By Lemma \ref{lem:ash}, for any $\ul n$ such that \eqref{eq:ul-n} holds, as $q_1 \gtrsim 1$, we find that
\begin{multline}\label{eq:P-cD}
-\log \P(\wt \cD_{\ul N}) \ge   \left[\sum_{k=1}^K n_k q_k \corA{d}^2 \log \corA{d} \right] \cdot (1-o(1)) \\
\ge \xi^{1/t} \cdot \left[(1- \wt \chi - o(1))^{1/t} + \f{(\wt \chi \corA{+} o(1))^{1/t}}{4}\right] \cdot (1-o(1))\corA{d}^2 \log \corA{d}
\ge \xi^{1/t}(1+\beta \wt \chi^{1/t}) \corA{d}^2 \log \corA{d},
\end{multline}
where the penultimate step is due to $q_1 \le q_2 \le \cdots \le q_K \le \xi^{1/t} (1- 2\wt \chi)^{1/t}$, \corA{the lower bound in \eqref{eq:ul-n}} and the fact that
\[
\min_{\ul x}\sum_{s=1}^m x_s^{1/t} \ge 1+ \f{x_\star^{1/t}}{4} \quad \text{ subject to } \sum_{s=1}^m x_s \ge 1 \text{ and } \max_{s=1}^m x_s \le 1 -x_\star,
\]
for any collection of nonnegative reals $\ul x :=\{x_i\}_{i \in [m]}$  and $x_\star \in (0,1/2)$. In the last step we used that $\wt \chi \le c_0$ for some sufficiently small absolute constant $c_0$. Since the number of choices of placing ${\bm w}$ objects into $K$ bins is trivially bounded by $K^{\bm w}$, upon performing a union bound, the bound \eqref{eq:deg-unlikely} is now immediate from \eqref{eq:P-cD}. This completes the proof. 
\end{proof}

Equipped with Lemmas \ref{lem:hom-bd-bipartite}-\ref{lem:deg-unlikely} we are now ready to complete the proof of Corollary \ref{cor:hom-main}. 
\begin{proof}[Proof of Corollary \ref{cor:hom-main}]
Fix $t \ge 3$ and $\chi >0$. We will show that for $\corA{d}$ satisfying \eqref{eq:p-hom-gen} and
\beq\label{eq:p-cor-hom-regime}
2^{-1/t} \log \corA{d} (1+\wh \chi) \le \f12 \log(1/p),
\eeq
for some absolute constant $\wh \chi >0$, the bound
\beq
\P({\rm UT}_t(\deltx) \cap {\rm UT}_\Delta((\deltx/2)^{1/t} (1-\chi)^{1/t})^\complement) \ll \P({\rm UT}_t(\deltx))
\eeq
holds. This will prove the corollary. Notice that \eqref{eq:p-cor-hom-regime} continues to hold even we shrink $\wh \chi$. 

Observe that for $\corA{d}$ satisfying \eqref{eq:p-hom-gen} and \eqref{eq:p-cor-hom-regime} we have that $\phi_t(\deltx) = (\deltx/2)^{1/t} \corA{d}^2 \log \corA{d}$ (recall \eqref{eq:phi-deltx}) and thus, for any $\zeta >0$, the lower bound
\beq\label{eq:log-lbd-new}
\log\P({\rm UT}_t(\deltx)) \ge -(1+ \zeta \corA{t^{-1}}) \left(\f{\deltx}{2}\right)^{1/t} \corA{d}^2 \log \corA{d}
\eeq
holds for all large $n$ (see the proof of the lower bound on $\P({\rm UT}_t(\deltx))$). 

Recall from the proof of Theorem \ref{thm:hom-main} that a core graph $\Graph$ with $e(\Graph) \le \bar C_\star \corA{d}^2$ contains a strong-core subgraph. 
Therefore, applying \eqref{eq:no-seed}, 
Lemmas \ref{lem:ps2c} and \ref{lem:core-too-many}, and the lower bound on $\P({\rm UT}_t(\deltx))$ we observe that it suffices to show that
\beq\label{eq:cor-hom0}
\P({\rm UT}_t(\deltx) \cap {\rm UT}_\Delta(\deltx^{1/t} (1-\chi)^{1/t})^\complement \cap \{\exists \Graph \subset \G(n,p): \Graph \text{ is strong-core}\}) \ll \P({\rm UT}_t(\deltx)).
\eeq
Let $\vep = c^t$, for some suitably chosen absolute constant $c>0$. Upon choosing $\zeta$ sufficiently small, depending only on $c$, by \eqref{eq:pr-cI-0}-\eqref{eq:pr-cI-1} \corA{and lower bounds \eqref{eq:p-cor-hom-regime} and \eqref{eq:log-lbd-new}} we already have that 
\beq\label{eq:cor-hom1}
\P\left(\cup_{i, {\bm e}, {\bm e}_\sharp} \cup_{j=1}^{s_0} \cJ_{i,j, {\bm e}, {\bm e}_\sharp}\right) \ll \P({\rm UT}_t(\deltx)),
\eeq
for all large $n$. Application of Lemma \ref{lem:hom-unlikely}(a)-(b), with $\vep$ as above (we may need to shrink $c$ \corA{and $\zeta$} depending only on $\chi$), $\upalpha_1=t^{-1} e^{-1} \chi/4$, $\chi_\star = t^{-1}$, and $\wt \chi = 2^{1-1/t}(1+\wh \chi) -1$ (\corA{the implications of these precise choices of these parameters will be clear from below}), \corA{and lower bounds \eqref{eq:p-cor-hom-regime} and \eqref{eq:log-lbd-new}} further yield that
\beq\label{eq:cor-hom2}
\P(\cK) \ll \P({\rm UT}_t(\deltx)),
\eeq
for all large $n$, where 
\[
\cK:= \cup_{i, {\bm e}, {\bm e}_\sharp, {\bm v}_\sharp} \wh{\cJ}_{i, {\bm e}, {\bm e}_\sharp, {\bm v}_\sharp} \bigcup \cup_{i, {\bm e}} \cup_{{\bm e}_\sharp \ge {\bm e}_\star} \cJ_{i, s_0+1, {\bm e}, {\bm e}_\sharp}.
\]
Further denote 
\[
\wt{\cK}:= \left\{\exists \Graph \subset \G(n,p): \Graph \text{ is strong-core}\right\} \setminus \left(\cK \bigcup \cup_{i, {\bm e}, {\bm e}_\sharp} \cup_{j=1}^{s_0} \cJ_{i,j, {\bm e}, {\bm e}_\sharp}\right).
\]
\corA{Now our goal is to show that on $\wt \cK$ there must exist a set $\cW$ of size $O(1)$ such that $\cA_\cW(\xi) \cap \cB_\cW(\varrho)$ holds for appropriate choices of $\xi$ and $\varrho$, which in turn allows us to apply Lemma \ref{lem:deg-unlikely}. To this end, we} recall notation from the proof of Lemma \ref{lem:strong-core} and observe that any $\Graph$ satisfying the hypothesis of the event $\wt\cK$ must have that $e({\sf H}_i)^t \le  (e/2) \deltx \corA{d}^{2t}$ and ${\rm Hom}(C_{2t}, {\sf H}_i) \ge (1-11 \vep) \deltx \corA{d}^{2t}$ for some $i \in [\wt C]$. Since ${\sf H}_i$ is a bipartite graph, applying Lemma \ref{lem:hom-bd-bipartite}, with $\upalpha =\upalpha_1$ and $\upalpha_1$ as above, we find that 
\beq\label{eq:cA}
\sum_{u \in \cW_i} \deg_{\G(n,p)}(u)^t \ge (1-11\vep - \chi/4) \f{\deltx}{2} \corA{d}^{2t} \ge (1- \chi/2) \f{\deltx}{2} \corA{d}^{2t},
\eeq
where $\wh \cW_i=\wh\cW_i(\Graph):=\cup_{j=1}^i V_j$ and $ \cW_i=\cW_i(\Graph):=V({\sf H}_i) \setminus \wh \cW_i$ are the vertex bipartition of ${\sf H}_i$. On the other hand, by the definition of $\wt C$ (see \eqref{eq:wt-C}) and Lemma \ref{lem:prod-deg-sc}(a) we further have that
\beq\label{eq:cB}
\min_{u \in \cW_i} \deg_{\G(n,p)}(u) \ge \varrho \corA{d}^2 \quad \text{ and } \quad |\cW_i| \le (e/2)^{1/t} \deltx^{1/t} \varrho^{-1}=:{\bm w}=O(1), 
\eeq
for some $\varrho = \varrho(\vep, t) >0$. Therefore, setting $\xi= (1- \chi/2) \f{\deltx}{2}$ by \eqref{eq:cA}-\eqref{eq:cB} we deduce that 
\[
\wt \cK \subset \bigcup_{{\bm w}' \le {\bm w}} \bigcup_{\cW \subset\binom{[n]}{{\bm w}'}} \cA_{\cW}(\xi)  \cap \cB_{\cW}(\varrho).
\]
Hence, applying Lemma \ref{lem:deg-unlikely}, with $\wt \chi= \chi/4$, and a union bound we now derive that
{\allowdisplaybreaks
\begin{multline}\label{eq:cor-hom3}
\P(\wt \cK \cap {\rm UT}_\Delta((\deltx/2)^{1/t} (1-\chi)^{1/t})^\complement) \lesssim n^{O(1)} \exp\left( - \left(\f{\deltx}{2}\right)^{1/t}(1-\chi/2)^{1/t}(1+\beta \chi^{1/t}/2)\corA{d}^2 \log \corA{d} \right)\\
\ll \P({\rm UT}_t(\deltx)),
\end{multline}
}where the last step follows for $\chi$ sufficiently small. 
From \eqref{eq:cor-hom1}, \eqref{eq:cor-hom2}, and \eqref{eq:cor-hom3} we obtain \eqref{eq:cor-hom0}. This completes the proof. 
\end{proof}

\begin{rmk}\label{rmk:hom-to-eig}
{Similar to} the proof of Theorem \ref{thm:eig-main} \corA{for $\al >0$} one may hope to use the proof of Corollary \ref{cor:hom-main} (\corA{for large $t$}) to derive that Corollary \ref{cor:typ-strc} continues to hold for $\corA{d}$ such that 
\beq\label{eq:p-rmk}
\log \corA{d} \gtrsim \log n \qquad \text{ and } \qquad \corA{d} \ll n^{1/3}. 
\eeq
%
%
%
\corA{One can investigate \eqref{eq:cor-hom2} to deduce that it can be improved to $\P(\cK) \le \exp(-(1+\gamma t^{-1}) \phi_t(\deltx(t)))$ for some small absolute constant $\gamma >0$ and all large $t$, where $\deltx(t):= (1+\delta)^{2t}-1$. However, for such a $\gamma$ and all large $t$ one can check that
\[
2^{-1/t}(1 +\gamma t^{-1}) \deltx^{\f{1}{t}} < (1+\delta)^2.
\]
Hence one cannot conclude that $\P(\cK) \ll \P({\rm UT}_\lambda(\delta))$, for large $t$. Therefore one cannot use the proof of  Corolary \ref{cor:hom-main} to derive an analogue of Corollary \ref{cor:typ-strc} for $d$ satisfying \eqref{eq:p-rmk}. New ideas are needed.} 
%
\end{rmk}

\section{Proof of Theorem \ref{cor:mf-lbd}}\label{sec:mf-lbd}
The lower bounds in \eqref{cor:mf-lbd-eq1} and \eqref{cor:mf-lbd-eq2}, as will be seen below, follow from Theorems \ref{thm:eig-main} and \ref{thm:hom-main}. To prove the upper bound we will need the following bounds on the variances.

\begin{lem}\label{lem:mf-var}
Let $\delta, \deltx,$ and $t$ be as in Theorem \ref{cor:mf-lbd}. For any $\chi \in (-1,1)$ define 
\beq\label{eq:sfS}
{\sf S}_{\deltx, t, \chi}:= \left\{{\bm \xi} \in [0,1]^N: \E_{\mu_{\bm \xi}}[{\rm Hom}(C_{2t}, \G_n)] \ge (1+\deltx (1+\chi)) \corA{d}^{2t}\right\}.
\eeq
\begin{enumerate}

\item[(a)] For any ${\bm \xi} \in [0,1]^N$
\[
\Var_{\mu_{\bm \xi}}(\lambda(\G_n)) \lesssim 1.
\]

\item[(b)] Let $p \in (0,1)$ be such that $np \gg \sqrt{\log n}$. Then, for any fixed $\chi \in (-1,1)$,
\beq\label{eq:var-ef}
\sup_{\mu_{\bm \xi} \in {\sf S}_{\deltx, t, \chi}} \f{\Var_{\mu_{\bm \xi}}({\rm Hom}(C_{2t}, \G_n))}{(\E_{\mu_{\bm \xi}}{\rm Hom}(C_{2t}, \G_n))^2} = o(1).
\eeq

\end{enumerate}
\end{lem}

\begin{proof}
The proofs of both parts will be consequences of Efron-Stein inequality (cf.~\cite[Theorem 3.1]{BLM}). In fact, part (a) has been already worked out in \cite[Example 3.14]{BLM}.  So we will only prove part (b). 

Let $A_n=\{a_{i,j}\}_{i,j \in [n]}$ be the random matrix which is the adjacency matrix of $\G_n$, where $\G_n \stackrel{d}{=}\G(n, {\bm \xi})$. For $i < j \in [n]$ we let $A_n^{(i,j)}$ to be the symmetric random matrix obtained from $A_n $ by replacing $a_{i,j}$ with $\wh a_{i,j}$, an independent copy of $a_{i,j}$. 
By \cite[Theorem 3.1]{BLM} we have that
\beq\label{eq:var-ef1}
\Var_{\mu_{\bm \xi}}({\rm Hom}(C_{2t}, \G_n)) = \Var_{\mu_{\bm \xi}}[\Tr(A_n^{2t})] \le \f12 \sum_{i< j \in [n]} \E \left[\left(\Tr\left(A_n^{2t} -( A_n^{(i,j)})^{2t}\right)\right)^2\right]. 
\eeq
Observe that only the $(i,j)$-th and the $(j,i)$-th entries of $A_n - A_n^{(i,j)}$ are non-zero. Therefore, noting that  
\[
\Tr(B_1^{2t}) - \Tr(B_2^{2t}) = \sum_{k=1}^{2t} \Tr(B_2^{k-1}B_1^{2t-k}(B_1- B_2)), 
\]
for any two square matrices $B_1$ and $B_2$, we derive that 
\beq\label{eq:var-ef2}
\Tr\left(A_n^{2t} -( A_n^{(i,j)})^{2t}\right) = \sum_{k=1}^{2t} \left[((A_n^{(i,j)})^{k-1} A_n^{2t-k})_{i,j} (a_{i,j} - \wh a_{i,j}) + ((A_n^{(i,j)})^{k-1} A_n^{2t-k})_{j,i}  (a_{i,j} - \wh a_{i,j})\right]. \notag
\eeq
Hence, by Cauchy-Schwarz inequality and the fact that $| a_{i,j} - \wh a_{i,j}| \le 1$ we obtain that
\beq\label{eq:var-ef3}
\sum_{i< j \in [n]}  \left(\Tr\left(A_n^{2t} -( A_n^{(i,j)})^{2t}\right)\right)^2 \le 4t \sum_{k=1}^{2t} \sum_{i \ne j} \left(((A_n^{(i,j)})^{k-1} A_n^{2t-k})_{i,j}\right)^2.
\eeq
So, in the light of \eqref{eq:var-ef1} and \eqref{eq:var-ef3} we deduce that it suffices to show that 
\beq\label{eq:var-ef4}
\gI_k:=\E \left[  \sum_{i \ne j} \left(((A_n^{(i,j)})^{k-1} A_n^{2t-k})_{i,j}\right)^2 \right] \ll \left( \E \Tr (A_n^{2t})\right)^2,
\eeq
for all $k \in [2t]$. Since $A_n \stackrel{d}{=}A_n^{(i,j)}$ it is enough to prove \eqref{eq:var-ef4} for $k \in [t]$. 

We will prove \eqref{eq:var-ef4} by an induction on $k$. We start with $k=1$. Notice that
\[
\E \left[  \sum_{i \ne j} \left( (A_n^{2t-1})_{i,j}\right)^2 \right] \le \E \left[\|A_n^{2t-1}\|_{\rm HS}^2\right] \le \E \left[\| A_n\|^{2(t-1)} \|A_n^t\|_{\rm HS}^2\right],
\]
where we have used that $\|B_1 B_2 \|_{\rm HS} \le \|B_1\| \cdot \|B_2\|_{\rm HS}$. Denote
\[
\mathscr{A}:= \left\{ \max\left\{\max_{i < j \in [n]}\{\|A_n^{(i,j)}\|\},  \|A_n\|\right\} \le 2 (\E [ \Tr(A_n^{2t})])^{1/(2t)} \right\}. 
\]
Since for a $n \times n$ symmetric matrix $B$ we have that $\|B\|= \max\{\lambda_1(B), - \lambda_n(B)\}$ we get from \cite[Example 8.7]{BLM} that
\beq\label{eq:var-ef5}
\P(\|A_n\| \ge 2 \theta_t) \le \P(\|A_n\| \ge \E\|A_n\|+ \theta_t) \le \exp( - 2c\theta_t^2),
\eeq
for some absolute constant $c>0$, where $\theta_t:= (\E [ \Tr(A_n^{2t})])^{1/(2t)} \ge \E\|A_n\|$. 
Using that $A_n \stackrel{d}{=} A_n^{(i,j)}$, $\theta_t \gtrsim \corA{d} \gg \sqrt{\log n}$ for $\mu_{\bm \xi} \in {\sf S}_{\deltx, t, \chi}$, and a union bound we obtain from \eqref{eq:var-ef5} that
\beq\label{eq:var-ef6}
\P(\mathscr{A}^\complement) \le \exp( - c\theta_t^2).
\eeq

Since $A_n$ is a matrix with entries bounded by one it follows that $\|A_n\| \le \|A_n\|_{\rm HS} \le n$. Therefore, as $\theta_t \gg \sqrt{\log n}$ an application of \eqref{eq:var-ef6} yields that
\beq\label{eq:var-ef7}
\E \left[\| A_n\|^{2(t-1)} \|A_n^t\|_{\rm HS}^2 {\bf 1}_{\mathscr{A}^\complement}\right] = o(1). 
\eeq
On the other hand,   
\beq\label{eq:var-ef8}
\E \left[\| A_n\|^{2(t-1)} \|A_n^t\|_{\rm HS}^2 {\bf 1}_{\mathscr{A}}\right] \lesssim \theta_t^{2(t-1)} \E[\Tr(A_n^{2t})] \ll \theta_t^{4t}. 
\eeq
Combining \eqref{eq:var-ef7}-\eqref{eq:var-ef8} we have \eqref{eq:var-ef4} for $k=1$. Now let us assume that \eqref{eq:var-ef4} holds for some $k= k_0 <t$. We proceed to prove that the same holds for $k=k_0+1$. To this end, observe that
\begin{multline*}
\left[(A_n^{(i,j)})^{k_0} A_n^{2t-k_0-1}\right]_{i,j} 
= \left[(A_n^{(i,j)})^{k_0-1} A_n^{2t-k_0}\right]_{i,j} +  \left[(A_n^{(i,j)})^{k_0-1}\right]_{i,j} (\wh a_{i,j} - a_{i,j}) \left[A_n^{2t-k_0-1}\right]_{i,j} \\
+ \left[(A_n^{(i,j)})^{k_0-1}\right]_{i,i} (\wh a_{i,j} - a_{i,j}) \left[A_n^{2t-k_0-1}\right]_{j,j}.
\end{multline*}
Hence
\beq\label{eq:var-ef9}
\gI_{k_0+1} \lesssim \gI_{k_0} + \sum_{\ell=1}^4 \gJ^{(\ell)},
\eeq
where
\[
\gJ^{(1)}:= \sum_{i \ne j} \left[(A_n^{(i,j)})^{k_0-1}\right]_{i,j}^2 \wh a_{i,j}  \left[A_n^{2t-k_0-1}\right]_{i,j}^2, \,\gJ^{(2)}:= \sum_{i \ne j} \left[(A_n^{(i,j)})^{k_0-1}\right]_{i,j}^2 a_{i,j} \left[A_n^{2t-k_0-1}\right]_{i,j}^2,
\]
\[
\gJ^{(3)}:= \sum_{i \ne j} \left[(A_n^{(i,j)})^{k_0-1}\right]_{i,i}^2 \wh a_{i,j}  \left[A_n^{2t-k_0-1}\right]_{j,j}^2, \, \text{ and } \, \gJ^{(4)}:= \sum_{i \ne j} \left[(A_n^{(i,j)})^{k_0-1}\right]_{i,i}^2 a_{i,j} \left[A_n^{2t-k_0-1}\right]_{j,j}^2.
\]
Now note that
\begin{multline}\label{eq:var-ef10}
\E[\gJ^{(1)}] \le \E\left[ \|A_n^{(i,j)}\|^{2(k_0-1)} \|A_n^{2t-k_0-1}\|_{\rm HS}^2\right] \le \E\left[ \|A_n^{(i,j)}\|^{2(k_0-1)} \|A_n\|^{2(t-k_0-1)} \|A_n^{t}\|_{\rm HS}^2\right]\\
 \le o(1) + \E\left[ \|A_n^{(i,j)}\|^{2(k_0-1)} \|A_n\|^{2(t-k_0-1)} \|A_n^{t}\|_{\rm HS}^2{\bf 1}_{\mathscr{A}}\right] \lesssim o(1)+ \theta_t^{2t-4} \E[\Tr(A_n^{2t})] \ll \theta_t^{4t},
\end{multline}
where in the third step we use \eqref{eq:var-ef6} and argue similarly as in \eqref{eq:var-ef7}. Notice that the same argument can be repeated to show that the same bound holds for $\E[\gJ^{(2)}]$. Next, as $\|A_n\|, \|A_n^{(i,j)}\| \le n$, using \eqref{eq:var-ef6} and that $\{\wh a_{i,j}\}_{i < j \in [n]}$ and $A_n$ are independent we see that
\begin{multline}\label{eq:var-ef11}
\E[\gJ^{(3)}] \lesssim o(1)+ np \theta_t^{2(k_0-1)} \E\left[  \left(\sum_{j=1}^n \left[A_n^{2t-k_0-1}\right]_{j,j}^2 \right) {\bf 1}_{\mathscr{A}} \right]\\
 \lesssim o(1)+  \theta_t^{2k_0-1} \E\left[  \|A_n^{2t-k_0-1}\|_{\rm HS}^2 {\bf 1}_{\mathscr{A}} \right] \ll \theta_t^{4t}.
\end{multline}
To bound $\E[\gJ^{(4)}]$ we use that $\sum_{i} a_{i,j} \le \|A_n\|^2$, for any $j \in [n]$, and proceed similarly as above to find that
\begin{equation}\label{eq:var-ef12}
\E[\gJ^{(4)}] \lesssim o(1)+  \theta_t^{2k_0} \E\left[  \left(\sum_{j=1}^n \left[A_n^{2t-k_0-1}\right]_{j,j}^2 \right) {\bf 1}_{\mathscr{A}} \right] \ll \theta_t^{4t}.
\end{equation}
Thus, from \eqref{eq:var-ef9}-\eqref{eq:var-ef12} and the induction hypothesis we conclude that \eqref{eq:var-ef4} holds for $k=k_0+1$. This completes the proof. 
%
%
%
\end{proof}

We also need to show that the solution of the variational problem \eqref{eq:mf-vp3} is not too small for $f(\cdot)={\rm Hom}(C_{2t}, \cdot)$ and $\lambda(\cdot)$.

\begin{lem}\label{lem:psi-lbd}
Fix $\delta, \deltx >0$, and $t \ge 3$.
\begin{enumerate}
\item[(a)] Under the same setup as in Theorem \ref{cor:mf-lbd}(a), for $f(\cdot)=\lambda(\cdot)$ we have that 
\beq\label{eq:psi-lbd}
\Psi_p(f(\cdot), \delta) \gtrsim p^2 \log(1/p),
\eeq
where $\Psi_p(\cdot, \cdot)$ is as in \eqref{eq:mf-vp3}. 
\item[(b)] Under the setup of Theorem \ref{cor:mf-lbd}(b) the lower bound \eqref{eq:psi-lbd}, holds for $\Psi_p(f(\cdot), \deltx)$ with $f(\cdot)= {\rm Hom}(C_{2t}, \cdot)$. 
\end{enumerate}
\end{lem} 

\begin{proof}
We will only prove part (b). The proof of (a) is similar. 

By a standard coupling argument it follows that
{\allowdisplaybreaks
\beq\label{eq:coup}
\E_{\mu_{\bm \xi}} ({\rm Hom}(C_{2t}, \G_n)) \le \E_{\mu_{\wh{\bm \xi}}} ({\rm Hom}(C_{2t}, \G_n)), \quad \text{ for } {\bm \xi} \preccurlyeq \wh{\bm \xi},
\eeq
}where the notation ${\bm \xi} \preccurlyeq \wh{\bm \xi}$ implies $\xi_\upalpha \le \wh \xi_\upalpha$ for all $\upalpha \in [N]$. 
On the other hand, by \eqref{eq:hom-exp} 
{\allowdisplaybreaks 
\beq\label{eq:coup1}
\Psi_p(f(\cdot), \delta) = \inf\left\{ I_p({\bm \xi}):  {\bm \xi} \in [0,1]^N, \E_{\mu_{\bm \xi}}[f] \ge (1+\delta)(1- o(1)) \corA{d}^{2t}\right\},
\eeq
}where $f(\cdot)= {\rm Hom}(C_{2t}, \cdot)$.
Note that if ${\bm \xi} \in [0,1]^N$ is such that $\max_{\upalpha} \xi_\upalpha \le (1+\delta/2)^{1/(2t)} p$ then ${\bm \xi}  \preccurlyeq (1+\delta/2)^{1/(2t)} p {\bm 1}$, where ${\bm 1}$ is the vector of all ones. Hence, by \eqref{eq:hom-exp} and \eqref{eq:coup} we obtain that $\E_{\mu_{\bm \xi}}[f] \le (1+2\delta/3) \corA{d}^{2t}$. Thus, such a ${\bm \xi}$ cannot belong to the set on the \abbr{RHS} of \eqref{eq:coup1}. So, for any ${\bm \xi}$ belonging to this set there must exist some index $\upalpha \in [N]$ such that $\xi_\upalpha \ge  (1+\delta/2)^{1/(2t)} p$. 
Since $I_p({\bm \xi}) \ge I_p(\xi_\upalpha) \gtrsim p^2 \log(1/p)$ (cf.~\cite[Corollary 3.5]{LZ}) the proof is now complete. 
\end{proof}

We are now ready to prove Theorem \ref{cor:mf-lbd}. 
\begin{proof}[Proof of Theorem \ref{cor:mf-lbd}]
We only prove (b). The proof of (a) being simpler and similar in nature is omitted. 

{First let us prove the lower bound in \eqref{cor:mf-lbd-eq2}.} We fix a bijection $\vartheta: \binom{[n]}{2}\mapsto [N]$, where $N= \binom{n}{2}$. Let $\wh{\bm \xi}, \wt{\bm \xi} \in [0,1]^N$ be such that 
\[
\wh \xi_\upalpha = p + (\deltx(1-\wt\chi)/2)^{1/t} np^2 {\bf 1}_{\{\vartheta^{-1}(\upalpha) \in \{(1,j), j \in [n]\}\}} \, \text{ and } \, \wt\xi_\upalpha = p +(1-p) {\bf 1}_{\{\vartheta^{-1}(\upalpha) \in \{ [\lceil (\deltx(1-\wt\chi))^{1/(2t)} \corA{d}\rceil ]^2\}\}},  
\]
for $\upalpha \in [N]$, where $\wt \chi=\chi/2$. It is easy to note that
\[
\min\left\{\E_{\mu_{\wh{\bm \xi}}}[{\rm Hom}(C_{2t}, \G_n)], \E_{\mu_{\wt{\bm \xi}}}[{\rm Hom}(C_{2t}, \G_n)] \right\} \ge (1+\deltx(1-\chi)) \corA{d}^{2t},
\]
for all large $n$. Therefore, $\mu_{\wh{\bm \xi}}, \mu_{\wt{\bm \xi}} \in {\sf S}_{\deltx, t, -\chi}$ (see \eqref{eq:sfS}), and hence from \eqref{eq:UT-t-bd} it follows that
\[
-\log \P({\rm UT}_t(\deltx)) \ge \min\{I_p(\wh {\bm \xi}), I_p(\wt{\bm \xi})\} \ge \Psi_p({\rm Hom}(C_{2t}, \cdot), \deltx(1- \chi)),
\]
for all large $n$, yielding the lower bound in \eqref{cor:mf-lbd-eq2}. 

We now turn to the proof of the upper bound. This part of the proof is inspired from the proofs of the lower bounds in \cite[Theorem 1.1]{chd} and \cite[Theorem 5]{eld}. 
Fix $\deltx , \chi>0$ and $t \ge 3$. 
Let $f(\G_n):= {\rm Hom}(C_{2t}, \G_n)$ and use the shorthand $\Psi:= \Psi_p(f, \deltx(1+\chi))$. Let ${\bm \xi}_\star \in [0,1]^N$ (with $N= \binom{n}{2}$) such that $\mu_{{\bm \xi}_\star} \in {\sf S}_{\deltx, t, \chi}$ and $I_p({\bm \xi}_\star) \le (1+\chi/2) \Psi$. Fix some large absolute constant $C >1$. Define  
\[
\wh f(\G_n):= \left\{
\begin{array}{ll}
\left({\f{f(\G_n)}{(1+\deltx(1+\chi))\E_{\mu_p}[f]}-1}\right)\Big/\left({\f{\chi \deltx}{1+\deltx(1+\chi)}}\right) & \mbox{ if } \E_{\mu_{{\bm \xi}_\star}}[f] \le C(1+\deltx) \E_{\mu_p}[f],\\
\left(\f{f(\G_n)}{\E_{\mu_{{\bm \xi}_\star}}[f]} - 1\right) \Big/ \left(1 - C^{-1}\right) & \mbox{ otherwise}.
\end{array}
\right. 
\]
Let 
\[
h(x):= \left\{
\begin{array}{ll}
2x+1 & \mbox{ for }x \le -1,\\
-x^2 & \mbox{ for } x \in [-1,0],\\
0 & \mbox{ for } x \ge 0.
\end{array}
\right.
\]
and  
\[
\wh h(\G_n):= 2\Psi h(\wh f(\G_n)).
\]
\corA{Observe that $h'(x)= - 2 (x \vee (-1)) {\bf 1}(x \le 0)$}. Therefore, by Taylor's theorem we find that 
\[
\E_{\mu_{{\bm \xi}_\star}}[\wh h(\G_n)] = 2\Psi \E_{\mu_{{\bm \xi}_\star}}[h(\wh f(\G_n))] \ge 2 \Psi h(\E_{\mu_{{\bm \xi}_\star}}[\wh f(\G_n)]) - 2 \Psi \Var_{\mu_{{\bm \xi}_\star}}(\wh f(\G_n)) = - 2 \Psi \Var_{\mu_{{\bm \xi}_\star}}(\wh f(\G_n)),
\]
where in the last step we have used that $\E_{\mu_{{\bm \xi}_\star}}[\wh f(\G_n)] \ge 0$. 

On the other hand, by the definition of $\wh h(\cdot)$ we find that $\wh h \le -2 \Psi$ on the event ${\rm UT}_t(\deltx)^\complement$. \corA{As $\wh h \le 0$} we deduce that
{\allowdisplaybreaks
\begin{multline}\label{eq:ut-lbd}
\P({\rm UT}_t(\deltx)) \ge \E_{\mu_p} [\exp(\wh h (\G_n))] - \E_{\mu_p}[\exp(\wh h (\G_n)) {\bf 1}_{{\rm UT}_t(\deltx)^\complement}]\\
\ge \exp\left( \E_{\mu_{{\bm \xi}_\star}}[\wh h(\G_n)]  - I_p({\bm \xi}_\star)\right)- \exp(-2 \Psi) \\
\ge \exp(- (1+\chi/2) \Psi - 2\Psi \Var_{\mu_{{\bm \xi}_\star}}(\wh f(\G_n)))- \exp(-2 \Psi) \ge \exp(- (1+3\chi/4) \Psi)- \exp(-2 \Psi),
\end{multline}
} for all large $n$, where in the last step we use Lemma \ref{lem:mf-var}(b). 

By Theorem \ref{thm:hom-main} we have that $\P({\rm UT}_t(\deltx)) \le e^{-\omega(\log n)}$. Therefore \eqref{eq:ut-lbd} shows that either $\Psi \ge \omega(\log n)$ or $1 - \exp( -(1-3\chi/4) \Psi) \le \exp(-\omega(\log n))$. If the latter condition holds then we have that 
\[
\Psi \lesssim - \log (1 - \exp(-\omega(\log n))) \asymp \exp(-\omega(\log n)).
\]
Since $p \gtrsim n^{-1}$ this yields a contradiction to Lemma \ref{lem:psi-lbd}(b). Thus, we must have that $\Psi \ge \omega(\log n) \gg 1$. Plugging this lower bound on $\Psi$ in \eqref{eq:ut-lbd} we finally deduce that 
\[
\P({\rm UT}_t(\deltx)) \ge \exp(- (1+\chi) \Psi), 
\]
for all large $n$, yielding the upper bound in \eqref{cor:mf-lbd-eq2}. This completes the proof.  
\end{proof}



\appendix

\section{Proof of Lemma \ref{lem:graph-eig}}\label{sec:graph-eig-pf}

The proof of $\lambda(\Graph) \le \sqrt{2 {\bm e}}$ is trivial. Indeed, denoting $\ga_{i,j}$ to be the $(i,j)$-th entry of the adjacency matrix of $\Graph$ we note that
\beq
\lambda(\Graph) \le \sqrt{\sum_{i} \lambda_i^2(\Graph)} = \sqrt{\sum_{i,j} \ga_{i,j}} = \sqrt{2 {\bm e}}. \notag
\eeq
The rest of (i), and (ii)-(v) are taken from \cite[Proposition 3.1]{KS}. 
The proof of (vi) is standard. It relies on the variational representation of eigenvalues and 
the fact that the Perron-Frobenius eigenvector of a matrix with nonnegative entires has nonnegative entries. We omit further details.
 
We are not able to locate the proof of (vii) in the literature. It follows from an intermediate step in the proof of (viii). A bound similar to (vii) was derived in \cite{Ho} for $\updelta =1$. We adapt their proof to deduce our bound for $\updelta \ge 2$.  
Turning to do that, for ease in writing let us assume that $\Graph$ is a graph on $[N]$. Let $\rho$ be an eigenvalue of its adjacency matrix $\mathfrak{A}$ and ${\bm x}=(x_1, x_2, \ldots, x_N)$ be the corresponding eigenvector of unit Euclidean norm. Then for any $i \in [N]$ we have that 
\beq\label{eq:b1}
\rho^2 x_i^2 = |(\mathfrak{A} x)_i|^2 = \left(\sum_{j: j \sim i} x_j\right)^2 \le \deg_\Graph(i) \cdot  \left(\sum_{j: j \sim i} x_j^2\right). 
\eeq
Now summing the both sides of \eqref{eq:b1} over $i \in [N]$ we deduce that 
\beq\label{eq:b2}
\rho^2 \le \sum_{j=1}^N x_j^2 \left(\sum_{i: i \sim j} \deg_\Graph(i)\right) = 2 {\bm e} - \sum_{j=1}^N x_j^2 \left(\sum_{i: i \nsim j} \deg_\Graph(i)\right).
\eeq
Observe that, as $\updelta \ge 2$, 
{\allowdisplaybreaks
\begin{multline*}
\sum_{j=1}^N x_j^2 \left(\sum_{i: i \nsim j} \deg_\Graph(i)\right) = \sum_{i=1}^N \deg_\Graph(i) x_i^2  + \sum_{i=1}^N \deg_\Graph(i)  \left(\sum_{\substack{j: j \nsim i\\ j \neq i}} x_j^2\right)\\
\ge \sum_{i=1}^N \deg_\Graph(i) x_i^2  + 2 \sum_{i=1}^N  \left(\sum_{\substack{j: j \nsim i\\ j \neq i}} x_j^2\right)
 = \sum_{i=1}^N \deg_\Graph(i) x_i^2 + 2 \sum_{i=1}^N (N- \deg_\Graph(i) -1) x_i^2 \\
 = 2(N-1) - \sum_{i=1}^N \deg_\Graph(i) x_i^2  \ge 2(N-1) - \Delta. 
\end{multline*}}
Plugging this bound in \eqref{eq:b2} the desired bound on $\lambda(\Graph)$ follows. 

It remains to prove (vii). Our starting point is \eqref{eq:b1}. Recall that $\Graph$ is a bipartite graph with vertex partition $V_1$ and $V_2$. Summing the both sides of \eqref{eq:b1} over $i \in V_1$ we obtain
\beq\label{eq:b4}
\rho^2 \sum_{i \in V_1} x_i ^2 \le \sum_{j \in V_2} x_j^2 \left(\sum_{i: i \sim j} \deg_\Graph(i)\right) \le \left[\max_{j \in V_2} \left(\sum_{i: i \sim j} \deg_\Graph(i)\right)\right] \cdot \sum_{j \in V_2} x_j^2. 
\eeq
Since $\Graph$ is a bipartite graph, upon permuting the vertices if required, we have that its adjacency matrix $\gA = \begin{bmatrix} 0 & \gA_0 \\ \gA^{\sf T}_0 & 0 \end{bmatrix}$ for some matrix $\gA_0$ of dimension $|V_1| \times |V_2|$. It is straightforward to see that $\rho$ is a singular value of $\gA_0$ if and only if $\pm \rho$ are eigenvalues of $\gA$. Furthermore, if ${\bm y}$ and ${\bm z}$ are the left and right singular vectors of $\gA_0$ corresponding to the singular value $\rho$, then $\begin{pmatrix} {\bm y} \\ \pm {\bm z}\end{pmatrix}$ are the eigenvectors of $\gA$ corresponding to the eigenvalues $\pm \rho$. This, immediately implies that $\sum_{i \in V_1} x_i^2 = \sum_{j \in V_2} x_j^2$. Thus, \eqref{eq:b4} yields that 
\[
\rho \le \max_{j \in V_2} \sqrt{\sum_{i: i \sim j} \deg_\Graph(i)}. 
\] 
Now reversing the roles of $V_1$ and $V_2$ the desired bound on $\lambda(\Graph)$ follows. This completes the proof. 
\qed

\section{Local Homomorphism count bounds and {proofs of Lemmas \ref{lem:prod-deg} and \ref{lem:prod-deg-sc}}}\label{app:loc-hom}
We start with a couple of local homomorphism bounds. 
\begin{lem}\label{lem:deg-prod}
Let $H$ be a $\Delta$-regular graph. For every graph $\Graph$ and an edge $e = \{u,v\} \in E(\Graph)$ we have
\[
\left|{\rm Hom}(H, \Graph, e)\right| \le 4 e_H \cdot (2 e_\Graph)^{\frac{v_H}{2} - \f{2\Delta-1}{\Delta}} \cdot (4 \deg_\Graph(u) \cdot \deg_\Graph(v))^{\f{\Delta-1}{\Delta}}. 
\]
\end{lem}

\begin{lem}\label{lem:excess-edge}
Let $H$ be a $\Delta$-regular graph. For every graph $\Graph$ and $\Graph' \subset \Graph$ we have
\[
\sum_{e \in E(\Graph')} \left|{\rm Hom}(H, \Graph, e)\right| \le e_H \cdot (2e_G)^{v_H/2} \cdot \left( \f{e_{\Graph'}}{e_\Graph}\right)^{1/\Delta}. 
\]
\end{lem}

Notice that analogs of Lemmas \ref{lem:deg-prod} and \ref{lem:excess-edge}  are proved in \cite[Section 5.3]{hms} for subgraph counts. To prove these two results we follow the same route. 
We need the following abstract result. 

\begin{lem}\label{lem:shearer}
Let $H$ be a $\Delta$-regular graph and $\cE$ be a collection of homomorphisms of $H$ into $\Graph$. For every edge $e=(a,b) \in E(H)$, we set 
\[
\cE_e:= \{\{\varphi(a), \varphi(b)\}, \varphi \in \cE\}.
\]
Then
\[
|\cE| \le \prod_{e \in E(H)} (2 |\cE_e|)^{1/\Delta}. 
\]
\end{lem}
Using Shearer's inequality the same bound was shown to hold in \cite{hms} (see Lemma 5.10 there) when $\cE$ is a collection of embeddings of $H$ into $\Graph$. The same argument works in our setup. So we omit the details.  

Equipped with Lemma \ref{lem:shearer} one can proceed as in the proofs of to \cite[Lemma 5.7]{hms} and \cite[Lemma 5.9]{hms} to obtain Lemmas \ref{lem:deg-prod} and \ref{lem:excess-edge}, respectively. We refrain from repeating it here. 

The proof of Lemma \ref{lem:prod-deg-sc} now follows upon using Lemmas \ref{lem:deg-prod} and \ref{lem:excess-edge} (instead of \cite[Lemmas 5.7 and 5.9]{hms}), and imitating the rest of the proof of  \cite[Lemma 4.2]{BB}. A similar comment applies for the proof of Lemma \ref{lem:prod-deg}. 
We spare the tedious details. 


\end{document}